\documentclass [twoside,reqno, 12pt] {amsart}

\usepackage{fullpage}

\usepackage{amsfonts,mathtools}
\usepackage{amssymb, tikz}
\usepackage{color}
\usepackage{hyperref}

\usepackage{cancel}
\usepackage{comment}
\usepackage{caption}
\allowdisplaybreaks

\theoremstyle{definition}

\usepackage{listings}
\usepackage{xcolor}

\definecolor{codegreen}{rgb}{0,0.6,0}
\definecolor{codegray}{rgb}{0.5,0.5,0.5}
\definecolor{codepurple}{rgb}{0.58,0,0.82}
\definecolor{backcolour}{rgb}{0.95,0.95,0.92}

\lstdefinelanguage{Magma}{
  morekeywords={PolynomialRing, ideal, Coefficients, GroebnerBasis, EliminationIdeal, MinimalBasis, print},
  sensitive=true,
  morecomment=[l]{//},
  morecomment=[s]{/*}{*/},
  morestring=[b]",
}

\lstdefinestyle{mystyle}{
    backgroundcolor=\color{backcolour},
    commentstyle=\color{codegreen},
    keywordstyle=\color{magenta},
    numberstyle=\tiny\color{codegray},
    stringstyle=\color{codepurple},
    basicstyle=\ttfamily\footnotesize,
    breakatwhitespace=false,
    breaklines=true,
    captionpos=b,
    keepspaces=true,
    numbers=left,
    numbersep=5pt,
    showspaces=false,
    showstringspaces=false,
    showtabs=false,
    tabsize=2
}

\newtheorem{definition}{Definition}[section]
\newtheorem{theorem}[definition]{Theorem}
\newtheorem{lemma}[definition]{Lemma}

\newtheorem{proposition}[definition]{Proposition}
\newtheorem{corollary}[definition]{Corollary}
\newtheorem{remark}[definition]{Remark}
\newtheorem{example}[definition]{Example}

\newcommand{\thmref}[1]{Theorem~\ref{#1}}

\newcommand{\stiffnesstensors}{F}
\newcommand{\positivetensors}{\stiffnesstensors^{\mathrm{pos.}}}
\newcommand{\analyticallyadmissibletensors}{\stiffnesstensors^{\mathrm{an.adm.}}}
\newcommand{\algebraicallyadmissibletensors}{\stiffnesstensors^{\mathrm{alg.adm.}}}

\newcommand{\pairsofmetrics}{E}
\newcommand{\admissiblepairsofmetrics}{\pairsofmetrics^{\mathrm{adm.}}}

\newcommand{\C}{\mathbb{C}}

\newcommand{\Id}{I}

\newcommand{\op}{\operatorname{op}}
\newcommand{\rank}{\operatorname{rank}}

\newcommand{\R}{\mathbb{R}}
\newcommand{\ps}{\mathbb{P}}

\newcommand{\tr}{\operatorname{tr}}

\newcommand{\abs}[1]{\left\lvert#1\right\rvert}

\newcommand{\spoly}[1]{P_{#1}}

\newcommand{\vanish}{\mathbb{V}}
\newcommand{\sphere}{\mathbb{S}}

\newcommand{\V}{\vanish}
\newcommand{\I}{\mathbb{I}}

\DeclareMathOperator{\Sym}{Sym}

\newcommand{\rotA}[1]{\operatorname{Rot}_0(#1)}
\newcommand{\rotB}[1]{\operatorname{Rot}_1(#1)}
\newcommand{\rotC}[1]{\operatorname{Rot}_2(#1)}
\newcommand{\reflA}[1]{\operatorname{Refl}_0(#1)}
\newcommand{\reflB}[1]{\operatorname{Refl}_1(#1)}
\newcommand{\reflC}[1]{\operatorname{Refl}_2(#1)}

\newcommand{\dummy}{{\cdot}}

\title{Travel time tomography for elliptically anisotropic elastic planar media}

\author[Ilmavirta]{Joonas Ilmavirta}
\address{J. Ilmavirta, Department of Mathematics and Statistics, University of Jyv\"askyl\"a, Jyv\"askyl\"a, 40014,  Finland} \email{joonas.ilmavirta@jyu.fi}

\author[Saksala]{Teemu Saksala}
\address{T. Saksala, Department of Mathematics\\
North Carolina State University, Raleigh\\ 
NC 27695, USA}
\email{tssaksal@ncsu.edu}

\author[Schl\"uter]{Hjørdis Schl\"uter}
\address{H. Schl\"uter, Department of Mathematics and Statistics, University of Helsinki, 00014, Finland} \email{hjordis.schluter@helsinki.fi }

\date{\today}

\begin{document}

\begin{abstract}
We study the gauge of the travel time tomography problem for elliptically anisotropic planar media. Under this assumption the respective slowness surfaces consist of an inner ellipse and an outer circle. Hence, the singularities of seismic waves propagate on geodesics given by two Riemannian metrics of which the one related to shear waves is conformally Euclidean. We show that if both of these metrics are lens rigid then, up to boundary fixing diffeomorphisms, each elliptically anisotropic stiffness tensor is uniquely determined by the lens data whenever it has at least one isotropic point. If none of those points exists, then there are exactly two elliptically anisotropic stiffness tensors with the same lens data. 
Clearly, not all boundary fixing diffeomorphisms preserve the imposed structure of the slowness surface. We show that the set of well-behaved diffeomorphisms is never a subgroup of all boundary fixing diffeomorphisms.
Our proofs are based on an interplay of algebra and geometry, and the key of the proofs is a characterization of the elliptic anisotropy via several polynomial conditions. 
\end{abstract}

\maketitle

\section{Introduction}
\label{sec:intro}

We show that a certain type of elliptic anisotropic structure of a two-dimensional planet is uniquely determined by travel times between points on the surface, up to an unusual gauge freedom.
The structural assumption we impose on the stiffness tensor field is that each slowness surface consist of a circle and an ellipse, and this model includes all isotropic materials. Our geometric inverse problem boils down to the famous boundary rigidity problem: ``Does a boundary distance function of a smooth compact Riemannian manifold determine the manifold and the metric tensor up to a smooth boundary fixing diffeomorphism?'' As this problem is not uniquely solvable, we need to set some geometric assumptions for both pressure- and shear-wave ($p$-wave and $s$-wave) geometries.
For instance, it has been proposed in~\cite{uhlmann2004cauchy} that lens rigidity (a slight strengthening of the boundary rigidity question we just asked) holds if the manifold has a strictly convex boundary and all geodesics exit the manifold in a finite time (non-trapping). If the non-trapping assumption is removed it is easy to construct non-isometric manifolds that have the same lens data (and therefore also the boundary distance functions).
Under the additional assumption of \textit{simplicity} (which precludes conjugate points; see \cite[Definition 3.8.1]{Paternain2023}) the affirmative answer for this question was provided by Pestov and Uhlmann in their seminal work~\cite{pestov2005two}.
Whether all convex two dimensional manifolds that have no trapped geodesics are lens rigid is among the most important open geometric inverse problems.

In this paper we assume that all our stiffness tensors give rise to lens rigid $p$- and $s$-wave geometries, and show that up to certain somewhat odd coordinate transformations, 
\begin{itemize}
    \item[(1)] nowhere isotropic, positive definite, elliptic, and smooth stiffness tensor fields, with prescribed $p$- and $s$-wave lens data, come in pairs that give the same data but are otherwise unique.
    \item[(2)] a smooth, positive definite, elliptic, stiffness tensor field with an isotropic point is uniquely determined from its $p$- and $s$-wave lens data. 
\end{itemize}
The conclusion of (1) is that in the elliptically anisotropic elasticity the gauge  for travel time tomography is larger than in the isotropic elasticity (or in the scalar case), while (2) tells us that even one isotropic point reintroduces the rigidity. 

The rigorous version of our first main result is given next. We refer to Section~\ref{subsec_Main_results} for the definition therein and properties of the map $\psi$ that takes a stiffness tensor to the respective $p$- and $s$-wave Riemannian metrics $(g_p,g_s)$. See also Section~\ref{sec:scope} that discusses the scope of this result.

\newcommand{\lens}{\mathcal{L}}
Let us denote the lens data of a Riemannian metric~$g$ by~$\lens_g$.

\begin{theorem}[Elastic lens rigidity]
\label{thm:travelTimes}
Let~$\Omega \subset \R^2$ be a smooth bounded connected domain with two smooth analytically admissible stiffness tensor fields $c_1$ and $c_2$ on $\bar\Omega$, as in Definition \ref{def:admissible}, with $\psi(c_1)=(g_{p,1},g_{s,1})$ and $\psi(c_2)= (g_{p,2},g_{s,2})$. 
Suppose that $(\bar\Omega,g_{p,1})$ and $(\bar\Omega,g_{s,1})$ are lens rigid Riemannian manifolds.
    Then the following are equivalent:
    \begin{enumerate}
        \item 
    $\lens_{g_{1,p}}=\lens_{g_{2,p}}$ and $\lens_{g_{1,s}}=\lens_{g_{2,s}}$.
    \item 
    There exists a diffeomorphism $\phi\colon \bar\Omega \rightarrow \bar\Omega$ with $\phi \vert_{\partial \Omega}=\operatorname{id}$ so that
    \begin{equation}
        c_1 = \psi_\pm^{-1}(\phi^* g_{p,2},g_{s,2}). 
    \end{equation}
    \end{enumerate}
\end{theorem}

\begin{remark}
Analytic admissibility of the stiffness tensor~$c$ requires that $g_p<g_s$ as quadratic forms at every point and the diffeomorphism~$\phi$ pulls back~$g_p$ but not~$g_s$, so sufficiently exotic diffeomorphisms will easily make $(\phi^* g_{p},g_{s})$ fall outside the range of~$\psi$.
The set of well-behaved diffeomorphisms~$\phi$ depends on the slowness surface of the stiffness tensor field~$c$. We show in Section~\ref{sec:Diffeos} that this set is open in~$C^1$ topology, but never a subgroup of all boundary fixing diffeomorphisms.
\end{remark}

Our proofs rely on an interplay between algebraic and geometric descriptions of elliptic elasticity. Our key tool is a polynomial characterization of elliptic anisotropy (see propositions~\ref{prop:alg-admissible} and~\ref{prop:an-admissible}) without assuming that the symmetry axes coincide with the coordinate axes.
The conditions are, of course, invariant under rotations and reflections of the plane.
Algebraic tools also allow us to show that in elliptic anisotropy the correspondence between smooth stiffness tensor fields and $p/s$-wave metrics is smooth, which helps keep the statement of Theorem~\ref{thm:travelTimes} smooth.
We give a brief introduction to algebraic geometry to the uninitiated reader in Appendix~\ref{app:AG}.

\subsection{Elliptically Anisotropic Elasticity}

We model the medium by an open, connected, and bounded domain $\Omega \subset \R^2$ with smooth boundary $\partial \Omega$.
The elastic properties in linear elasticity are described by its density $\rho$ (a scalar field) and stiffness tensor $c$ (a rank 4 tensor field modeling the linear strain-to-stress map).
Stress and strain are both symmetric $2\times2$ matrices, which leads to the minor symmetry
\begin{equation}\label{eq:sym1}
c_{ijkl}
=
c_{jikl}
=
c_{ijlk}
.
\end{equation}
Energy considerations lead to the additional major symmetry
\begin{equation}\label{eq:sym2}
c_{ijkl}
=
c_{klij}
\end{equation}
and the positivity property%
\footnote{Here and henceforth we follow the Einstein sum convention despite all our indices being down. That is, $c_{ijkl}A_{ij}A_{kl}\coloneqq\sum_{i,j,k,l=1}^2c_{ijkl}A_{ij}A_{kl}$.}
\begin{equation}
\label{eq:pos}
c_{ijkl}A_{ij}A_{kl} >0
\end{equation}
for all symmetric non-zero $2\times 2$-matrices~$A$.
These properties are satisfied at all points $x\in\bar\Omega$.

The small displacements $u(t,x)$ from equilibrium satisfy the elastic wave equation
\begin{equation}
\label{eq:elastic_wave_eq}
\rho(x)\partial_t^2u_i(t,x)
-
\partial_{x_j}[c_{ijkl}(x)\partial_{x_k}u_l(t,x)]
=
0.
\end{equation}
The propagation of singularities of a solution $u$ to this equation can be understood through microlocal analysis.
As it turns out that the kinematics of singularities (that is, ignoring amplitude information) only depends on the ratio $c/\rho$.
We therefore set $\rho=1$ for the rest of the discussion, as it will have no kinematic effect.

The slowness surface (or the affine characteristic variety of the elastic wave operator) is the vanishing set of the determinant of the principal symbol.
More concretely, fix any point $x\in\bar\Omega$.
Let us denote the cotangent vector atop it by $\xi\in T_x^*\R^2=\R^2$.
The Christoffel matrix $ \Gamma_{i\ell}(x,\xi)$ is given by
\begin{equation}
    \Gamma_{i\ell}(x,\xi)\coloneqq c_{ijk\ell}(x)\xi_j \xi_k, \quad x \in \bar \Omega.
\end{equation}
Positivity and symmetry of this matrix follow from the properties of the stiffness tensor.
The determinant is given by the slowness polynomial
\begin{equation}\label{eq:spoly}
\spoly{c}(x,\xi)
\coloneqq 
\det \left( \Gamma(x,\xi) - I \right),
\end{equation}
which is indeed a degree $4$ polynomial in $\xi$.
The slowness surface at $x\in\bar\Omega$ is the vanishing set of $\spoly{c}(x,\dummy)$ in $T_x^*\bar\Omega$.

A stiffness tensor is said to be elliptic when the slowness surface is a union of two ellipses.
(This is not to be confused with the spatial part of the elastic wave operator being elliptic, which is a far more general condition.)
Our focus is on a subclass of elliptic stiffness tensors we call \emph{admissible}.

\begin{definition}\label{def:admissible}
A  positive definite symmetric stiffness tensor, as in \eqref{eq:sym1}--\eqref{eq:pos}, is said to be \emph{algebraically admissible} 
if its slowness surface is a union of a circle and an ellipse, both centered at the origin.
If in addition the circle and the ellipse do not meet, the stiffness tensor is said to be \emph{analytically admissible}.
\end{definition}
Figure~\ref{fig:Admissible} illustrates slowness surfaces of admissible stiffness tensors and figure~\ref{fig:Elliptic} those of inadmissible elliptic ones.

Through out the paper we use the Voigt notation to write the stiffness tensor as
\begin{equation}
\label{eq:Voigt-matrix}
c
=
\begin{bmatrix}
    c_{1111} & c_{1122} & c_{1112}\\
     & c_{2222} & c_{1222}\\
    &  & c_{1212}
\end{bmatrix}
=
\begin{bmatrix}
    c_{11} & c_{12} & c_{13}\\
    & c_{22} & c_{23}\\
    && c_{33}
\end{bmatrix}
\end{equation}
for the stiffness tensors and note that due to the symmetries \eqref{eq:sym1}--\eqref{eq:sym2} and positivity~\eqref{eq:pos} of the stiffness tensor this $(3\times 3)$-matrix is symmetric and positive definite (see Lemma~\ref{lem:posdefcond} for details). 

In Lemma~\ref{lem:fixed_axes} we show that the cartesian coordinate axes are the symmetry axes of $c$ if and only if $c_{13}=c_{23}=0$.
In Lemma~\ref{lma:special-characterization} we show that such a $c$ is analytically admissible if and only if
\begin{equation}
\label{eq:ellipticity_in_fixed_axes}
(c_{12}+c_{33})^2=(c_{11}-c_{33})(c_{22}-c_{33})\neq 0.
\end{equation}
Our description of elliptic materials in 2D is, of course, closely related to that in 3D.

\begin{figure}
    \centering
    \begin{minipage}[t]{0.5\textwidth}
        \centering
        \includegraphics[width=0.7\textwidth]{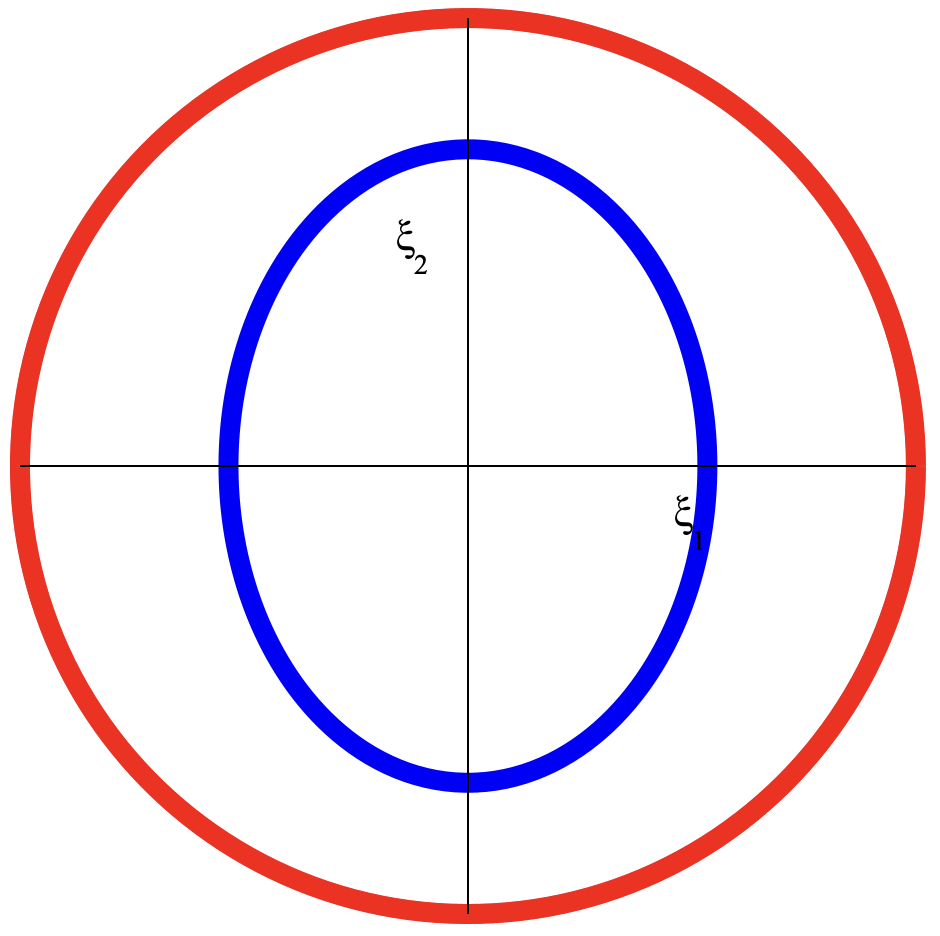}
        \caption*{Analytically admissible.}
    \end{minipage}%
    \begin{minipage}[t]{0.5\textwidth}
        \centering
        \includegraphics[width=0.7\textwidth]{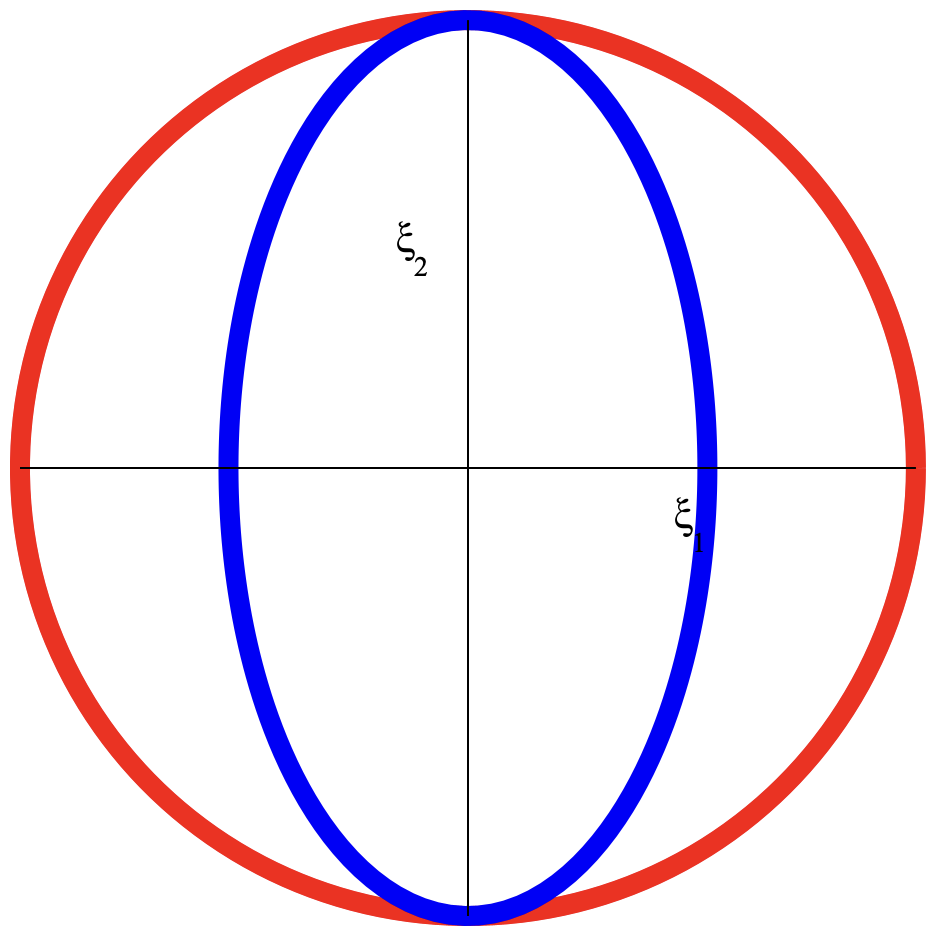}
        \caption*{Algebraically admissible.}
    \end{minipage}
    \caption{Slowness surfaces corresponding to the stiffness tensors $\begin{bmatrix}
       7 & -2-\sqrt{10} & 0\\
       -2-\sqrt{10} & 4 & 0\\
       0 & 0 & 2
       \end{bmatrix}$
       and 
       $\begin{bmatrix}
       7 & -2 & 0\\
       -2 & 2 & 0\\
       0 & 0 & 2
    \end{bmatrix}$.}
    \label{fig:Admissible}
\end{figure}
\begin{figure}
    \centering
    \begin{minipage}[t]{0.33\textwidth}
        \centering
        \includegraphics[width=0.8\textwidth]{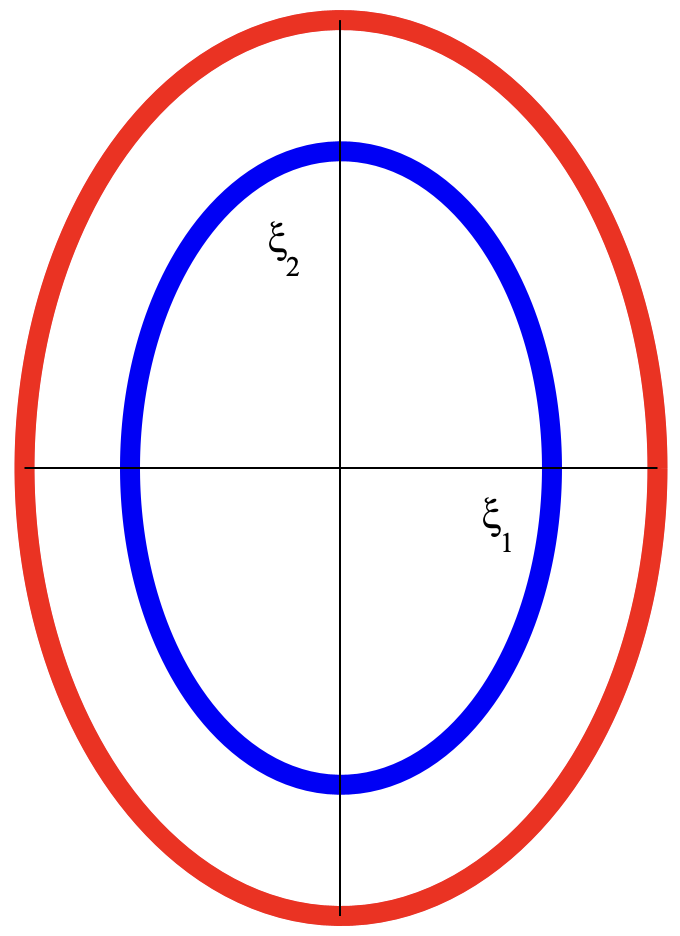}
    \end{minipage}%
    \begin{minipage}[t]{0.33\textwidth}
        \centering
        \includegraphics[width=\textwidth]{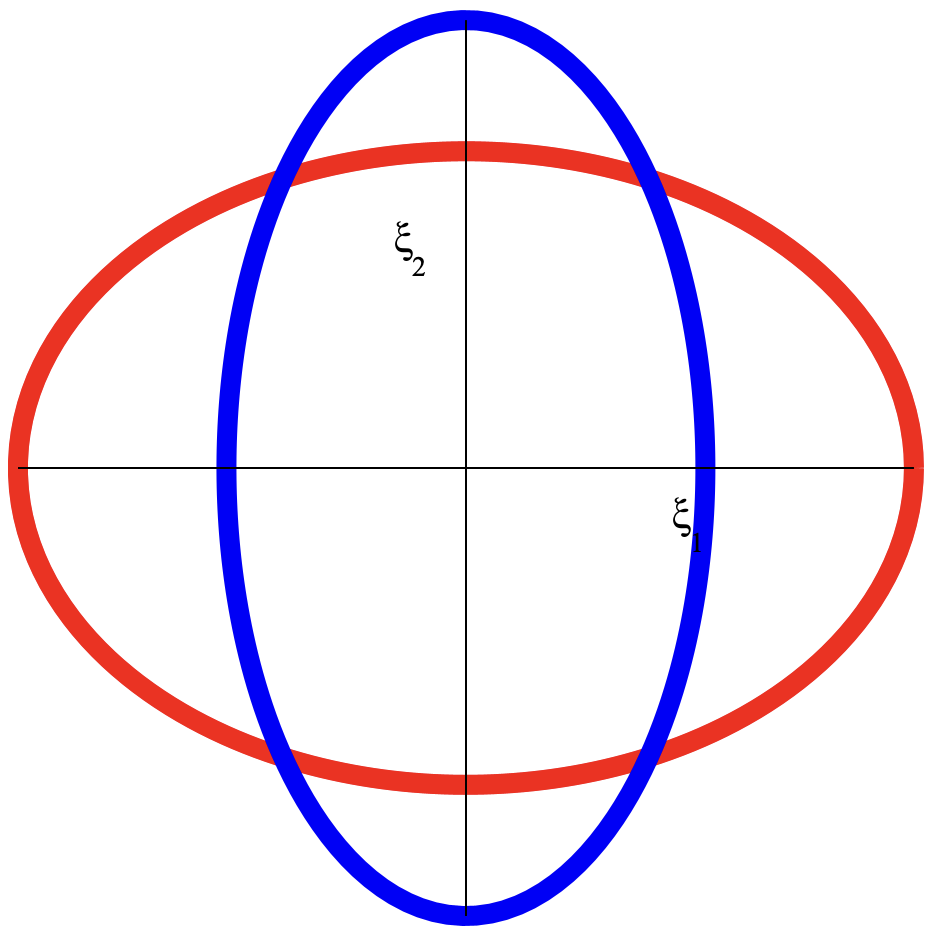}
    \end{minipage}%
    \begin{minipage}[t]{0.33\textwidth}
        \centering
        \includegraphics[width=0.9\textwidth]{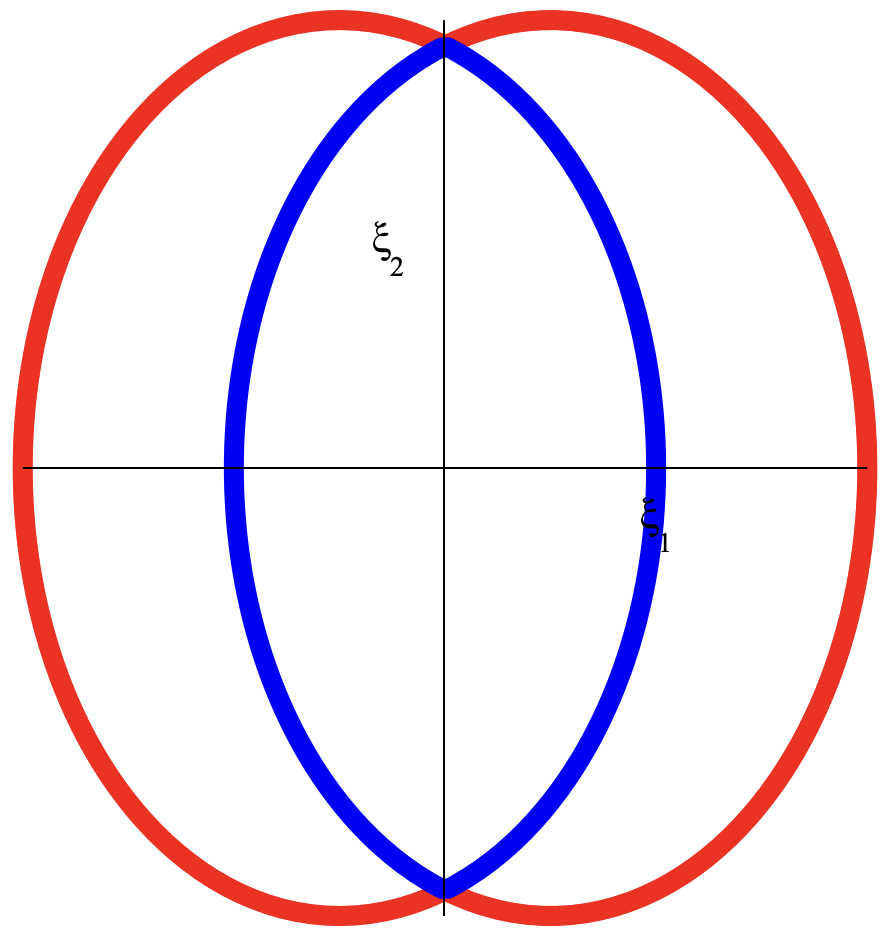}
    \end{minipage}
    \caption{Slowness surfaces of some elliptic stiffness tensors that are not considered in this paper. Each one is a union of two ellipses in a different way. The corresponding stiffness tensors are
{\tiny$\begin{bmatrix}
9 & -4 & 0\\
-4 & 2 & 0\\
0 & 0 & 4
\end{bmatrix}$,
$\begin{bmatrix}
7 & -2 & 0\\
-2 & 4 & 0\\
0 & 0 & 2
\end{bmatrix}$,
$\begin{bmatrix}
4 & 0 & 0\\
0 & 1 & 0\\
0 & 0 & 1
\end{bmatrix}$
}.
}
    \label{fig:Elliptic}
\end{figure}

In 3D Chadwick and Norris~\cite{ChadwickNorris1990} describe conditions for a slowness surface to be a union of aligned ellipsoids, and a condition very similar to~\eqref{eq:ellipticity_in_fixed_axes} makes several appearances.
The indices and accompanying linear assumptions vary depending on the case, but orthorhombic symmetry is always required.
Mazzucato and Rachele~\cite{MazzucatoRachele2008} describe all classes of inhomogeneous, transversely isotropic elastic media for which the sheets associated to each wave mode are ellipsoids. 
They also derive pointwise conditions that guarantee that the sheet of the slowness surface corresponding to a given wave mode is disjoint from the others.
They also show~\cite{MazzucatoRachele2007} that the hyperbolic DN-map of the elastic wave equation uniquely determines the travel time between boundary points for a disjoint mode if there is one.
Elliptically anisotropic materials have also been studied experimentally~\cite{bhuiyan2016anisotropic}.

Ellipticity in 2D is the same as reducibility of the slowness polynomial.
Among all stiffness tensors reducibility is rare as shown in~\cite{de2023reconstruction} and further quantified in~\cite{2d-elastic-characterization}.
In the present paper we do not classify all cases of elliptic anisotropy in 2D, but rather focus on giving a complete geometric description of a specific subcase.

\subsection{Kinematic inverse problems}

An ellipse centered at the origin is the unit sphere of a norm arising from an inner product.
For an admissible stiffness tensor the slowness surface atop each $x\in\bar\Omega\subset\R^2$ is given by an ellipse and a circle, so for an admissible stiffness tensor field the slowness surface at each point is the union of the unit cospheres of two Riemannian metric tensor fields.

The singularities of solutions to the elastic wave equation~\eqref{eq:elastic_wave_eq} propagate along a Hamiltonian flow whose level sets these sheets of the slowness surface are~\cite{dencker1982propagation}.
In other words, the singularities of the elastic wave equations 
\eqref{eq:elastic_wave_eq} follow the Riemannian geodesics of the two metrics.
One of the metrics corresponds to $p$-waves (faster) and the other one to $s$-waves (slower).
The distance between two points in this geometry is the travel time of waves between the points.
Without any ellipticity assumption of the stiffness tensor the singularities will typically not follow any Riemannian geodesic flow.

Arrival time measurements at $\partial\Omega$ for elastic waves from sources on $\partial\Omega$ therefore correspond to the restriction of the Riemannian distance function to $\partial\Omega\times\partial\Omega$.
Whether this boundary distance data uniquely determines the Riemannian manifold is the boundary rigidity question.
If the geometry is assumed to be simple, 
one can also compute the initial and final directions of any geodesic between two boundary points.
This enriched data is called lens data, and the corresponding geometric inverse problem is that of lens rigidity.

A geometric description of propagation of singularities is not restricted to the elliptic stiffness tensors.
In non-elliptic cases the singularities on the $qP$ (quasi-$p$) branch follow Finsler geometries.
Regularity of this geometry in 2D was studied in~\cite{IKK25}, and it gives a concrete condition for the $qP$ and $qS$ branches of the slowness surface never touching without any assumption of ellipticity.

\subsection{Anomalous companions}

As described above, a stiffness tensor (material parameters at a point) gives rise to a slowness surface, and these slowness surfaces at all points determine the kinematics of elastic waves.
But does the slowness surface uniquely determine the underlying stiffness tensor?

This uniqueness is often but not always true~\cite{de2023reconstruction,2d-elastic-characterization}.
Two stiffness tensors that give rise to the same slowness surface are called \emph{anomalous companions}.
In other words, if $f$ is the map that takes a stiffness tensor $c$ to the corresponding slowness surface $\Sigma_c$, then the set of anomalous companions of $c$ is $f^{-1}(f(c))\setminus\{c\}$.
In our specific setting the role of $f$ is played by the map $\psi$ defined below.

\subsection{Our scope}
\label{sec:scope}

This paper focuses on applying the results for these classical geometric inverse problems to the kinematics of the elastic wave equation for admissible materials.
Our main task is to find all analytically admissible stiffness tensor fields that give rise to the same lens data for both $p$ and $s$ waves.

The complete answer we give to this question requires understanding the gauge freedom of the geometric inverse problem and the correspondence between stiffness tensors and slowness surfaces.
Theorem~\ref{thm:travelTimes} pertains to this correspondence.

The reason to focus on elliptic anisotropy is to be able to draw from the large body of literature on geometric inverse problems for Riemannian manifolds.
A more general theory adapted to elastic geometry instead of e.g. fully general Finsler geometry is not sufficiently developed yet.
As it turns out, even in this seemingly simple setting the complete answer is subtle.

There are many reasons to focus on two dimensions.
First, in three dimensions every slowness surface of an almost but not quite isotropic stiffness tensor is non-smooth~\cite{IKK25}, so analytically admissible stiffness tensors in the sense one might hope simply do not exist.
Second, the algebraic geometry we use to analyze the correspondence between material properties and geometry is much more tractable (both by hand and by computer) in two dimensions.
Third, boundary rigidity is not known for general simple manifolds in three or more dimensions.

In an isotropic medium the kinematic information or lens data can be extracted from the hyperbolic Dirichlet-to-Neumann map of the elastic wave equation~\eqref{eq:elastic_wave_eq}  \cite{rachele2000inverse,stefanov2017elastic}. See also Section~\ref{sec:related} for more details. We expect this to be true for quite general stiffness tensors including the elliptic ones we study in this paper.
Full PDE data will contain more information than the lens data, but typical analysis will start with extracting the principal data and using it to solve a geometric inverse problem.
Only then does one usually proceed to study subprincipal effects and further constrain the material parameters \cite{bhattacharyya2018local, rachele2003uniqueness, zhai2025determination}.
In our admissible setting these further constraints would concern the density field $\rho$, the diffeomorphism $\phi$, and the anomalous companions. 
The present paper lays a major part of the foundation for such an analysis by fully solving the principal level geometric inverse problem in terms of the material parameters.

\subsection{Main results}
\label{subsec_Main_results}

\subsubsection{Pointwise properties}
\label{sec:pointwise}

Let us denote by $\analyticallyadmissibletensors$ the space of analytically admissible stiffness tensors, see Definition~\ref{def:admissible}, while the notation $\admissiblepairsofmetrics$ is reserved for the space of pairs $(g_p,g_s)$ of positive definite quadratic forms so that $g_s-g_p$ is positive definite but not zero and $g_s$ is a multiple of the identity matrix.
We denote by $\admissiblepairsofmetrics_-\subset\admissiblepairsofmetrics$ the subset of those pairs $(g_p,g_s)$ where the first quadratic form is not a multiple of the identity.
We also denote $\admissiblepairsofmetrics_+=\admissiblepairsofmetrics$ for convenience.
Finally, we let $\psi\colon \analyticallyadmissibletensors \to \admissiblepairsofmetrics$ be the map that takes an analytically admissible stiffness tensor to the corresponding quadratic forms associated to $p/s$-wave metrics:
\begin{equation}
    \psi(c)=(g_p,g_s).
\end{equation}

We verify some properties of the quantities defined above. Particularly, we show that the correspondence between $c$ and $(g_p,g_s)$ is diffeomorphic. 
These are later used to state our second main theorem (the size of the gauge group).

\begin{proposition}
\label{prop:components+psismooth}
The map $\psi\colon \analyticallyadmissibletensors\to \admissiblepairsofmetrics$ satisfies the following properties:
\begin{enumerate}
\item
The map $\psi$ is onto.
\item
The set $\analyticallyadmissibletensors\subset\R^6$ is a smooth $4$-dimensional submanifold and has two connected components, $\analyticallyadmissibletensors_+$ and $\analyticallyadmissibletensors_-$.
\item
The map $\psi$ is smooth 
and the restrictions $\psi_\pm=\psi|_{\analyticallyadmissibletensors_\pm}\colon \analyticallyadmissibletensors_\pm\to \admissiblepairsofmetrics_\pm$ are diffeomorphisms.
\item
The set $\psi^{-1}(\admissiblepairsofmetrics\setminus \admissiblepairsofmetrics_-)\subset \analyticallyadmissibletensors$ coincides with the set of isotropic stiffness tensors.
\item
The map $\psi$ is 1-to-1 on the set of isotropic stiffness tensors and 2-to-1 outside it.
\end{enumerate}
\end{proposition}

We would like to point out two implications of this result: Firstly, all positive definite isotropic stiffness tensors are contained in $\analyticallyadmissibletensors_+$. Secondly, no analytically admissible stiffness tensor can have a slowness surface where the ellipse encloses the circle. Note that this is not explicitly forbidden in Definition~\ref{def:admissible} but is a consequence of positiveness~\eqref{eq:pos}. 

We characterize algebraically and analytically admissible ones among all stiffness tensors in Propositions~\ref{prop:alg-admissible} and~\ref{prop:an-admissible} in terms of explicit polynomial conditions.

\subsubsection{Global properties}
\label{sec:global}

Section~~\ref{sec:pointwise} concerns stiffness tensors and quadratic forms, while our main results concern stiffness tensor fields and Riemannian metrics.
In Theorem~\ref{thm:travelTimes} we described the geometric inverse problem of recovering an analytically admissible smooth stiffness tensor $c\in \analyticallyadmissibletensors$ from the lens data of the respective pair of Riemannian metrics $\psi(c)$, which by Proposition~\ref{prop:components+psismooth} are smooth.
In general the eigenvalues $\lambda_1(x,\xi),\lambda_2(x,\xi)$ of the Christoffel matrix $\Gamma(x,\xi)$ need not be smooth in $x$.
This is one of the crucial reasons why we require the ellipse and circle to be disjoint; this assumption makes the eigenvalues simple and therefore well-behaved.

As the lens data stay invariant over all diffeomorphisms $\phi \colon \bar \Omega \to \bar\Omega$ that fix all boundary points, this introduces the natural gauge for problems of this type.
A Riemannian metric can only be defined up to a boundary fixing diffeomorphism from its lens data.
Unlike in many other geometric inverse problems, not all diffeomorphisms $\phi$ can appear in our elastic setting due to the imposed structure of the set $\admissiblepairsofmetrics$.
Our next main theorem describes this phenomenon.

\begin{theorem}[Size of the preimage of $\psi$]
\label{thm:how-many-stiffness-tensor-fields}
Let $\Omega \subset \R^2$ be a smooth bounded connected domain and let $c\colon\bar\Omega\to \analyticallyadmissibletensors$ be a smooth stiffness tensor field.
For a diffeomorphism $\phi\colon\bar\Omega\to\bar\Omega$ and two metrics $(g,h)\in C^\infty(\bar\Omega,\admissiblepairsofmetrics)$, let us denote $\phi^*(g,h)=(\phi^*g,h)$ so that the pullback only acts on the first metric tensor field.

Let us denote by $\psi^{-1}(\phi^*\psi(c))$ the set of admissible stiffness tensor fields $c'$ for which $\psi(c')=\phi^*\psi(c)$.
We always have that $\#\psi^{-1}(\phi^*\psi(c))\in\{0,1,2\}$.

More specifically, denoting $\psi(c)\eqqcolon(g_p,g_s)$ we have
\begin{enumerate}
\item
$\#\psi^{-1}(\phi^*\psi(c))=0$ if and only if there is  $x \in \bar\Omega$ and $v\in \R^2\setminus\{0\}$ for which $\abs{v}_{\phi^\ast g_p}\geq \abs{v}_{g_s}$ at $x$.
\item
$\#\psi^{-1}(\phi^*\psi(c))=1$ if and only if $\phi^\ast g_p$ is a multiple of the identity matrix at some point $x\in\bar\Omega$ and
$g_s>\phi^\ast g_p$ in all of $\bar\Omega$.
\item
$\#\psi^{-1}(\phi^*\psi(c))=2$ if and only if $\phi^\ast g_p$ is nowhere a multiple of the identity and
$g_s>\phi^\ast g_p$ in all of $\bar\Omega$.
\end{enumerate}
\end{theorem}

According to case (3) of Theorem \ref{thm:how-many-stiffness-tensor-fields} nowhere isotropic analytically admissible stiffness tensor fields that give rise to the same wave speed metrics come in pairs. These two stiffness tensor fields are the anomalous companions. The important consequence of Theorem~\ref{thm:how-many-stiffness-tensor-fields} is its implication to the geometric inverse problem stated in Theorem~\ref{thm:travelTimes}.

The set $\analyticallyadmissibletensors$ has two connected components as per Proposition~\ref{prop:components+psismooth}(2).
Therefore a smooth stiffness tensor field on a connected domain will only take values in one component.

\subsubsection{Proof of Theorem~\ref{thm:travelTimes}, and implications of theorems~\ref{thm:travelTimes} and~\ref{thm:how-many-stiffness-tensor-fields}}

\begin{proof}[Proof of Theorem~\ref{thm:travelTimes}]
First we note that by Proposition~\ref{prop:components+psismooth} the metrics $g_{p/s,1/2}$ are smooth. Hence $(\bar\Omega, g_{p/s,1/2})$ are smooth Riemannian manifolds.

Suppose that (1) is true.
Since $g_{p/s,1}$ is lens rigid and the lens data of $g_{p/s,2}$ is that of $g_{p/s,1}$ there are some boundary fixing diffeomorphisms $\phi, \varphi\colon \bar \Omega \to \bar \Omega$, such that 
\begin{equation}
g_{p,1}=\phi^\ast g_{p,2}
\quad \text{and} \quad
g_{s,1}=\varphi^\ast g_{s,2}.
\end{equation}
Since $g_{s,1}$ and $g_{s,2}$ are conformally Euclidean metrics, the map $\varphi$ is conformal, and as it is an identity on the boundary of $\Omega$ it is actually the identity of $\bar \Omega$~\cite{ashwin-elasticity}.
Therefore, $g_{s,1}=g_{s,2}$, and we have proved 
\begin{equation}
\psi(c_1)=(g_{p,1},g_{s,1})=(\phi^\ast g_{p,2},g_{s,2}).
\end{equation}
Since $c_1 \in \analyticallyadmissibletensors$ it holds that $(\phi^\ast g_{p,2},g_{s,2}) \in \admissiblepairsofmetrics$.
By part (3) of Proposition \ref{prop:components+psismooth} $\psi_\pm$ is a bijective diffeomorphism and it follows that
    \begin{equation}
        c_1 = \psi_\pm^{-1}(\phi^* g_{p,2},g_{s,2}).
    \end{equation}

Suppose that (2) is true. Since $c_1 \in \psi^{-1}(\phi^*\psi(c_2))$ it follows by the definition of $\psi$ and $\phi^*(g,h)=(\phi^*g,h)$ that 
\begin{equation}
    (g_{p,1},g_{s,1})=\psi(c_1)=\phi^* \psi(c_2)=\phi^*(g_{p,2},g_{s,2})=(\phi^*g_{p,2},g_{s,2}).
\end{equation}
Thus, the metrics $g_{p,2}$ and $\phi^*g_{p,1}$, as well as $g_{s,2}$ and $g_{s,1}$, have the same lens data.
Since $\phi$ is a smooth diffeomorphism fixing all boundary points, also the metrics $g_{p,2}$ and $g_{p,1}$ have the same lens data.
\end{proof}

The conclusion of Theorem \ref{thm:travelTimes} needs an interpretation. Since we have that
\begin{equation}
\psi(c_1)=(\phi^* g_{p,2},g_{s,2})
\end{equation}
we see in the light of Theorem \ref{thm:how-many-stiffness-tensor-fields} that
$
c_1 = \psi_+^{-1}(\phi^* g_{p,2},g_{s,2})
$
if $c_1$ has even one isotropic point in $\bar \Omega$, and thus no anomalous companions. On the other hand, if no such points exist then $c_1$ has exactly one anomalous companion $\tilde c_1$. Particularly, the pairs of Riemannian metrics arising from the elastic tensors $c_1$, $\tilde c_1$ and $c_2$ do all have the same lens data and we cannot distinguish these tensors just from this kinematic data.

On the other hand one can turn the question around and ask:
Given an analytically admissible stiffness tensor field $c$, for which diffeomorphisms $\phi$ do we have $\phi^*\psi(c) \in \admissiblepairsofmetrics$?
Analytic admissibility requires that $g_p<g_s$ as quadratic forms at every point and the diffeomorphism pulls back $g_p$ but not $g_s$, so sufficiently exotic diffeomorphisms will easily make $\phi^*\psi(c)$ fall outside the range of $\psi$.
The set of well-behaved diffeomorphism $\phi$ depends on the stiffness tensor field $c$, is open in $C^1$ topology, but is never a subgroup of the group of all diffeomorphisms.
In addition to proving these properties in Section~\ref{sec:Diffeos} we provide a characterization of admissible diffeomorphisms. Particularly, we prove in Theorem~\ref{thm:admdiff:admissibility-preserving-diffeomorphisms} that a diffeomorphism $\phi \colon \bar \Omega \to \bar \Omega$ preserves the admissible structure exactly when it is a $q$-Lipschitz map from $(\bar \Omega,g_s)$ to $(\bar \Omega, g_p)$ for some $q\in (0,1)$.

\subsection{Structure of the paper}

The main contributions of this paper are Theorem~\ref{thm:travelTimes}, Proposition~\ref{prop:components+psismooth}, Theorem~\ref{thm:how-many-stiffness-tensor-fields}, and Theorem~\ref{thm:admdiff:admissibility-preserving-diffeomorphisms}. 
The bulk of our work lies in characterizing admissible stiffness tensors via polynomial conditions. This is carried out in sections \ref{sec:AlgAdStiffness} and \ref{sec:CharacSymmetry}. In sections~\ref{sec:psibranches} and~\ref{sec:ImplicitFuncThm} we prove  Proposition~\ref{prop:components+psismooth} and Theorem~\ref{thm:how-many-stiffness-tensor-fields}.
Once Proposition~\ref{prop:components+psismooth} is established the proof of Theorem~\ref{thm:how-many-stiffness-tensor-fields} follows directly from the definitions of the map $\psi$ and the pullback $\phi^*$ as well as Lemma~\ref{rem:discond} (the circle encloses the ellipse).
In the final section~\ref{sec:Diffeos} we study the properties of admissible diffeomorphisms and characterize this set.
We conclude the introductory section~\ref{sec:intro} to a review of related works.

\subsection{Inverse problems for wave equations via geometrization}
\label{sec:related}

A central approach to inverse problems for wave equations is to recover
the geometry governing wave propagation from boundary measurements and
then determine the coefficients from this geometry. In elasticity,
microlocal constructions and propagation of singularities \cite{Gr-Sj,dencker1982propagation} connect the hyperbolic
Dirichlet-to-Neumann (DN) map to travel times and lens data. This
approach underlies Rachele's work on isotropic elasticity \cite{rachele2000boundary,rachele2000inverse} and its extensions to
anisotropic media with a disjoint wave mode \cite{MazzucatoRachele2007,MazzucatoRachele2008}. When the corresponding
slowness sheets are ellipsoids, the resulting geometric inverse
problems concern Riemannian metrics. This problem is the boundary or lens rigidity.

Michel's conjecture~\cite{michel1981rigidite}, which asserts boundary
rigidity of simple Riemannian manifolds, was proved in two dimensions
by Pestov and Uhlmann~\cite{pestov2005two}. In higher dimensions,
results include rigidity within a conformal class~\cite{croke1991rigidity}, near the Euclidean metric~\cite{burago2010boundary}, and for generic simple metrics~\cite{SU3}.
The local methods of Uhlmann--Vasy~\cite{uhlmann2016inverse} and
Stefanov--Uhlmann--Vasy~\cite{stefanov2017local} also yield rigidity
results under suitable strictly convex foliation assumptions.
These assumptions have become central to several elastic inverse
problems. Further background can be found in
\cite{survey,paternain2023book}.

For the scalar wave equation on a compact Riemannian manifold, the
Boundary Control method recovers the geometry, up to isometry, from
sufficiently long boundary measurements without simplicity or convex
foliation assumptions
\cite{belishev1987approach,belishev1992reconstruction,Katchalov2001}.
The method relies on unique continuation~\cite{Tataru} and connects
the hyperbolic problem with the inverse boundary spectral problem.
For time-dependent Lorentzian metrics, unique continuation and
reconstruction require additional hypotheses
\cite{eskin2007inverse,Alexakis_Feiz_Oksanen_22,Alexakis_Feiz_Oksanen_23}.
Rigidity relative to Minkowski space~\cite{oksanen2024rigidity} and
nonuniqueness examples~\cite{oksanen2026counterexamples} illustrate
the different behavior of the general Lorentzian problem.

The scalar theory does not directly extend to elasticity because
elastic waves form a coupled system with several propagation speeds.
For isotropic elasticity, unique continuation is available
\cite{eller2002uniqueness}, but the boundary control results in
\cite{belishev2002dynamical} concern the smaller region reached by
the slower $s$-waves, rather than the entire region reached by the
$p$-waves. This limits the direct application of the Boundary Control
method and motivates approaches based on wave splitting or the
recovery of travel-time geometry.

For isotropic media, wave splitting yields determination of the wave
speeds when the shear modulus is constant~\cite{belishev2006dynamical}.
Rachele's results
\cite{rachele2000boundary,rachele2000inverse,rachele2003uniqueness}
recover the Lam\'e parameters and density under simplicity and
additional curvature assumptions, with density recovery excluding
the locus $\lambda=2\mu$. Under strictly convex foliation assumptions,
Stefanov--Uhlmann--Vasy~\cite{stefanov2017elastic} recover the
$p$- and $s$-wave speeds, and~\cite{bhattacharyya2018local} extends
the recovery to $(\lambda,\mu,\rho)$ away from that locus.
Zhai~\cite{zhai2025determination} removes the restriction
$\lambda=2\mu$. A recent rigidity result in~\cite{ilmavirta2026rigidity} shows that equality of the hyperbolic
DN map with that of a homogeneous Lam\'e system forces the unknown
system to be homogeneous, without imposing additional geometric
assumptions on the unknown coefficients.

In anisotropic elasticity, coordinate transformations introduce
further issues. Mazzucato and Rachele~\cite{mazzucato2006partial}
study transformations fixing the boundary to first order, invariance
of the hyperbolic DN map, and partial uniqueness for anisotropic
media in the orbit of isotropic media. The work
\cite{ilmavirta2023gauge} instead treats displacement as a vector
field under coordinate changes and studies the resulting invariances
of the Euclidean and Riemannian elastic wave equations. These include
joint conformal scalings of the metric, stiffness tensor, and density.
Such equation-level invariances must be distinguished from the
ambiguities of travel-time data considered here, which retain only
the geometry of propagation.

For general anisotropy, even the geometric reduction presents
difficulties: a disjoint wave mode may define a Finsler geometry
rather than a Riemannian one~\cite{de2019inverse}. Simple Finsler
manifolds need not be boundary rigid~\cite{ivanov2013local}, although
these counterexamples do not by themselves establish nonuniqueness
within the class of geometries arising from elasticity. Positive
results therefore often exploit additional material structure.
For piecewise analytic coefficients,~\cite{de2019unique} proves
global uniqueness in an orthorhombic setting and also treats fully
anisotropic, piecewise constant stiffness tensors.

For transversely isotropic media,~\cite{de2020recovery} uses
$qSH$ travel times to recover the symmetry-axis tilt and some
material parameters under suitable convexity assumptions. The
$qP$ and $qSV$ data determine further parameters when the remaining
ones are known or satisfy suitable relations and nondegeneracy
conditions. The stability estimates of~\cite{zou2024partial}
provide related partial recovery results, including injectivity for
parameter differences supported in sets of small width. Another
approach, developed in~\cite{lassas2026determination}, determines
fully anisotropic perturbations of an isotropic homogeneous reference
from single-scattering data, with additional stability and rigidity
results for isotropic perturbations.

The present paper studies the ambiguity remaining at the level of
travel times for elliptically anisotropic planar media. The two wave
modes define Riemannian metrics, with the $s$-wave metric conformally
Euclidean. Assuming lens rigidity, we characterize the admissible
transformations of the $p$-wave metric and the algebraic ambiguity
in recovering the stiffness tensor from the metric pair. This
identifies the roles of coordinate freedom and anomalous companions
in this class of media.

\subsection*{AI disclosure and use of computational tools}

During the preparation of this manuscript, OpenAI’s ChatGPT GPT-6 Astra and Anthropic’s
Claude Opus 5 \& 5.5 were used for language editing and condensing presentation.
The proofs of Proposition~\ref{prop:two-circles}, Lemma~\ref{lem:rotation-orbit}, Proposition~\ref{prop:components+psismooth} (1),(4) \& (5) and Theorem~\ref{thm:how-many-stiffness-tensor-fields} were carried out with minor assistance from Claude. The proofs of Proposition~\ref{prop:alg-admissible}, Proposition~\ref{prop:two-circles} and Proposition~\ref{prop:components+psismooth} (2) \& (3) were computer assisted by Magma~\cite{magma}. 
Section~\ref{sec:Diffeos} is based on conversations with ChatGPT. Most of the computations in this section were carried out by ChatGPT GPT-6 Astra while verified by the authors.
The authors take responsibility for the all mathematical statements and proofs presented in this paper.

\subsection*{Acknowledgments}
JI was supported by the Research Council of Finland (Flagship of Advanced Mathematics for Sensing Imaging and Modelling grant 359208; Centre of Excellence of Inverse Modelling and Imaging grant 353092; and other grants 351665, 351656, 358047, 360434) and a Väisälä project grant by the Finnish Academy of Science and Letters.
TS was supported by the National Science Foundation (DMS-2510272) and the  Simons Foundation Travel Support for Mathematicians (MPS-TSM-00013291).
HS was supported by the Research Council of Finland (Flagship of Advanced Mathematics for Sensing Imaging and Modelling grant 359208).
We thank Antonio Cocan for helping troubleshoot the proof of Lemma~\ref{lem:smoothMap}. We are grateful to Jian Zhai and Hanming Zhou for helpful discussions and suggestions.

\section{Characterizations and properties of algebraically and analytically admissible stiffness tensors}
\label{sec:AlgAdStiffness}
This section is devoted to characterizing algebraically and analytically admissible stiffness tensors and subclasses of these such as isotropic stiffness tensors. Additionally, we prove properties related to rotations and reflections of such stiffness tensors. These characterizations and properties are needed for the proofs of Proposition \ref{prop:components+psismooth} and Theorem \ref{thm:how-many-stiffness-tensor-fields}.
For rotations and reflections we consider the following notation throughout this paper:

We note that the symmetric tensor square $\Sym^2(\R^2)$ of $\R^2$ is the space of symmetric real $2\times2$ matrices.
Stiffness tensors are elements of the space $\Sym^2(\Sym^2(\R^2))$.
The tautological representation of $O(2)$ on $\R^2$ induces representations on $\Sym^2(\R^2)$ and $\Sym^2(\Sym^2(\R^2))$.
For the rotation by an angle $\theta$,
\begin{equation}
\rotA{\theta}
=
\begin{bmatrix}
\cos(\theta) & -\sin(\theta)\\
\sin(\theta) & \cos(\theta)
\end{bmatrix}
\in O(2)
,
\end{equation}
and the reflection around the line at angle $\theta$,
\begin{equation}
\reflA{\theta}
=
\begin{bmatrix}
\cos(2\theta)  & \sin(2\theta) \\
\sin(2\theta)  & -\cos(2\theta) 
\end{bmatrix}
\in O(2)
,
\end{equation}
we denote the induced actions by
$\rotB{\theta},\reflB{\theta}\in GL(\Sym^2(\R^2))$
and
$\rotC{\theta},\reflC{\theta}\in GL(\Sym^2(\Sym^2(\R^2)))$.
The actions of $\rotB{\theta}$ and $\reflB{\theta}$ on a symmetric real $2\times2$ matrix $A$ are
\begin{align}
    (\rotB{\theta}A)_{mn}
&=
A_{ij}
\rotA{\theta}_{mi}
\rotA{\theta}_{nj}
,
\label{eq:rotB}\\
    (\reflB{\theta}A)_{mn}
&=
A_{ij}
\reflA{\theta}_{mi}
\reflA{\theta}_{nj}
.\label{eq:reflB}\
\end{align}

The actions of $\rotC{\theta}$ and $\reflC{\theta}$ on a stiffness tensor $c$ are
\begin{align}
    (\rotC{\theta}c)_{mnop}
&=
c_{ijk\ell}
\rotA{\theta}_{mi}
\rotA{\theta}_{nj}
\rotA{\theta}_{ok}
\rotA{\theta}_{p \ell},\label{eq:rotC}\\\
(\reflC{\theta}c)_{mnop}
&=
c_{ijk\ell}
\reflA{\theta}_{mi}
\reflA{\theta}_{nj}
\reflA{\theta}_{ok}
\reflA{\theta}_{p \ell}
.\label{eq:reflC}
\end{align}
The entries of $\rotC{\theta}c$ and $\reflC{\theta}c$ are reported in Appendix \ref{ap:rotrefl}. Before we characterize algebraically and analytically admissible stiffness tensors we make the following observation concerning positive definiteness of stiffness tensors that will be used throughout this paper:
\begin{lemma}
\label{lem:posdef}
A stiffness tensor $c$ is positive definite if and only if the matrix~\eqref{eq:Voigt-matrix} is positive definite.
\end{lemma}

\begin{proof}
If we express a symmetric $2\times2$ matrix as a $3$-vector as $\hat A=(A_{11},A_{22},2A_{12})$, then
\begin{equation}
        c_{ijk\ell}A_{ij}A_{k\ell} 
        =
        \hat A^Tc \hat A,
\end{equation}
where $c$ on the right-hand side is from~\eqref{eq:Voigt-matrix}.
This proves the claim.
\end{proof}

\subsection{Characterization of algebraic and analytic admissibility}

We characterize an admissible stiffness tensor of Definition \ref{def:admissible} through the polynomials 
\begin{equation}\label{eq:P1}
P_1(c)
=
-c_{11}c_{23} + c_{12}c_{13} - c_{12}c_{23} + c_{13}c_{22}
\end{equation}
and
\begin{equation}\label{eq:P2}
P_2(c)
=
-c_{11}c_{22} + c_{11}c_{33} + c_{12}^2 + 2c_{12}c_{33} - c_{13}^2 - 2c_{13}c_{23} + c_{22}c_{33} - c_{23}^2
\end{equation}
in the space of stiffness tensors written in the Voigt notation~\eqref{eq:Voigt-matrix}.

\begin{proposition}[Algebraically admissible $c$]
\label{prop:alg-admissible}
The slowness surface of a positive definite stiffness tensor $c$ is a union of a concentric circle and ellipse if and only if $P_1(c)=P_2(c)=0$.
That is,
\begin{equation}
\label{eq:char-of-alg-adm}
\algebraicallyadmissibletensors
\cap\positivetensors
=
\V(P_1,P_2)
\cap\positivetensors
,
\end{equation}
where $\V(P_1,P_2)$ is the vanishing set of $P_1$ and $P_2$ and $\positivetensors$ is the set of all positive definite stiffness tensors.
\end{proposition}

\begin{proof}
We give only a brief proof outline; a more detailed proof can be found in Appendix~\ref{app:long-proofs}.
The slowness polynomial can be spelled out explicitly as a polynomial in $(\xi_1,\xi_2)$.
If its vanishing set is to contain a circle and an ellipse, it must be divisible by polynomials vanishing on the circle and on the ellipse.
Given degree constraints, it has to be the product of these two.
We can parametrize the circle polynomial by a radius $r$ and the ellipse polynomial by coefficients $k$, multiply them, and compare with the slowness polynomial.
Setting all the coefficients to zero defines polynomial conditions in the variables $(c,r,k)$.

The ideal generated by these conditions defines a variety in the product space.
We project this variety to the $c$ component by elimination theory.
The elimination ideal turns out to be generated by $P_1(c)$ and $P_2(c)$.

The space of parameters $(r,k)$ being projected out is affine rather than projective, so we have to ensure that the variety $\V(P_1,P_2)$ does not contain spurious points projecting down from infinity.
It turns out that none of these possibly problematic stiffness tensors are positive definite.
\end{proof}

The proof hints that~\eqref{eq:char-of-alg-adm} does not hold without requiring positive definiteness.
Physical assumptions therefore simplify the algebra.

To describe stiffness tensors that are not only algebraically but also analytically admissible, we introduce more polynomials in $c$.
The key polynomials are
\begin{align}
\label{eq:Lnonzero}
R(c)&\coloneqq R_1(c) + R_2(c)
\quad\text{and}\\
Q(c)&\coloneqq R_1(c)R_2(c),\label{eq:Qnonzero}
\end{align}
defined using the auxiliary polynomials
\begin{align}
R_1(c)&\coloneqq (c_{33}+c_{12})^2-4c_{13}c_{23}, \label{eq:D1}
\quad\text{and}\\
R_2(c)&\coloneqq (c_{13}-c_{23})^2+(c_{11}-c_{33})(c_{22}-c_{33})\label{eq:D2}
.
\end{align}

\begin{proposition}[Analytically admissible $c$]
\label{prop:an-admissible}
The slowness surface of a positive definite stiffness tensor $c$ is a union of a concentric circle and ellipse that do not meet if and only if $P_1(c)=P_2(c)=0$, $R(c)\geq 0$ and $Q(c)\neq 0$.
That is,
\begin{equation}
\analyticallyadmissibletensors
=
\bigl(
\V(P_1,P_2)
\cap
\{R \geq 0\}
\setminus
\V(Q)
\bigr)\cap \positivetensors
.
\end{equation}
\end{proposition}

\begin{proof}
    Proposition \ref{prop:alg-admissible} characterizes positive definite algebraically admissible stiffness tensors through the vanishing set of the polynomials $P_1$ and $P_2$.
    The circle and ellipse meet if there is a $\xi_1 \in \mathbb{R}^2 \setminus \{0\}$ such that
    \begin{equation}
        \lambda_1(x,\xi_1)=1 \quad \text{and} \quad \lambda_2(x,\xi_1)=1,
    \end{equation}
    where $\lambda_1$ and $\lambda_2$ are the eigenvalues of the Christoffel matrix $\Gamma$. From \cite[Prop. 26]{IKK25} it follows that $\Gamma$ can have a multiple eigenvalue if and only if $R(c)\geq 0$ and $Q(c) \neq 0$. 
\end{proof}

Additionally, we observe the following special properties of the variety $\V(P_1,P_2)$:
\begin{lemma}\label{lem:SingLoc}
The variety $S=\V(P_1,P_2) \subset \mathbb{C}^6$ satisfies the following:
\begin{enumerate}
    \item The singular locus of $S$ is given by the vanishing of the ideal generated by $P_1$, $P_2$, and the $2\times 2$ minors of the Jacobian matrix of $(P_1,P_2)$.
    \item No real positive definite stiffness tensor lies in the singular locus.
    \item The set of isotropic stiffness tensors is contained in $S$.
\end{enumerate}
\end{lemma}

\begin{proof}
The description of the singular locus follows from standard results in algebraic geometry: a point of $S$ is singular if and only if the Jacobian matrix of the defining polynomials drops rank. Since $S$ is defined by two equations, this occurs precisely when all $2\times 2$ minors of the Jacobian matrix $J=\begin{bmatrix}
    \frac{\partial P_1}{\partial c_{11}} & \dots & \frac{\partial P_1}{\partial c_{33}}\\
    \frac{\partial P_2}{\partial c_{11}} & \dots & \frac{\partial P_2}{\partial c_{33}}\\
\end{bmatrix}$ vanish. Thus the singular locus is defined by the ideal generated by $P_1$, $P_2$, and these minors. A minimal generating set for this ideal can be computed using magma (see Listing~\ref{lst:magma3}):
\begin{equation}
    c_{13}+c_{23}, \quad 
    c_{12}+c_{22}, \quad 
    c_{11}-c_{22}.
\end{equation}
A stiffness tensor $c$ satisfying these conditions is singular, proving the second claim. 

Finally, the inclusion of isotropic stiffness tensors is verified by direct substitution. The isotropic conditions
\begin{equation}
c_{13} = 0, \quad c_{23} = 0, \quad c_{22} = c_{11}, \quad 2c_{33} = c_{11} - c_{12}
\end{equation}
define an ideal whose vanishing set is contained in $\V(P_1,P_2)$. This can be checked by verifying that $P_1$ and $P_2$ lie in the ideal generated by these relations; see Listing~\ref{lst:magma3}.
\end{proof}

\subsection{Slowness surfaces that are the union of two circles}
Similarly to characterizing algebraically admissible stiffness tensors in Proposition \ref{prop:alg-admissible} one can characterize stiffness tensors that give rise to a slowness surface that is the union of two concentric circles in Proposition \ref{prop:two-circles}. Here we focus on not necessarily positive definite stiffness tensors, as these play a role in the proof of Proposition \ref{prop:components+psismooth} and Theorem \ref{thm:how-many-stiffness-tensor-fields}. The most well-known examples of stiffness tensors with this property are isotropic stiffness tensors. Throughout this paper we consider the following definition of an isotropic stiffness tensor $c$:
\begin{definition}[Isotropic stiffness tensors]
    A stiffness tensor $c$ is isotropic if and only if $c$ is parameterized by Lamé parameters $\lambda$ and $\mu$:
        \begin{equation}\label{eq:iso_tenor}
            c_{ijk\ell}= \lambda \delta_{ij}\delta_{k\ell}
                +
            \mu(\delta_{ik}\delta_{j\ell}+\delta_{i\ell}\delta_{jk})
        \end{equation}

\end{definition}
It is shown in Lemma \ref{lem:fixed_axes} that the above definition is equivalent to every line being a symmetry axis of $c$. In light of this definition we formulate the following proposition:

\begin{proposition}
\label{prop:two-circles}
A stiffness tensor $c$, not necessarily positive definite, gives rise to a slowness surface that is the union of two concentric circles with radii $r_1$ and $r_2$ if and only if $c$ is parameterized by \begin{align}\label{eq:isoslow1}
    c&= \begin{bmatrix}
        c_{11} & c_{11}-2c_{33} & 0\\
        c_{11}-2c_{33} & c_{11} & 0\\
        0 & 0 & c_{33}
    \end{bmatrix}\\
\intertext{or}
c&= \begin{bmatrix}
        c_{11} & -c_{11} & c_{13}\\
        -c_{11} & c_{11} & -c_{13}\\
        c_{13} & -c_{13} & c_{33}
    \end{bmatrix}.\label{eq:isoslow2}
\end{align}

Additionally, we have the following properties:
\begin{itemize}
    \item The stiffness tensors in~\eqref{eq:isoslow1} and~\eqref{eq:isoslow2} give rise to the same slowness surface (same two circles) if and only if their values of $c_{11}+c_{33}$ and $c_{11}c_{33}-c_{13}^2$ agree
    \item The radii of the two circles coincide if and only if $c$ is parameterized by
    \begin{equation}\label{eq:isocoincide}
        c=\begin{bmatrix}
            c_{11} & -c_{11} & 0\\
            -c_{11} & c_{11} & 0\\
            0 & 0 & c_{11}
        \end{bmatrix}
    \end{equation}
    \item The stiffness tensors of the form~\eqref{eq:isoslow1} are precisely the isotropic
    ones, and no stiffness tensor of the form~\eqref{eq:isoslow2} is positive definite. In
    particular, a positive definite stiffness tensor whose slowness surface is the union of
    two concentric circles is isotropic.
\end{itemize}

\end{proposition}

This proposition has the following direct implications for isotropic stiffness tensors

\begin{remark}\label{rem:Isoparam}
Proposition~\ref{prop:two-circles} implies that the following four stiffness tensors give rise to the same slowness surface (as their values of $c_{11}+c_{33}$ and $c_{11}c_{33}-c_{13}^2$ agree):
\begin{itemize}
    \item The isotropic stiffness tensor in~\eqref{eq:iso_tenor}:
    \begin{equation}\label{eq:isoparam}
        c=\begin{bmatrix}
        \lambda+2\mu & \lambda & 0\\
        \lambda & \lambda+2\mu & 0\\
        0 & 0 & \mu
    \end{bmatrix},
    \end{equation}
    corresponding to~\eqref{eq:isoslow1} with $c_{11}=\lambda+2\mu$ and $c_{33}=\mu$
    \item the non-isotropic stiffness tensor:
    \begin{equation}\label{eq:circlesemi}
        c=\begin{bmatrix}
        \lambda+2\mu & -\lambda-2\mu & 0\\
        -\lambda-2\mu & \lambda+2\mu & 0\\
        0 & 0 & \mu
        \end{bmatrix},
    \end{equation}
    corresponding to~\eqref{eq:isoslow2} with $c_{11}=\lambda+2\mu, c_{33}=\mu$ and $c_{13}=0$
    \item  the isotropic stiffness tensor in~\eqref{eq:iso_tenor}
        \begin{equation}\label{eq:isoparamflipped}
        c=\begin{bmatrix}
        \mu & -2\lambda-3\mu & 0\\
        -2\lambda-3\mu & \mu & 0\\
        0 & 0 & \lambda+2\mu
    \end{bmatrix},
    \end{equation}
    corresponding to~\eqref{eq:isoslow1} with $c_{11}=\mu$ and $c_{33}=\lambda+2\mu$
    \item the non-isotropic stiffness tensor:
    \begin{equation}\label{eq:circlesemi2}
        c=\begin{bmatrix}
        \mu & -\mu & 0\\
        -\mu & \mu & 0\\
        0 & 0 & \lambda+2\mu
    \end{bmatrix},
    \end{equation}
    corresponding to~\eqref{eq:isoslow2} with $c_{11}=\mu, c_{33}=\lambda+2\mu$ and $c_{13}=0$
\end{itemize}
With $\mu,\lambda+\mu>0$ the isotropic stiffness tensor in~\eqref{eq:isoparam} is positive definite, the isotropic stiffness tensor in~\eqref{eq:isoparamflipped} is indefinite and the non-isotropic stiffness tensors in~\eqref{eq:circlesemi} and~\eqref{eq:circlesemi2} are positive semi definite. 
\end{remark}

\begin{proof}[Proof of Proposition~\ref{prop:two-circles}]
We give only a brief proof outline; the structure of the proof is very similar to the proof of Proposition \ref{prop:alg-admissible} and a more detailed version can be found in Appendix \ref{app:long-proofs}. The slowness polynomial is spelled out as a polynomial in $(\xi_1,\xi_2)$. If its vanishing set is to contain two circles, it must be divisible by polynomials vanishing on the two circles. Given degree constraints, it has to be the product of these two polynomials. We can parameterize the circle polynomials by radii $r_1$ and $r_2$, multiply them, and compare with the slowness polynomial. Setting all the coefficients to zero defines polynomial conditions in the variables $(c,r_1,r_2)$. 
The ideal generated by these conditions defines a variety in the product space. We project this variety to the $c$ component by elimination theory. The elimination ideal yields the following necessary and sufficient conditions on $c$:
\begin{equation}
    c_{13}+c_{23}=0,\quad c_{11}-c_{22}=0,\quad c_{12}c_{23}+c_{22}c_{23}=0,\quad c_{12}^2-c_{22}^2+2c_{12}c_{33}+2c_{22}c_{33}=0.
\end{equation}
The last two factor as two irreducible components according to whether $c_{12}+c_{22}=0$ or not. Each of these components exactly give the stiffness tensors of the form~\eqref{eq:isoslow1} or~\eqref{eq:isoslow2}.

Additionally, the Gröbner basis in all variables contains the following relations for the radii $r_1$ and $r_2$:
\begin{equation}\label{eq:radiieq}
    r_1 + r_2 = c_{11}+c_{33}, \qquad r_2^2 - (c_{11}+c_{33})r_2 + (c_{11}c_{33}-c_{13}^2) = 0.
\end{equation}
From Vieta's formula it follows that both $r_2$ and $r_1$ are roots in the above second order polynomial and that they satisfy $r_1 r_2 = c_{11}c_{33}-c_{13}^2$. This implies that tensors on the form~\eqref{eq:isoslow1} and one of the form~\eqref{eq:isoslow2} give rise to the same slowness surface if and only if their values of $c_{11}+c_{33}$ and $c_{11}c_{33}-c_{13}^2$ agree. The two roots $r_1$ and $r_2$ coincide if and only if the discriminant of the second order polynomial in~\eqref{eq:radiieq} vanishes:
\begin{equation}
    (c_{11}+c_{33})^2 - 4(c_{11}c_{33}-c_{13}^2) = (c_{11}-c_{33})^2 + 4c_{13}^2 = 0.
\end{equation}
Over $\mathbb{R}$, this holds if and only if $c_{11}=c_{33}$ and $c_{13}=0$, which yields the parameterization in~\eqref{eq:isocoincide} for which the two circles to coincide.
Finally,~\eqref{eq:isoslow1} is the Voigt form of~\eqref{eq:iso_tenor} which is evident from Remark~\ref{rem:Isoparam}, while the first two Voigt rows of
\eqref{eq:isoslow2} satisfy (row $1$) $=-$(row $2$), so that such a tensor is singular.

\end{proof}

The following Lemma gives an explicit characterization of stiffness tensors that give rise to the same slowness surface as $c$ in~\eqref{eq:isoslow1} in terms of rotations of a stiffness tensor on the form~\eqref{eq:isoslow2}. Rotation actions on a stiffness tensor are defined in~\eqref{eq:rotC}.
\begin{lemma}
\label{lem:rotation-orbit}
Define:
\begin{equation}
c=\begin{bmatrix}
        c_{11} & c_{11}-2c_{33} & 0\\
        c_{11}-2c_{33} & c_{11} & 0\\
        0 & 0 & c_{33}
    \end{bmatrix}
\qquad \text{and} \qquad
    \tilde c \coloneqq \begin{bmatrix} c_{11} & -c_{11} & 0\\ -c_{11} & c_{11} & 0\\ 0 & 0 & c_{33}\end{bmatrix},
\end{equation}
where $c$ is a stiffness tensor of the form~\eqref{eq:isoslow1} and $\tilde{c}$ is a stiffness tensor of the form~\eqref{eq:isoslow2} with $c_{13}=0$. These satisfy the following:
\begin{itemize}
    \item $c$ and $\tilde c$ give rise to the same slowness surface
    \item A stiffness tensor $c'$ of the form~\eqref{eq:isoslow2} gives rise to the same slowness surface as $c$ if and only if $c'=\rotC{\theta}\tilde{c}$ for some $\theta \in [0,2\pi)$.
\end{itemize}
\end{lemma}

\begin{proof}
We give only a brief proof outline; a more detailed proof can be found in Appendix~\ref{app:long-proofs}.
Assign to a stiffness tensor of the form~\eqref{eq:isoslow1} or~\eqref{eq:isoslow2} the
symmetric \emph{block}
\begin{equation}
    B \coloneqq \begin{pmatrix} c_{11} & c_{13}\\ c_{13} & c_{33}\end{pmatrix},
\end{equation}
formed from its $(1,1)$, $(1,3)$ and $(3,3)$ Voigt entries, the off-diagonal entry being the
$(1,3)$-entry $+c_{13}$ and not the $(2,3)$-entry $-c_{13}$. By
Proposition~\ref{prop:two-circles} two such tensors share a slowness surface if and only if
their blocks have the same trace and determinant. Both $c$ and $\tilde c$ have block
$B_0=\operatorname{diag}(c_{11},c_{33})$, which proves the first assertion; write $B'$ for
the block of $c'$.

Specialising~\eqref{eq:rotCexplicit} to~\eqref{eq:isoslow2} with $c_{13}=0$ shows that
$\rotC{\theta}\tilde c$ is again of the form~\eqref{eq:isoslow2}, with
\begin{align}
    (\rotC{\theta}\tilde c)_{11}
    &= \tfrac{1}{2}(c_{11}+c_{33})+\tfrac{1}{2}(c_{11}-c_{33})\cos(4\theta),\\
    (\rotC{\theta}\tilde c)_{33}
    &= \tfrac{1}{2}(c_{11}+c_{33})-\tfrac{1}{2}(c_{11}-c_{33})\cos(4\theta),\\
    (\rotC{\theta}\tilde c)_{13}
    &= \tfrac{1}{2}(c_{11}-c_{33})\sin(4\theta),
\end{align}
that is, with the block
\begin{equation}\label{eq:Bconj1}
    B^{\theta} = \rotA{2\theta}\,B_0\,\rotA{2\theta}^{T}.
\end{equation}
In particular, if $c'=\rotC{\theta}\tilde c$ then $B'=B^{\theta}$ is an orthogonal conjugate
of $B_0$ and hence has the same trace and determinant, so $c'$ and $c$ share a slowness
surface.

Conversely, let $c'$ be of the form~\eqref{eq:isoslow2} and share a slowness surface with
$c$. Then $B_0$ and $B'$ have equal trace and determinant, hence the same characteristic
polynomial $\lambda^{2}-\operatorname{tr}(B_0)\lambda+\det(B_0)$ and the same eigenvalues
$c_{11},c_{33}$. Both are real symmetric, so the spectral theorem gives $B'=Q\,B_0\,Q^{T}$
for some $Q\in O(2)$, after reordering the eigenvalues by a permutation matrix in $O(2)$ if
necessary. If $\det(Q)=-1$, write $Q=\reflA{\psi}=\rotA{2\psi}D$ with
$D=\operatorname{diag}(1,-1)$; as $B_0$ is diagonal, $DB_0D^{T}=B_0$ and therefore
$QB_0Q^{T}=\rotA{2\psi}B_0\rotA{2\psi}^{T}$. Every $O(2)$-conjugate of $B_0$ is thus already
a rotation conjugate, so $B'=\rotA{\phi}B_0\rotA{\phi}^{T}$ for some $\phi$. With
$\theta\coloneqq\phi/2$,~\eqref{eq:Bconj1} shows that $\rotC{\theta}\tilde c$ is of the form
\eqref{eq:isoslow2} with block $B'$, and such a tensor is determined by its block, so
$c'=\rotC{\theta}\tilde c$.
\end{proof}

\subsection{Symmetry axes}

The matrix $\reflA{\theta}\in O(2)$ is the reflection across the line
\begin{equation}
L^\theta\coloneqq
\{
t(\cos(\theta),\sin(\theta))\in \R^2
;
t\in \R
\}.
\end{equation}
The induced reflection of a stiffness tensor is given in~\eqref{eq:reflC}.

\begin{definition}\label{def:symmAx}
    The line $L^\theta$ is a symmetry axis for $c$ if and only if $c=\reflC{\theta}c$.
\end{definition}

We observe that there is the following equivalent definition:

\begin{lemma}
\label{lem:char_of_reflection_invariance}
For $\theta \in\R$, denote
$v=(\cos(\theta),\sin(\theta))$
and
$w=(-\sin(\theta),\cos(\theta))$.
The line $L^\theta$ is a symmetry axis for $c$ if and only if
    \begin{equation}
        \label{eq:symmetry_axis_condition}
    c(v,v,v,w)=c(w,w,v,w)=0.
    \end{equation}
If $L^\theta$ is a symmetry axis, so is $L^{\theta+\pi/2}$.
\end{lemma}

\begin{proof}
We denote $\reflA{\theta}=R$ for simplicity.
First we compute
\begin{equation}
c(\xi_1,\xi_2,\xi_3,\xi_4),
\end{equation}
where each $\xi_j$ is either $v$ or $w$.
Let $q\in \{0,1,2,3,4\}$ be the number of $w$s.
The vectors were chosen so that
$Rv=v$
abd
$Rw=-w$.
Thus by linearity
\begin{equation}
(\reflC{\theta}c)(\xi_1,\xi_2,\xi_3,\xi_4)
=
c(R\xi_1,R\xi_2,R\xi_3,R\xi_4)
=
(-1)^q c(\xi_1,\xi_2,\xi_3,\xi_4).
\end{equation}
As $\{v,w\}$ is a basis of $\R^2$, we have indeed that $c=\reflC{\theta}c$ is equivalent with~\eqref{eq:symmetry_axis_condition}.

As $\reflA{\theta+\pi/2}=-\reflA{\theta}$, we have $\reflC{\theta+\pi/2}=\reflC{\theta}$.
Therefore $L^\theta$ and $L^{\theta+\pi/2}$ are either both or neither a symmetry axis.
\end{proof}

We use this Lemma to observe the following properties of algebraically admissible and isotropic stiffness tensors:
\begin{lemma}\label{lem:fixed_axes}
Let $c$ be a positive definite algebraically admissible stiffness tensor  written in standard Euclidean coordinates of $\R^2$. 
The tensor $c$ is isotropic if and only if every line is a symmetry axis of $c$. If $c$ is non-isotropic then it has exactly two orthogonal symmetry axes. Furthermore, the symmetry axes of $c$ are the coordinate axes if and only if 
\begin{equation}
\label{eq:fixed_axes}
    c_{13}=c_{23}=0.
\end{equation}
\end{lemma}

\begin{proof}
We use Lemma~\ref{lem:char_of_reflection_invariance} to identify symmetry axes.
Define $v$ and $w$ as in that lemma.
We can spell out $c(v,v,w,v)$ and $c(w,w,v,w)$ in terms of $a=\cos(\theta)$ and $b=\sin(\theta)$:
\begin{equation}
\begin{split}
c(v,v,w,v)
&=
a^4 c_{13}
+a^3b(-c_{11}+c_{12}+2c_{33})
+3a^2b^2(c_{23}-c_{13})
\\&\qquad
+ab^3(-c_{12}+c_{22}-2c_{33})
-b^4c_{23}
\eqqcolon Q_1(a,b)
,\\
c(w,w,v,w)
&=
a^4 c_{23}
+a^3b(-c_{12}+c_{22}-2c_{33})
+3a^2b^2(c_{13}-c_{23})
\\&\qquad
+ab^3(-c_{11}+c_{12}+2c_{33})
-b^4c_{13}
\eqqcolon Q_2(a,b)
.
\end{split}
\end{equation}
These $Q_1(a,b)$ and $Q_2(a,b)$ are homogeneous polynomials in $(a,b)$.

All lines are symmetry axes if and only if both polynomials $Q_1$ and $Q_2$ vanish identically.
It is easy to check that the following are equivalent:
\begin{enumerate}
\item
$Q_1\equiv0$.
\item
$Q_2\equiv0$.
\item
$
c_{13}
=
c_{23}
=
c_{11}-c_{22}
=
c_{12}-c_{22}+2c_{33}
=
0
$.
\item
$c$ is isotropic.
\end{enumerate}

Suppose then that $c$ is not isotropic.
Because the two polynomials $Q_1$ and $Q_2$ have degree $4$, there are up to $4$ points on $S\coloneqq\V(Q_1,Q_2)\subset\R \mathbb{P}^1$.
By the last part of Lemma~\ref{lem:char_of_reflection_invariance} $(a,b)\in\V(Q_1,Q_2)$ if and only if $(b,-a)\in\V(Q_1,Q_2)$, so points of $S$ come in pairs.
Thus $\#S\in\{0,2,4\}$.
By construction of the projective space $\R \mathbb{P}^1$, the pairs $(a,b)$ and $(-a,-b)$ are identified (and they also produce the exact same symmetry axis), so the number of points on $\V(Q_1,Q_2)\cap\sphere^1\subset\R^2$ is twice the size of $S$.

The extreme case $\#S=4$ happens if and only if $Q_1=\alpha Q_2$ for some $\alpha\in\R$.
Imposing this condition on the coefficients of the polynomials $Q_1$ and $Q_2$ and the conditions $P_1(c)=P_2(c)=0$ leads to a system of equations.
It is a straightforward computer algebra exercise that the system has three linear families of solutions and that the determinant of $c$ vanishes identically on two of them.
The only positive definite solution is
\begin{equation}
c_{22}=c_{11},
c_{12}=c_{11}-2c_{33},
c_{13}=c_{23}=0,
\end{equation}
which, as noted just above, gives exactly the isotropic stiffness tensors.

Suppose then the other extreme case $\#S=0$.
If the two polynomials have shared roots, they have to be complex.
Due to the real coefficients they come in conjugate pairs.
By the symmetry noted above, there have to be two conjugate pairs.
This is already all four roots, so the polynomials would have to be constant multiples of one another, which we just ruled out.
Therefore, the two polynomials have no shared roots at all, so their resultant is non-zero.
It is easy to verify that the resultant is
\begin{equation}
\begin{split}
&
[
(-c_{11}-c_{22}+4c_{33}+2c_{12})P_1(c)
+
2(c_{23}-c_{13})P_2(c)
]^2
\\&\times
(16c_{33}^2-8c_{22}c_{33}+16c_{12}c_{33}-8c_{11}c_{33}+16c_{23}^2-32c_{13}c_{23}+c_{22}^2
\\&\quad
-4c_{12}c_{22}+2c_{11}c_{22}+16c_{13}^2+4c_{12}^2-4c_{11}c_{12}+c_{11}^2)
,
\end{split}
\end{equation}
so it vanishes under our assumption $P_1(c)=P_2(c)=0$.
Therefore $\#S\neq0$.

We have thus found $\#S=2$.
The last claim follows from Lemma~\ref{lem:char_of_reflection_invariance}.
\end{proof}

\begin{corollary}
\label{cor:symmetry_axis_of_c_and_sigma_agree}
    If $c$ is an algebraically admissible stiffness tensor then its symmetry axes agree with the symmetry axes of its slowness surface $\spoly{c}$. 
\end{corollary}

\begin{proof}
If $c$ is isotropic then all lines are symmetry axes for it and its slowness surface.

If $c$ is not isotropic it has exactly two symmetry axes by Lemma~\ref{lem:fixed_axes}.
The slowness surface consists of a circle and an ellipse, and the ellipse cannot be a
circle: otherwise the slowness surface would be the union of two concentric circles and
$c$ would be isotropic by Proposition~\ref{prop:two-circles}.
Any symmetry of the stiffness tensor is clearly a symmetry of the slowness surface: $\spoly{\reflC{\theta}c}=\reflA{\theta}\spoly{c}$.
The ellipse only has two symmetry axes, so they can only be the two symmetry axes of $c$.
\end{proof}

\begin{remark}
Due to the observation made in Remark~\ref{rem:Isoparam} we note that in Corollary~\ref{cor:symmetry_axis_of_c_and_sigma_agree} the positive definitiveness is crucial.
\end{remark}

In the next lemma, under the assumption that the symmetry axes of the stiffness tensor are the coordinate axes, we present a simplified version for the characterization of the algebraically admissible stiffness tensors.

Define the polynomial
\begin{equation}\label{eq:Deq}
    D\coloneqq (c_{11}-c_{33})(c_{22}-c_{33}).
\end{equation}

We are interested in those $c$ satisfying
\begin{equation}\label{eq:Dcond}
    D=(c_{12}+c_{33})^2,
\end{equation}
corresponding to the equality in~\eqref{eq:ellipticity_in_fixed_axes}.

\begin{lemma}
\label{lma:special-characterization}
A positive definite stiffness tensor $c$ that has the coordinate axes as symmetry axes, that is $c_{13}=c_{23}=0$, is algebraically admissible if and only if~\eqref{eq:Dcond} holds.
\end{lemma}

\begin{proof}
With $c_{13}=c_{23}=0$ we have
\begin{equation}
P_1(c)
=
0,
\qquad
P_2(c)=D-(c_{12}+c_{33})^2
,
\end{equation}
so the claim follows from Proposition~\ref{prop:alg-admissible}.

\end{proof}

\begin{lemma}
\label{lma:special-characterization-analytic}
A positive definite stiffness tensor $c$ that has the coordinate axes as symmetry axes, that is $c_{13}=c_{23}=0$, is analytically admissible if and only if~\eqref{eq:Dcond} holds and $D\neq 0$.
\end{lemma}

\begin{proof}
By Proposition~\ref{prop:an-admissible} $c$ is analytically admissible if and only if $P_1(c)=P_2(c)=0$, $R(c)\geq 0$ and $Q(c)\neq 0$, we call these conditions property A. We show that under the assumption $c_{13}=c_{23}=0$ property A is equivalent to $D\neq0$. We have 
\begin{align}
        R_1(c)&=(c_{12}+c_{33})^2
        \quad \text{and} \quad
        R_2(c)=(c_{11}-c_{33})(c_{22}-c_{33}).
\end{align}
Suppose that property A holds. From Lemma \ref{lma:special-characterization} the equality $(c_{11}-c_{33})(c_{22}-c_{33})=(c_{12}+c_{33})^2$ follows. This yields $R_1(c)=R_2(c)$ and $Q(c)=R_1(c)R_2(c)=R_2(c)^2=D^2(c)$. Since $Q(c)\neq 0$ also $D\neq 0$.

Suppose that $D=(c_{12}+c_{33})^2$ and $D\neq 0$. Lemma \ref{lma:special-characterization} gives $P_1(c)=P_2(c)=0$. Furthermore, $R_2(c)=R_1(c)\geq 0$ yield
    \begin{equation}
        R(c)=R_1(c)+R_2(c)=2R_1(c)\geq 0.
    \end{equation}
as well as
    \begin{equation}
        Q(c)=R_2(c)^2=D^2(c)\neq 0.
    \end{equation}
We have verified property $A$.
\end{proof}

The following Lemma shows that algebraic admissibility as in Proposition \ref{prop:alg-admissible} is a coordinate invariant property:

\begin{lemma}
\label{lem:P1_P2_vanish_is_rotationally_invariant}
Let $c$ be a positive definite stiffness tensor then $\rotC{\theta}c$ is a positive definite stiffness tensor and the polynomials $P_1(\rotC{\theta}c)$ and $P_2(\rotC{\theta}c)$ can be computed from $P_1(c),P_2(c)$ via the invertible linear system
\begin{equation}
\label{eq:transformation_of_P1_P2}
\begin{pmatrix}
2P_1(\rotC{\theta}c)\\
P_2(\rotC{\theta}c)
\end{pmatrix}
=
\begin{pmatrix}
\cos(4\theta) & -\sin(4\theta)\\
\sin(4\theta) & \cos(4\theta)
\end{pmatrix}
\begin{pmatrix}
2P_1(c)\\
P_2(c)
\end{pmatrix}
\end{equation}

\end{lemma}
\begin{proof}
The tensor $\rotC{\theta}c$ inherits the symmetries~\eqref{eq:sym1} and~\eqref{eq:sym2}
from $c$. Let $A$ be a non-zero symmetric $(2\times2)$-matrix and set
$B\coloneqq\rotA{\theta}^{T}A\,\rotA{\theta}$. Grouping the factors in~\eqref{eq:rotC} gives
\begin{align}
\begin{aligned}
\label{eq:c_theta_A_A}
(\rotC{\theta}c)_{mnop}A_{mn}A_{op}
&=
c_{ijk\ell}
\bigl(\rotA{\theta}_{mi}A_{mn}\rotA{\theta}_{nj}\bigr)
\bigl(\rotA{\theta}_{ok}A_{op}\rotA{\theta}_{p\ell}\bigr)\\
&=
c_{ijk\ell}B_{ij}B_{k\ell}.
\end{aligned}
\end{align}
Since $\rotA{\theta}$ is invertible, $B$ is a non-zero symmetric matrix, so the right-hand
side of~\eqref{eq:c_theta_A_A} is positive because $c$ is positive definite. Hence
$\rotC{\theta}c$ is positive definite.

We verify the transformation~\eqref{eq:transformation_of_P1_P2}. Differentiating $\rotC{\theta}c$ in Voigt notation at $\theta=0$ gives
\begin{align}
\tfrac{d}{d\theta}\big|_{\theta=0}(\rotC{\theta}c)_{11}&=-4c_{13},
&
\tfrac{d}{d\theta}\big|_{\theta=0}(\rotC{\theta}c)_{22}&=4c_{23},\\
\tfrac{d}{d\theta}\big|_{\theta=0}(\rotC{\theta}c)_{12}&=2(c_{13}-c_{23}),
&
\tfrac{d}{d\theta}\big|_{\theta=0}(\rotC{\theta}c)_{13}&=c_{11}-c_{12}-2c_{33},\\
\tfrac{d}{d\theta}\big|_{\theta=0}(\rotC{\theta}c)_{23}&=c_{12}-c_{22}+2c_{33},
&
\tfrac{d}{d\theta}\big|_{\theta=0}(\rotC{\theta}c)_{33}&=2(c_{13}-c_{23}).
\end{align}
Substituting these into~\eqref{eq:P1} and~\eqref{eq:P2}, a direct computation yields
\begin{equation}
\tfrac{d}{d\theta}\big|_{\theta=0}P_1(\rotC{\theta}c)=-2P_2(c),
\qquad
\tfrac{d}{d\theta}\big|_{\theta=0}P_2(\rotC{\theta}c)=8P_1(c).
\end{equation}
As $\rotA{s}\rotA{\theta}=\rotA{s+\theta}$, we have $\rotC{s}\rotC{\theta}=\rotC{s+\theta}$,
so applying the two identities above to $\rotC{\theta}c$ in place of $c$ gives them for
every $\theta$. Writing $u(\theta)\coloneqq 2P_1(\rotC{\theta}c)$ and
$v(\theta)\coloneqq P_2(\rotC{\theta}c)$, this is the linear system
\begin{equation}
\frac{d}{d\theta}
\begin{pmatrix}u\\ v\end{pmatrix}
=
\begin{pmatrix}0 & -4\\ 4 & 0\end{pmatrix}
\begin{pmatrix}u\\ v\end{pmatrix},
\end{equation}
whose solution with initial value $(u(0),v(0))=(2P_1(c),P_2(c))$ is
\begin{equation}
\begin{pmatrix}u(\theta)\\ v(\theta)\end{pmatrix}
=
\exp\left(\theta\begin{pmatrix}0 & -4\\ 4 & 0\end{pmatrix}\right)
\begin{pmatrix}u(0)\\ v(0)\end{pmatrix}
=
\begin{pmatrix}
\cos(4\theta) & -\sin(4\theta)\\
\sin(4\theta) & \cos(4\theta)
\end{pmatrix}
\begin{pmatrix}2P_1(c)\\ P_2(c)\end{pmatrix}.
\end{equation}
This is~\eqref{eq:transformation_of_P1_P2}, and the system is invertible since the matrix
is a rotation, of determinant $1$.
\end{proof}

By Lemma~\ref{lem:P1_P2_vanish_is_rotationally_invariant} admissibility is preserved under
rotation. The next lemma shows that an admissible tensor can moreover always be rotated so
that the coordinate axes become symmetry axes.

\begin{lemma}
\label{lem:rotated_tensors}
    Let $c$ be an algebraically admissible stiffness tensor. Then there is
    $\theta\in[0,\pi)$ such that
    \begin{equation}
        (\rotC{\theta}c)_{13}=(\rotC{\theta}c)_{23}=0.
    \end{equation}
    For such a $\theta$ the tensor $\rotC{\theta}c$ is again algebraically admissible, and
    its admissibility reduces to
    \begin{equation}
        D(\rotC{\theta}c)
        =\bigl((\rotC{\theta}c)_{12}+(\rotC{\theta}c)_{33}\bigr)^{2}.
    \end{equation}
    If $c$ is analytically admissible, then in addition $D(\rotC{\theta}c)\neq 0$.
\end{lemma}

\begin{proof}
If $c$ is isotropic there is nothing to prove. Thus, it suffices to assume that $c$ is non-isotropic. Therefore, Lemma~\ref{lem:fixed_axes} yields that $c$ has exactly two symmetry axes. Particularly, there are orthogonal unit vectors $\{v,w\}$ such that $c(v,v,w,v)=c(w,w,v,w)=0$. We choose $\theta \in [0,2\pi)$ such that $\rotA{\theta}$ is the rotation that takes $v$ to $e_1$ and $w=e_2$. We denote $x^\theta=\rotA{\theta} x$ for $x \in \R^2$. 
Since for all $x,y,z,w \in \R^2$ we have that
\begin{equation}
(\rotC{\theta}c)(x^\theta,y^\theta,z^\theta,w^\theta)
=
c(x,y,z,w)
\end{equation}
we get
\begin{equation}
(\rotC{\theta}c) (e_1,e_1,e_2,e_1)
=
c(v,v,w,v)
=0,
\end{equation}
and
\begin{equation}
(\rotC{\theta}c) (e_2,e_2,e_1,e_2)
=
c(w,w,v,w)
=0.
\end{equation}
Hence, $(\rotC{\theta}c)_{13}=(\rotC{\theta}c)_{23}=0$.

Since, $c$ is algebraically admissible and $(\rotC{\theta}c)_{13}=(\rotC{\theta}c)_{23}=0$ we get from Lemma~\ref{lma:special-characterization} that $D(\rotC{\theta}c) =\bigl((\rotC{\theta}c)_{12}+(\rotC{\theta}c)_{33}\bigr)^{2}$. If $c$ is analytically admissible we get from Lemma~\ref{lma:special-characterization-analytic} that in addition $D(\rotC{\theta}c)\neq 0$.
\end{proof}

\section{Characterization of anomalous companions with the coordinate axes as symmetry axes}
\label{sec:CharacSymmetry}
Due to Lemma~\ref{lem:rotated_tensors} we can without loss of generality assume in this section that all positive definite and algebraically admissible stiffness tensors have coordinate axes as their symmetry axes. That is we assume that the equations~\eqref{eq:fixed_axes} and~\eqref{eq:Dcond} are true. In other words, throughout this section we have
\begin{equation}
\label{eq:cparamcoord}
    c=\begin{bmatrix}
        c_{11} & c_{12}& 0\\
        &c_{22} & 0\\
        && c_{33}\\
    \end{bmatrix},
\end{equation}
where
\begin{equation}
    D\coloneqq (c_{11}-c_{33})(c_{22}-c_{33}),
\end{equation}
satisfies
\begin{equation}
    D=(c_{12}+c_{33})^2.
\end{equation}

We observe, that for stiffness tensors of the form~\eqref{eq:cparamcoord} positive definiteness in Lemma~\ref{lem:posdef} is equivalent to:
\begin{lemma}
\label{lem:posdefcond}
A stiffness tensor $c$ of the form~\eqref{eq:cparamcoord} is positive definite if and only if
\begin{equation}
\label{eq:pos_def_cond}
c_{11},
c_{22},
c_{33},
c_{11}c_{22}-c_{12}^2
>0
. 
\end{equation}
\end{lemma}

\begin{proof}
By Sylvester's criterion a $3\times 3$ matrix is positive definite if and only if all the leading minors
$c_{11}$,
$c_{11}c_{22}-c_{12}^2$, and
$c_{33}(c_{11}c_{22}-c_{12}^2)$
of $c$ are positive.
These three conditions are equivalent to~\eqref{eq:pos_def_cond}.
\end{proof}

With this definition of positive definiteness in mind we are able to give analytical admissible stiffness tensors with the coordinate axes as symmetry axis a more concrete interpretaion through the following Lemma and subsequent observations:

\begin{lemma}\label{rem:discond}
    A positive definite stiffness tensor $c$ that satisfies~\eqref{eq:fixed_axes} and~\eqref{eq:Dcond} is analytically admissible if and only if
    \begin{equation}
    \label{eq:major_minor_axis_vs_radius_cond}
        \min(c_{11},c_{22})>c_{33}.
    \end{equation}
\end{lemma}
\begin{proof}
Since $c$ satisfies~\eqref{eq:fixed_axes} and~\eqref{eq:Dcond} it follows from  Lemma \ref{lma:special-characterization-analytic} that $c$ is analytically admissible if and only if $D\neq 0$. If the condition~\eqref{eq:major_minor_axis_vs_radius_cond} holds then equation~\eqref{eq:Deq} yields $D>0$. 

Thus, for the rest of the proof we assume that $D\neq 0$. Due to $D=(c_{12}+c_{33})^2$ it follows $D>0$. Also, the numbers $c_{11}-c_{33}$ and $c_{22}-c_{33}$ have the same sign. If the sign is positive the estimate~\eqref{eq:major_minor_axis_vs_radius_cond} follows immediately. To finish the proof we exclude the case of negative sign by showing that this violates the positive definiteness of $c$. Particularly,  $c_{11}c_{22}-c_{12}^2>0$ cannot hold. 

We denote:
\begin{equation}
a= c_{33}-c_{11}, \quad \text{and} \quad b=c_{33}-c_{22}.
\end{equation}
We use~\eqref{eq:Deq}, and $D=ab$ with the assumptions $a,b>0$ to obtain
\begin{align}
    c_{11}c_{22}-c_{12}^2&=(c_{33}-a)(c_{33}-b)-c_{12}^2
    \\
    &=c_{33}^2+ab-c_{33}(a+b)-c_{12}^2\\
&=c_{33}(2(c_{33}+c_{12})-(a+b))
\\
&\leq c_{33}(2(c_{33}+c_{12})-2\sqrt{ab})
\\
&= 0.
\end{align}
This is a contradiction.
\end{proof}

\begin{lemma}\label{lem:cct}
A positive definite stiffness tensor that satisfies~\eqref{eq:fixed_axes} and~\eqref{eq:Dcond} is given by one of the following forms:  
    \begin{equation}\label{eq:cparam}
    c_+=\begin{bmatrix}
        c_{11} & -c_{33}+\sqrt{D} & 0\\
        &c_{22} & 0\\
        && c_{33}\\
    \end{bmatrix}
\end{equation}
or
\begin{equation}\label{eq:sibling}
    c_-=\begin{bmatrix}
        c_{11} & -c_{33}-\sqrt{D} & 0\\
        &c_{22} & 0\\
        && c_{33}\\
    \end{bmatrix},
\end{equation}
where $D$ is as in~\eqref{eq:Deq}.

Furthermore, we have the following properties:
\begin{itemize}
    \item The maps $\psi_+:c_+ \mapsto (g_p,g_s) $ and $\psi_-:c_- \mapsto (g_p,g_s)$ in Proposition~\ref{prop:components+psismooth} are defined explicitly as follows:
    \begin{equation}\label{eq:psidef}
    \psi_\pm(c_\pm)
    =(g_p,g_s)=
    \left(\begin{bmatrix}
        c_{11}^{-1} & 0\\
        0 & c_{22}^{-1}
    \end{bmatrix}, c_{33}^{-1}\begin{bmatrix}
        1 & 0\\
        0 & 1
    \end{bmatrix}\right).
    \end{equation}
        \item The tensors $c_+$ and $c_-$ share the same slowness surface and are thus
    anomalous companions. Their polarization vectors are related by the reflection
    $\reflA{0}$ across the $x$ axis, equivalently by $\reflA{\pi/2}$ across the $y$ axis.
    \item The determinants of the Voigt matrices satisfy
    \begin{equation}\label{eq:detcpm}
        \det(c_\pm)=c_{33}^{2}\bigl(c_{11}+c_{22}-2c_{33}\pm2\sqrt{D}\bigr),
    \end{equation}
    so that $\det(c_+)>\det(c_-)$ whenever $D\neq0$.
\end{itemize}
\end{lemma}

\begin{proof}
The representations~\eqref{eq:cparam} and~\eqref{eq:sibling} follow directly from~\eqref{eq:fixed_axes} and~\eqref{eq:Dcond}. Note that $D$ is non-negative such that both $c_+$ and $c_-$ give rise to real stiffness tensors.
By a direct computation the associated Christoffel matrices are given by
\begin{align}
    \Gamma_{c_{\pm}}(\xi_1,\xi_2)&=\begin{bmatrix}
        c_{11}\xi_1^2+c_{33}\xi_2^2 & \pm \sqrt{D}\xi_1\xi_2\\
        \pm \sqrt{D}\xi_1 \xi_2 & c_{33}\xi_1^2+c_{22}\xi_2^2
    \end{bmatrix}.
\end{align}
The eigenvalues of the Christoffel matrix of $c_+$ and $c_-$ are identical:
\begin{equation}
    \label{eq:eigenvalues}
    \lambda_{c_\pm,1}(\xi_1,\xi_2)=c_{11}\xi_1^2+c_{22}\xi_2^2,
    \quad \text{and} \quad 
    \lambda_{c_\pm,2}(\xi_1,\xi_2)=c_{33}(\xi_1^2+\xi_2^2).
\end{equation}
corresponding to the metrics $g_s=c_{33}^{-1}\begin{bmatrix}
        1 & 0\\
        0 & 1
\end{bmatrix}$ and $g_p=\begin{bmatrix}
        c_{11}^{-1} & 0\\
        0 & c_{22}^{-1}
\end{bmatrix}$. These are positive definite following from Lemma~\ref{lem:posdefcond} as $c$ is positive definite.
$c_+$ and $c_-$ share the same slowness surface:
\begin{equation}
\Sigma_{c_\pm}:=\{\xi\in \R^2: \lambda_{c_\pm,1}(\xi_1,\xi_2)=1, \quad \text{or} \quad  \lambda_{c_\pm,2}(\xi_1,\xi_2)=1\}
\end{equation}
and are therefore by definition anomalous companions. Furthermore, because
\begin{equation}
\Gamma_{c_{-}}(\xi_1,\xi_2)
=
\reflA{0}\,\Gamma_{c_{+}}(\xi_1,\xi_2)\,\reflA{0}^{-1},
\end{equation}
the two Christoffel matrices have the same eigenvalues, while $v$ is an eigenvector of
$\Gamma_{c_{+}}(\xi)$ if and only if $\reflA{0}v$ is one of $\Gamma_{c_{-}}(\xi)$ for the
same eigenvalue. Since $\reflA{\pi/2}=-\reflA{0}$, the identity holds verbatim with
$\reflA{\pi/2}$ in place of $\reflA{0}$, and the two reflections act identically on
polarization directions.

For~\eqref{eq:detcpm} we use that the Voigt matrix of $c_\pm$ is block diagonal, so
$\det(c_\pm)=c_{33}\bigl(c_{11}c_{22}-c_{12}^{2}\bigr)$. Writing
$a\coloneqq c_{11}-c_{33}$ and $b\coloneqq c_{22}-c_{33}$, so that $D=ab$ and
$c_{12}=-c_{33}\pm\sqrt{ab}$, a direct computation gives
$c_{11}c_{22}-c_{12}^{2}=c_{33}\bigl(a+b\pm2\sqrt{ab}\bigr)$, which is
\eqref{eq:detcpm}. The two values differ by $4c_{33}^{2}\sqrt{D}$, which is positive
when $D\neq0$ since $c_{33}>0$ by Lemma~\ref{lem:posdefcond}.
\end{proof}

With this explicit definition of the map $\psi$ we make the following sanity check of analytic admissibility as in Definition~\ref{def:admissible}:
\begin{remark}
    If $c$ is analytically admissible the condition~\eqref{eq:major_minor_axis_vs_radius_cond} is equivalent to that the metrics $g_s$ and $g_p$ in~\eqref{eq:psidef} corresponding to circle and ellipse in the slowness surface do not meet.
\end{remark}

Additionally, we observe the following for the symmetry axes of $c_+$ and $c_-$:
\begin{remark}
By Lemma~\ref{lem:fixed_axes} the coordinate axes $L^0$ and $L^{\pi/2}$ are the symmetry
axes of both $c_+$ and $c_-$, and by
Corollary~\ref{cor:symmetry_axis_of_c_and_sigma_agree} they are also the symmetry axes of
the slowness surface that the two tensors share. The polarizations of $c_+$ and $c_-$ are
related by a reflection in these very axes.
\end{remark}

In the next lemma we show that all non-isotropic analytically admissible stiffness tensors have exactly one positive definite anomalous companion. Meanwhile, isotropic tensors have no positive definite anomalous companions. 

\begin{lemma}[Anomalous companions with fixed symmetry axes] \label{lem:anIsoAC}
Let $c$ be an analytically admissible stiffness tensor that satisfies~\eqref{eq:fixed_axes}, \eqref{eq:Dcond} and $D\neq 0$.
\begin{enumerate}
    \item If $c$ is not isotropic it has a unique positive definite anomalous companion that satisfies~\eqref{eq:fixed_axes}, \eqref{eq:Dcond} and $D\neq 0$.
    \item If $c$ is isotropic it has no positive definite anomalous companions. However, it has two unique non-isotropic positive semi-definite and a unique isotropic but indefinite anomalous companion that satisfy~\eqref{eq:fixed_axes}, \eqref{eq:Dcond} and $D\neq 0$.
\end{enumerate}

\end{lemma}

\begin{proof}
By Lemma~\ref{lem:cct} $c$ can be parameterized by $c_+$ \eqref{eq:cparam} or else $c_-$ \eqref{eq:sibling}. Furthermore, $c_+$ and $c_-$ are anomalous companions. Suppose that a positive definite tensor $b$ that satisfies~\eqref{eq:fixed_axes} and~\eqref{eq:Dcond} with $D\neq 0$ is an anomalous companion of $c$. 
By~\eqref{eq:eigenvalues} the eigenvalues of the Christoffel matrix corresponding to $c\in \{c_+,c_-\}$ are 
\begin{equation}
    \lambda_{q,1}(\xi_1,\xi_2)=q_{11}\xi_1^2+q_{22}\xi_2^2,
    \quad \text{and} \quad 
    \lambda_{q,2}(\xi_1,\xi_2)=q_{33}(\xi_1^2+\xi_2^2).
\end{equation}
Since $\spoly{c}=\spoly{b}$ we see immediately from the eigenvalues that $c_{ii}=b_{ii}$ for all $i\in\{1,2,3\}$. Therefore, $b\in \{c_+,c_-\}$, and we have shown that the only possible positive definite anomalous companion of $c$ is $c_{\pm}$. For the rest of the proof we let $c= c_{\pm}$ and $b=c_\mp$.

\medskip

Proof of (1): As the non-isotropic $c$ is positive definite Lemma~\ref{lem:posdefcond} gives $c_{11}, c_{22}, c_{33}>0$. Since $b_{ii}=c_{ii}$ for all $i\in\{1,2,3\}$ it follows that $b_{ii}>0$, and in light of Lemma~\ref{lem:posdefcond} it remains to show that $b_{11}b_{22}-(b_{12})^2>0$. 

Depending on whether $b$ is parameterized by~\eqref{eq:cparam} or~\eqref{eq:sibling} the proof boils down to establishing
\begin{align}
    b_{11}b_{22}-(-b_{33}+\sqrt{(b_{11}-b_{33})(b_{22}-b_{33})})^2&>0 \label{eq:param1}
    \\
    \intertext{or}
    b_{11}b_{22}-(-b_{33}-\sqrt{(b_{11}-b_{33})(b_{22}-b_{33})})^2&>0. \label{eq:param2}
\end{align}

First we suppose that $b$ is parameterized by~\eqref{eq:cparam} and show that the condition~\eqref{eq:param1} holds true. Since  $D$ is non-negative
the terms $c_{11}-c_{33}$ and $c_{22}-c_{33}$ have the same sign. If all $c_{ii}$ would coincide then $D=0$ and $c$ would be of the form:
\begin{equation}
        c=\begin{bmatrix}
        c_{22} & -c_{22} & 0\\
        -c_{22}&c_{22} & 0\\
        0& 0& c_{22}\\
    \end{bmatrix},
\end{equation}
corresponding to the stiffness tensor in~\eqref{eq:isocoincide}, where the slowness surface is the union of two circles that coincide and contradicting positive definiteness of $c$. Hence, $c_{11}, \:c_{22}$ and $c_{33}$ cannot all agree and it holds that $(\sqrt{c_{11}-c_{33}}+\sqrt{c_{22}-c_{33}})^2>0$. Thus, the condition~\eqref{eq:param1} follows by a direct computation with the fact $b_{ii}=c_{ii}>0$.  

If $b$ is given by~\eqref{eq:sibling} we verify that the condition~\eqref{eq:param2} holds true. If the expression
\begin{equation}
\label{eq:pos_condition_2}
(\sqrt{c_{11}-c_{33}}-\sqrt{c_{22}-c_{33}})^2
\end{equation}
vanishes then $c_{11}=c_{22}$, and in the current case $c$ would correspond to the parameterization~\eqref{eq:cparam} and be of the form:
\begin{equation}
        c=\begin{bmatrix}
        c_{22} & c_{22}-2c_{33} & 0\\
        c_{22}-2c_{33} &c_{22} & 0\\
        0 & 0 & c_{33}\\
    \end{bmatrix}.
\end{equation}
This is exactly the parameterization of an isotropic stiffness tensor in~\eqref{eq:isoslow1} that can be parameterized by the Lamé parameters in~\eqref{eq:isoparam} or~\eqref{eq:isoparamflipped} respectively. This contradicts the assumption that $c$ is not isotropic.
Thus, the expression~\eqref{eq:pos_condition_2} is strictly positive and~\eqref{eq:param2} follows by a direct computation with the fact $b_{ii}=c_{ii}>0$.

\medskip
Proof of (2): Since $c$ is isotropic and positive definite it follows from Remark~\ref{rem:Isoparam} that it can be parameterized by 
\begin{equation}\label{eq:isoreference}
        c=\begin{bmatrix}
        \lambda + 2\mu & \lambda & 0\\
        \lambda & \lambda + 2 \mu & 0\\
        0 & 0 & \mu
    \end{bmatrix}
\end{equation}
with $\mu,\lambda+\mu>0$. Since $c_{12}=-c_{33}\pm \sqrt{D}$ and $D= (\lambda+\mu)^2$ this parameterization corresponds to $c_+$ as in~\eqref{eq:cparam}. Therefore, the parameterization $c_-$ \eqref{eq:sibling} looks like
\begin{equation}\label{eq:isoreferencesemi}
    b=\begin{bmatrix}
        \lambda + 2\mu & -\lambda-2\mu & 0\\
        -\lambda-2\mu & \lambda + 2 \mu & 0\\
        0 & 0 & \mu
    \end{bmatrix}.
\end{equation}
This is the non-isotropic and positive semi definite stiffness tensor as stated in~\eqref{eq:circlesemi}.

Remark \ref{rem:Isoparam} implies that the following isotropic stiffness tensor through a linear transformation of the Lamé parameters gives rise to the same slowness surface as $c$ in~\eqref{eq:isoreference}:
\begin{equation}\label{eq:isoreference2}
c_2=\begin{bmatrix}
\mu & -2\lambda-3\mu & 0\\
-2\lambda-3\mu & \mu & 0\\
0 & 0 & \lambda+2\mu
\end{bmatrix}.
\end{equation}
Since $c_{12}=-c_{33}\pm \sqrt{D}$ and $D= (-\lambda-\mu)^2$ this parameterization corresponds to $c_-$ \eqref{eq:cparam}. Therefore, the parameterization $c_+$ \eqref{eq:sibling} looks like
\begin{equation}\label{eq:isoreferencesemi2}
    b_2=\begin{bmatrix}
        \mu & -\mu & 0\\
        -\mu & \mu & 0\\
        0 & 0 & \lambda + 2\mu
    \end{bmatrix}.
\end{equation}
This is the non-isotropic stiffness tensor in~\eqref{eq:circlesemi2} that is positive semi definite. By Lemma \ref{lem:rotation-orbit} it follows that $b$ and $b_2$ are the only stiffness tensors in the fixed symmetry axis case that share the same slowness surface with the isotropic stiffness tensors $c$ and $c_2$.  Hence, $c, b, c_2$ and $b_2$ are the only stiffness tensors that satisfy~\eqref{eq:fixed_axes} and~\eqref{eq:Dcond} and whose slowness surface consists of two origocentric circles with radii $\frac{1}{\sqrt{\mu}}$ and $\frac{1}{\sqrt{\lambda+2\mu}}$ respectively.
\end{proof}

\begin{example}[Illustration of Lemma~\ref{lem:cct}]
    The following tensor is an example of a positive definite analytically admissible stiffness tensor:
    \begin{equation}\label{ex:c}
    c_+=\begin{bmatrix}
        7 & -2+\sqrt{10} & 0\\
        -2+\sqrt{10} & 4 & 0\\
        0 & 0 & 2
    \end{bmatrix},
\end{equation}
whose anomalous companion is given by
\begin{equation}\label{ex:ct}
    c_-=\begin{bmatrix}
        7 & -2-\sqrt{10} & 0\\
        -2-\sqrt{10} & 4 & 0\\
        0 & 0 & 2
        \end{bmatrix}.
\end{equation}
The slowness surfaces (same) as well as polarization vectors (reflections across $x$- and $y$-axis respectively) of $c_\pm$
are illustrated in Figure \ref{fig:Polarizations2}. 

\begin{figure}[h]
    \centering
    \begin{minipage}[t]{0.5\textwidth}
        \centering
        \includegraphics[width=0.8\textwidth]{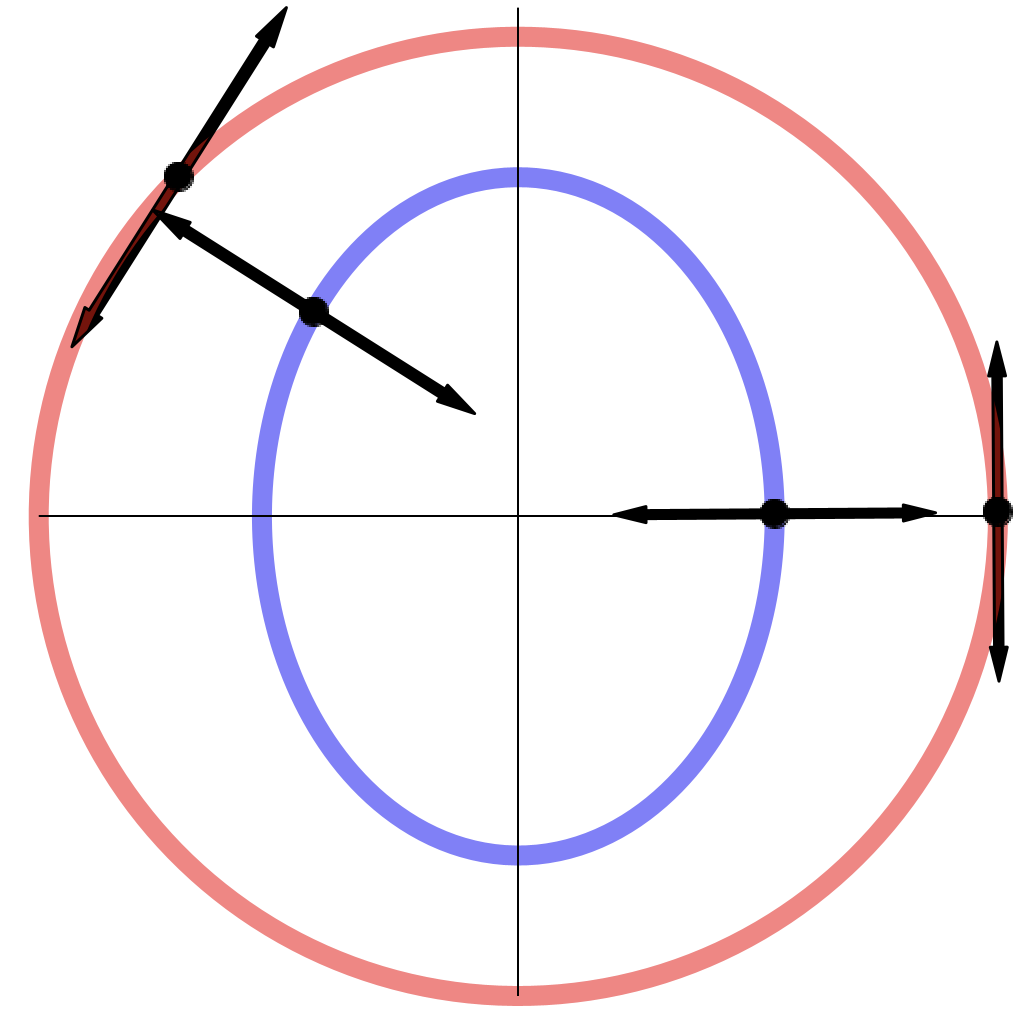}
    \end{minipage}%
    \begin{minipage}[t]{0.5\textwidth}
        \centering
        \includegraphics[width=0.75\textwidth]{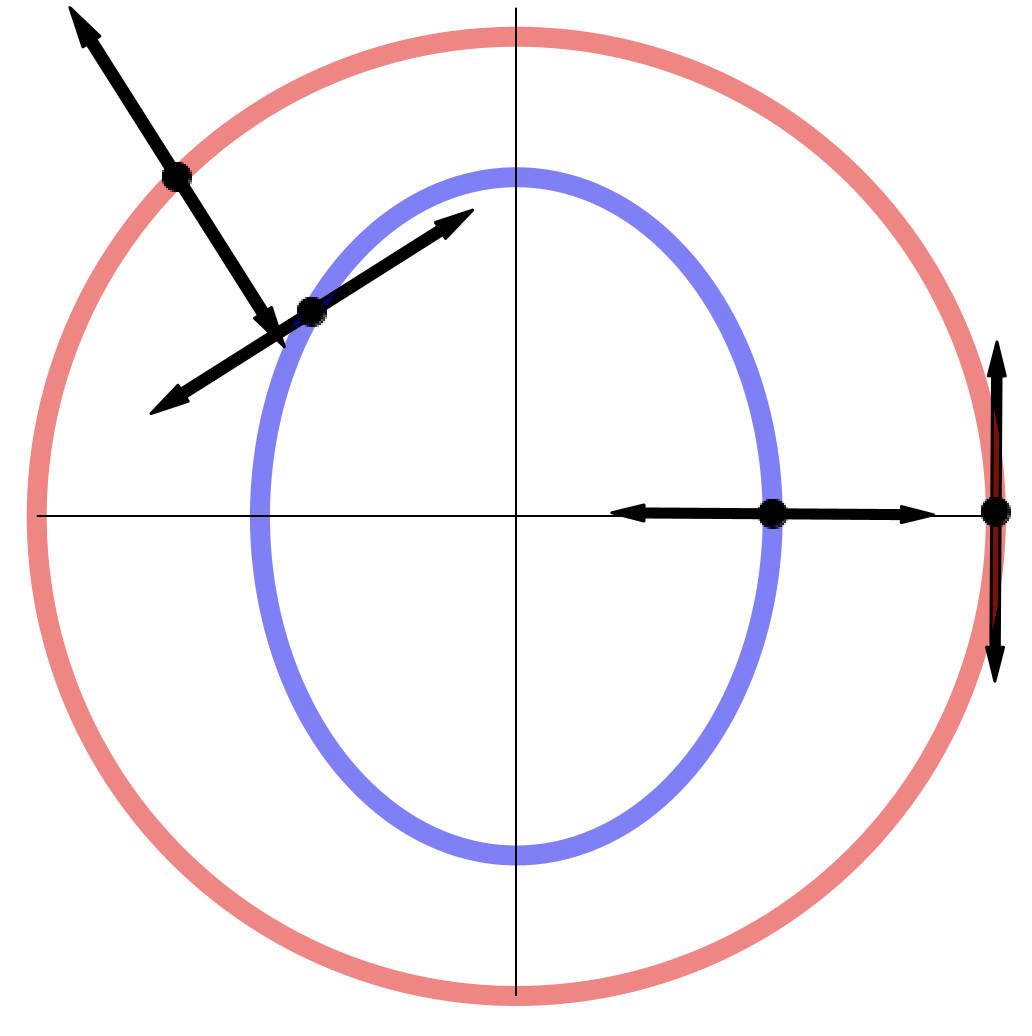}
    \end{minipage}
    \caption{The slowness surface corresponding to $c_+$ and $c_-$ in \eqref{ex:c}--\eqref{ex:ct} with two of their polarization vectors.}
    \label{fig:Polarizations2}
\end{figure}
\end{example}

\section{Anomalous companions and the two branches of $\psi$}\label{sec:psibranches}

Throughout this section $\analyticallyadmissibletensors$ denotes the set of analytically
admissible stiffness tensors and $\admissiblepairsofmetrics$ the set of pairs $(g_p,g_s)$
of Riemannian metrics for which $g_s$ is a positive multiple of the Euclidean one and
\begin{equation}\label{eq:setE}
    \abs{v}_{g_p}<\abs{v}_{g_s}
    \qquad\text{for all } v\in\R^2\setminus\{0\},
\end{equation}
that is, the $p$-wave speed exceeds the $s$-wave speed in every direction. The map
$\psi\colon\analyticallyadmissibletensors\to\admissiblepairsofmetrics$ sends a stiffness
tensor to its pair of wave speed metrics, $\psi(c)=(g_p,g_s)$.

In this section we generalize Lemma~\ref{lem:anIsoAC} to the setting without fixed symmetry axes and use the results from the previous two sections and this Lemma to prove surjectivity of the map $\psi$ (Proposition~\ref{prop:components+psismooth}(1)), that $\psi$ is 1-1 on the set of isotropic stiffness tensors and 2-1 outside it (Proposition~\ref{prop:components+psismooth}(5)) and that the set $\psi^{-1}(\admissiblepairsofmetrics \backslash \admissiblepairsofmetrics_-)\subset \analyticallyadmissibletensors$ coincides with the set of isotropic stiffness tensors (Proposition~\ref{prop:components+psismooth}(4)).

The following Lemma addresses how the slowness surface is affected by rotations of the stiffness tensor and which consequences this has for anomalous companions, analytic admissibility and the map $\psi$:
\begin{lemma}\label{lem:psi-equivariant}
    Let $c$ be an algebraically admissible stiffness tensor and $\theta\in[0,2\pi)$. Then
    \begin{equation}\label{eq:rotated_slowness_surface}
        \spoly{\rotC{\theta}c}(\xi)=\spoly{c}\bigl(\rotA{\theta}^{T}\xi\bigr)
        \qquad\text{for all }\xi\in\R^2 .
    \end{equation}
    Consequently:
    \begin{enumerate}
        \item two stiffness tensors are anomalous companions if and only if their rotations
              by $\theta$ are;
        \item $\rotC{\theta}c$ is algebraically admissible, and it is analytically
              admissible if and only if $c$ is;
        \item if $c\in\analyticallyadmissibletensors$, then
              $\psi(\rotC{\theta}c)=\rotB{\theta}\psi(c)$.
    \end{enumerate}
\end{lemma}
\begin{proof}
    For every $\xi\in\R^2$, writing $\eta\coloneqq\rotA{\theta}^{T}\xi$,
    \begin{align}
        \Gamma_{\rotC{\theta}c}(\xi)_{mp}
        &=(\rotC{\theta}c)_{mnop}\,\xi_n\xi_o\\
        &=c_{ijk\ell}\,\rotA{\theta}_{mi}\rotA{\theta}_{p\ell}
         \bigl(\rotA{\theta}_{nj}\xi_n\bigr)\bigl(\rotA{\theta}_{ok}\xi_o\bigr)\\
        &=\rotA{\theta}_{mi}\,\Gamma_{c}(\eta)_{i\ell}\,\rotA{\theta}_{p\ell},
    \end{align}
    that is,
    $\Gamma_{\rotC{\theta}c}(\xi)=\rotA{\theta}\,\Gamma_{c}(\rotA{\theta}^{T}\xi)\,\rotA{\theta}^{T}$.
    Since $\rotA{\theta}^T=\rotA{\theta}^{-1}$ it follows that $\Gamma_{\rotC{\theta}c}(\xi)$ and $\Gamma_{c}(\rotA{\theta}^{T}\xi)$ have the same characteristic polynomial and thus the same
    eigenvalues. Writing these for $c$ as $\eta^T g_p^{-1}\eta$ and
    $\eta^T g_s^{-1}\eta$, those for $\rotC{\theta}c$ are
    $\xi^T\bigl(\rotB{\theta}g_p^{-1}\bigr)\xi$ and
    $\xi^T\bigl(\rotB{\theta}g_s^{-1}\bigr)\xi$. As
    $\rotB{\theta}(M^{-1})=(\rotB{\theta}M)^{-1}$ for symmetric positive definite matrices $M$ this implies the following for the slowness polynomials:
    \begin{align}
    \spoly{\rotC{\theta}c}(\xi)&=(\eta^T g_p^{-1}\eta-1)(\eta^T g_s^{-1}\eta-1)\\
    &=(\xi^T\bigl(\rotB{\theta}g_p\bigr)^{-1}\xi-1)(\xi^T\bigl(\rotB{\theta}g_s\bigr)^{-1}\xi-1)\\
    &=\spoly{c}\bigl(\rotA{\theta}^{T}\xi\bigr)
    \end{align}
    for all $\xi \in \R^2$, which is~\eqref{eq:rotated_slowness_surface}.

    We now turn to the three consequences. Since the substitution
    $\xi\mapsto\rotA{\theta}^{T}\xi$ in~\eqref{eq:rotated_slowness_surface} is the same for
    every stiffness tensor, two tensors have the same slowness polynomial if and only if
    their rotations by $\theta$ do, which is (1). For (2) we note that the slowness surface
    of $\rotC{\theta}c$ is the rotation of the one of $c$, so it is a union of a concentric
    circle and ellipse, and these do not meet for $\rotC{\theta}c$ if and only if they do
    not meet for $c$; hence $\rotC{\theta}c$ is algebraically admissible, and analytically
    admissible precisely when $c$ is. Finally, if $c\in\analyticallyadmissibletensors$ and
    $\psi(c)=(g_p,g_s)$, then it follows from the above that
    \begin{equation}
        \psi(\rotC{\theta}c)=(\rotB{\theta}g_p,\rotB{\theta}g_s)=\rotB{\theta}\psi(c),
    \end{equation}
    which is (3).
\end{proof}

Using the Lemma, the following Corollary follows straightforwardly after rotations of stiffness tensors relative to stiffness tensors in the fixed coordinate axis setting as in Lemma~\ref{lem:anIsoAC}:
\begin{corollary}[Anomalous companions without fixed symmetry axes] \label{cor:posdef}
Let $c$ be an analytically admissible stiffness tensor. 

\begin{enumerate}
    \item If $c$ is not isotropic it has a unique  analytically admissible anomalous companion $b$.
    \item If $c$ is isotropic it has infinitely many anomalous companions. They satisfy $P_1(c)=P_2(c)=0$, $R(c)\geq 0$ and $Q(c)\neq 0$ for $P_1$, $P_2$, $R$ and $Q$ defined in~\eqref{eq:P1}, \eqref{eq:P2}, \eqref{eq:Lnonzero} and~\eqref{eq:Qnonzero}, but none of them is positive definite.
\end{enumerate}
\end{corollary}
\begin{proof}

\textit{Proof of (1):}
We take $\theta \in [0,2\pi)$ such that $\rotC{\theta}c$ satisfies~\eqref{eq:fixed_axes},~\eqref{eq:Dcond} with $D\neq 0$. Thus, it follows from Lemma~\ref{lem:anIsoAC} that $\rotC{\theta}c$ has exactly one anomalous companion $b$ for which~\eqref{eq:fixed_axes}, \eqref{eq:Dcond} and $D\neq 0$ are true. Hence, it follows from Lemma~\ref{lem:psi-equivariant}(1) that $\rotC{-\theta}b$ is an anomalous companion of $c$. We have shown that $c$ has at least one anomalous companion. 

If $b'$ is an algebraically admissible anomalous companion of $c$ then it follows from Lemma~\ref{lem:psi-equivariant}(1) that $\rotC{\theta}b'$ is an anomalous companion of $\rotC{\theta}c$. Therefore, we get from Corollary~\ref{cor:symmetry_axis_of_c_and_sigma_agree} that $\rotC{\theta}b'$ and $\rotC{\theta}c$ have the same symmetry axes, so that $\rotC{\theta}b'$ satisfies~\eqref{eq:fixed_axes}, and hence~\eqref{eq:Dcond} with $D\neq 0$ follows from Lemma~\ref{lma:special-characterization-analytic}. Finally, it follows from Lemma~\ref{lem:anIsoAC} that $\rotC{\theta}b'=b$. Thus, we have proven that $c$ has exactly one anomalous companion.

\medskip

\textit{Proof of (2):}
We note that $c$ and $c_2$ in~\eqref{eq:isoreference} and~\eqref{eq:isoreference2} are invariant under rotations, whereas $b$ and $b_2$ in~\eqref{eq:isoreferencesemi} and~\eqref{eq:isoreferencesemi2} are not. By Lemma~\ref{lem:rotation-orbit}, the algebraically admissible stiffness tensors giving rise to the same slowness surface as $c$ (equivalently $c_2$) are exactly the rotations of $b$, $\rotC{\theta}b$. By the proof of the same Lemma for each $\theta\in [0,2\pi)$ $\rotC{\theta}b$ is a stiffness tensor of the form~\eqref{eq:isoslow2} and is analytically admissible by Lemma~\ref{lem:psi-equivariant}(2).
$\rotC{\theta}b$ has the following entries:

\begin{align}
    (\rotC{\theta}b)_{11}
    &= \tfrac{1}{2}(\lambda+\mu)\cos(4\theta)+\tfrac{1}{2}(\lambda+3\mu),\\
    (\rotC{\theta}b)_{33}
    &= -\tfrac{1}{2}(\lambda+\mu)\cos(4\theta)+\tfrac{1}{2}(\lambda+3\mu),\\
    (\rotC{\theta}b)_{13}
    &= \tfrac{1}{2}(\lambda+\mu)\sin(4\theta),
\end{align}

The rotations of $b_2$ yield no further tensors: since $b$ and $b_2$ share the block $\begin{psmallmatrix}\lambda+2\mu & 0\\ 0 & \mu\end{psmallmatrix}$ up to swapping its diagonal entries—itself a rotation by $\tfrac{\pi}{4}$—one has $\rotC{\theta}b_2=\rotC{\theta+\pi/4}b$, so the two orbits coincide.

Since $c$ is positive definite we have $\mu>0$ and $\lambda+\mu>0$, so $(\rotC{\theta}b)_{13}=\tfrac12(\lambda+\mu)\sin(4\theta)$ is non-constant in $\theta$; hence $\{\rotC{\theta}b:\theta\in[0,2\pi)\}$ is an infinite set, and together with $c_2$ these are the anomalous companions of $c$ that satisfy $P_1(c)=P_2(c)=0$, $R(c)\geq 0$ and $Q(c)\neq 0$.

Finally, none of these companions is positive definite. Each $\rotC{\theta}b$ is of the form~\eqref{eq:isoslow2} and therefore not positive
definite by Proposition~\ref{prop:two-circles}. The remaining companion $c_2$ is indefinite, hence likewise not positive definite. This proves the claim.

\end{proof}

Additionally, we make the following observation for isotropic stiffness tensors that is used in the proof of Theorem~\ref{thm:how-many-stiffness-tensor-fields}:
\begin{remark}\label{rem:isotropic-conformal}
    A tensor $c\in\analyticallyadmissibletensors$ is isotropic if and only if $g_p$ is a
    multiple of the identity, where $\psi(c)=(g_p,g_s)$.
\end{remark}

\begin{proof}
    If $c$ is isotropic, then $c_{11}=c_{22}$ and $g_p$ is a multiple of the identity by~\eqref{eq:psidef}.

    Conversely, assume that $g_p$ is a multiple of the identity. Since $g_s$ is a multiple of the identity as
    well, both sheets of the slowness surface of $c$ are circles centred at the origin, so
    $c$ is isotropic by Proposition~\ref{prop:two-circles}.
\end{proof}

With this in mind we are ready to prove that the map $\psi$ is onto:
\begin{proof}[Proof of Proposition~\ref{prop:components+psismooth}\,(1)]
    Let $(g_p,g_s)\in\admissiblepairsofmetrics$. As $g_s$ is conformal we have
    $g_s=\sigma I$ for some $\sigma>0$, and by the spectral theorem there is $\theta$ such
    that $\rotB{-\theta}g_p$ is diagonal. We write
    $\rotB{-\theta}g_p=\operatorname{diag}(\mu_1,\mu_2)$ and note that
    $\rotB{-\theta}g_s=g_s$. Since conjugation by an orthogonal matrix preserves the
    ordering of symmetric matrices, the admissibility condition $g_p<g_s$ implies
    $0<\mu_1,\mu_2<\sigma$.

    We now consider the stiffness tensor $c$ of the form~\eqref{eq:cparam} with parameters
    \begin{equation}
        c_{11}\coloneqq\mu_1^{-1},
        \qquad
        c_{22}\coloneqq\mu_2^{-1},
        \qquad
        c_{33}\coloneqq\sigma^{-1} .
    \end{equation}
    By construction $c$ satisfies~\eqref{eq:fixed_axes} and~\eqref{eq:Dcond}, and
    $\mu_1,\mu_2<\sigma$ gives $\min(c_{11},c_{22})>c_{33}$, so that
    $D=(c_{11}-c_{33})(c_{22}-c_{33})>0$.

    We first check that $c$ is positive definite. Its Voigt matrix is block diagonal, so
    this amounts to $c_{11}>0$, $c_{33}>0$ and $c_{11}c_{22}-c_{12}^{2}>0$. Writing
    $a\coloneqq c_{11}-c_{33}>0$ and $b\coloneqq c_{22}-c_{33}>0$, we have $D=ab$ and
    $c_{12}=-c_{33}+\sqrt{ab}$, and thus
    \begin{equation}
        c_{11}c_{22}-c_{12}^{2}
        =c_{33}(a+b)+2c_{33}\sqrt{ab}
        =c_{33}\bigl(\sqrt{a}+\sqrt{b}\bigr)^{2}
        >0 .
    \end{equation}
    Hence $c$ is positive definite, and it follows by Lemma~\ref{rem:discond} and
   ~\eqref{eq:major_minor_axis_vs_radius_cond} that $c$ is analytically admissible.

    By~\eqref{eq:psidef} we have
    \begin{equation}
        \psi(c)
        =\left(\operatorname{diag}(c_{11}^{-1},c_{22}^{-1}),\,c_{33}^{-1}I\right)
        =\left(\operatorname{diag}(\mu_1,\mu_2),\,\sigma I\right)
        =\rotB{-\theta}(g_p,g_s),
    \end{equation}
    and thus by Lemma~\ref{lem:psi-equivariant} the tensor $\rotC{\theta}c$ is analytically
    admissible and satisfies\\
    $\psi(\rotC{\theta}c)=\rotB{\theta}\psi(c)=(g_p,g_s)$. Hence, $\psi$ is onto.
\end{proof}

In the following Lemma we show that we can distinguish between the two branches of $\psi$. For an analytically admissible $c$ let $\det c$ denote the determinant of its Voigt matrix.

\begin{lemma}\label{lem:companion-normal-form}
    Let $c\in\analyticallyadmissibletensors$ be non-isotropic and let $b$ be its unique
    analytically admissible anomalous companion then $\det b\neq\det c$. Consequently every fibre
    $\psi^{-1}(\psi(c))=\{c,b\}$ has a unique element of largest determinant and a unique element of smallest determinant. This distinction is highlighted by denoting the stiffness tensor associated to the former by $c_+$ and to the latter by $c_-$.
\end{lemma}

\begin{proof}
    By Lemma~\ref{lem:rotated_tensors} there is $\theta\in[0,2\pi)$ with
    $(\rotC{\theta}c)_{13}=(\rotC{\theta}c)_{23}=0$, so $\rotC{\theta}c$ is one of~\eqref{eq:cparam} and~\eqref{eq:sibling}. By Lemma~\ref{lem:psi-equivariant}(1) the
    tensor $\rotC{\theta}b$ remains an anomalous companion of $\rotC{\theta}c$, and it is
    algebraically admissible by Lemma~\ref{lem:psi-equivariant}(2). Hence $\rotC{\theta}b$
    equals the other of~\eqref{eq:cparam} and~\eqref{eq:sibling} by the uniqueness in
    Corollary~\ref{cor:posdef}(1). For these two tensors $\det(c_+)>\det(c_-)$ by
    Lemma~\ref{lem:cct}, and the determinant of the Voigt matrix is invariant under
    $\rotC{\theta}$, so $\det c\neq\det b$.
\end{proof}

Using Corollary~\ref{cor:posdef} and Lemma~\ref{lem:companion-normal-form} we are now able to prove that $\psi$ is 1-to-1 or 2-to-1 respectively:

\begin{proof}[Proof of Proposition~\ref{prop:components+psismooth}\textup{(5)}]
    Let $c\in\analyticallyadmissibletensors$. By definition, $c'\in\psi^{-1}(\psi(c))$ if
    and only if $c'=c$ or $c'$ is an anomalous companion of $c$.

    If $c$ is isotropic, then by Corollary~\ref{cor:posdef}\,(2) no anomalous companion of
    $c$ is positive definite, hence none is admissible, and $\psi^{-1}(\psi(c))=\{c\}$.

    If $c$ is not isotropic, then $c$ is analytically admissible, so
    Corollary~\ref{cor:posdef}\,(1) provides exactly one analytically admissible anomalous
    companion $b$, so $\psi^{-1}(\psi(c))\subset\{c,b\}$.
\end{proof}

By Theorem~\ref{prop:components+psismooth}\textup{(5)} we are now certain that the map $\psi$ is either 1-to-1 or 2-to-1. It then follows
by Lemma~\ref{lem:companion-normal-form} the two elements of a two-element fibre of $\psi$
are distinguished by their determinant. We may therefore define
\begin{equation}
    \admissiblepairsofmetrics_-
    \coloneqq
    \{(g_p,g_s)\in\admissiblepairsofmetrics ;
    \#\psi^{-1}(g_p,g_s)=2\},
    \qquad
    \admissiblepairsofmetrics_+
    \coloneqq
    \admissiblepairsofmetrics,
\end{equation}
and maps
\begin{equation}
    \psi_+^{-1}\colon\admissiblepairsofmetrics_+\to\analyticallyadmissibletensors,
    \qquad
    \psi_-^{-1}\colon\admissiblepairsofmetrics_-\to\analyticallyadmissibletensors,
\end{equation}
sending $(g_p,g_s)$ to the element of $\psi^{-1}(g_p,g_s)$ of largest, respectively
smallest, determinant. These are well defined by
Proposition~\ref{prop:components+psismooth}\,(1) and~(5), and we set
\begin{equation}
    \analyticallyadmissibletensors_\pm
    \coloneqq
    \psi_\pm^{-1}\bigl(\admissiblepairsofmetrics_\pm\bigr).
\end{equation}
For a tensor $c$ satisfying~\eqref{eq:fixed_axes} and~\eqref{eq:Dcond} this labelling agrees
with the one in Lemma~\ref{lem:cct}, since $\det(c_+)>\det(c_-)$ there as well.

\begin{proof}[Proof of Proposition~\ref{prop:components+psismooth}\textup{(4)}]
    By construction $(g,h)\in\admissiblepairsofmetrics\setminus\admissiblepairsofmetrics_-$
    if and only if $\psi^{-1}(g,h)$ has exactly one element. By Proposition~\ref{prop:components+psismooth}\textup{(5)} a fibre of $\psi$ is a
    singleton precisely when its element is isotropic. Hence $c$ lies in
    $\psi^{-1}(\admissiblepairsofmetrics\setminus\admissiblepairsofmetrics_-)$ if and only
    if $\psi^{-1}(\psi(c))=\{c\}$, that is, if and only if $c$ is isotropic.
\end{proof}

\section{Smooth correspondence between metric tensors and stiffness tensors}
\label{sec:ImplicitFuncThm}

This section is devoted to proving parts~(2) and~(3) of Proposition~\ref{prop:components+psismooth}.
They are captured by Lemma~\ref{lem:smoothMap}. Once these parts are established we can prove Theorem~\ref{thm:how-many-stiffness-tensor-fields} at the end of this section.

For the proof of Lemma~\ref{lem:smoothMap} it is convenient to record a pair
$(g_p,g_s)\in\admissiblepairsofmetrics$ by the inverse quadratic forms
\begin{equation}
    Q\coloneqq g_p^{-1}\in\R^{2\times2}
    \qquad\text{and}\qquad
    g_s^{-1}\eqqcolon qI,\quad q\in\R,
\end{equation}
which is possible since $g_s$ is a multiple of the Euclidean metric. As inversion reverses
the order on positive definite matrices, condition~\eqref{eq:setE} becomes $Q>qI$, and by
a slight abuse of notation we write
\begin{equation}\label{eq:setEquadratic}
    \admissiblepairsofmetrics
    =
    \{(Q,q)\in\R^{2\times2}\times\R
    \;;\;
    Q=Q^{T},\ Q>qI>0\},
    \qquad
    \psi(c)=(Q,q).
\end{equation}
Once we show that this map is smooth it follows immediately that $\psi(c)=(g_p,g_s)$ is smooth as the inverse map is smooth on the open set of invertible matrices. We will think of $\admissiblepairsofmetrics$ as an open subset of $\R^3\times\R$ that consists of pairs of quadratic forms on a cotangent space.
By Proposition~\ref{prop:an-admissible} we know that $\analyticallyadmissibletensors$ is a subset of the variety defined by the polynomials~$P_1$ and~$P_2$, and therefore we think of it as a $4$-dimensional subset of~$\R^6$.

In the previous section we defined the sets~$\admissiblepairsofmetrics_\pm$ and
$\analyticallyadmissibletensors_\pm$ and showed in
Proposition~\ref{prop:components+psismooth}\,(1) and~(5) that the restriction
$\psi_\pm\colon\analyticallyadmissibletensors_\pm\to\admissiblepairsofmetrics_\pm$ is a
bijection with inverse~$\psi_\pm^{-1}$.
We next show that~$\analyticallyadmissibletensors_\pm$ is a smooth submanifold of~$\R^6$ and that the maps~$\psi_\pm$ are diffeomorphisms.

\begin{lemma}\label{lem:smoothMap}
The sets $\analyticallyadmissibletensors_\pm$ are smooth submanifolds of $\R^6$ of dimension $4$ and the maps $\psi_\pm\colon \analyticallyadmissibletensors_\pm\to \admissiblepairsofmetrics_\pm$ are diffeomorphisms.
\end{lemma}

\begin{proof}[Proof of Lemma~\ref{lem:smoothMap} corresponding to Proposition~\ref{prop:components+psismooth}(1)\&(2)]
    The proof is split into 6 steps:
    \begin{enumerate}
        \item Check whether the locus, where the polynomials $P_1$ and $P_2$ vanish, is smooth
        \item Prepare for the implicit function theorem: Define the graph of the map $\psi$ as the vanishing set of a map $f:\admissiblepairsofmetrics \times \analyticallyadmissibletensors \rightarrow \R^N$
        \item Find the rank of the Jacobian $Df$
        \item Use the rank to determine smoothness of the graph $G:=V(f)\subset \admissiblepairsofmetrics \times \analyticallyadmissibletensors$
        \item Use the implicit function theorem to conclude that the projections of the product space to $\admissiblepairsofmetrics$ and $\analyticallyadmissibletensors$ respectively restricted to $G$ are diffeomorphisms 
        \item In order for the last statement to hold, check that the derivatives $\partial_E f$ and $\partial_F f$ are injective
    \end{enumerate}

\textbf{Step (1): Smoothness of the admissible locus.}
We define the variety
\begin{equation}
\Sigma
=
\{\, c\in\R^6 \mid P_1(c)=P_2(c)=0 \,\}
\subset \R^6.
\end{equation}
The singular locus of $\Sigma$ consists of those stiffness tensors at which $\Sigma$ fails to be smooth. By Lemma~\ref{lem:SingLoc}(2), no positive definite stiffness tensor lies in the singular locus of $\Sigma$. Hence the positive definite part of $\Sigma$ lies in the regular locus of the variety.

The subsets~$\analyticallyadmissibletensors_\pm$ are disjoint and distinguished continuously by the determinant ordering of anomalous companions. Since the two branches do not meet and each subset is connected, they form connected components of the analytically admissible region. Therefore,~$\analyticallyadmissibletensors_\pm$ are relatively open subsets of~$\Sigma$. Since they lie in the regular locus of the $4$-dimensional variety~$\Sigma$, it follows that~$\analyticallyadmissibletensors_\pm$ are smooth $4$-dimensional submanifolds of~$\R^6$. By construction, the two submanifolds are disjoint.

\textbf{Step (2): Finding the graph condition.} Now both~$\analyticallyadmissibletensors_\pm$ and~$\admissiblepairsofmetrics_\pm$ have smooth structure. We want to use the implicit function theorem to assess whether~$\psi_\pm$ is smooth. This is a local argument and is thus the same for~$\psi_+$ and~$\psi_-$. Therefore we do not distinguish between the two maps in the following. We want to define a graph of~$\psi$ that encodes both~$\psi$ and its inverse symmetrically. If we can show that the projections from~$\admissiblepairsofmetrics\times \analyticallyadmissibletensors$ to~$\analyticallyadmissibletensors$ and~$\admissiblepairsofmetrics$ restricted to~$G$ are diffeomorphisms this implies that~$\psi$ is a diffeomorphism. For that purpose we want to use an adapted version of the implicit function theorem. Define the graph as the vanishing set of a map 
$
f:\admissiblepairsofmetrics \times \analyticallyadmissibletensors \rightarrow \R^N
$,
that is obtained by demanding that the slowness polynomial is a product of $v^T Qv-1$ and $q\abs{v}^2-1$. As discussed in section \ref{sec:AlgAdStiffness} there are stiffness tensors whose slowness surface factor as circle and ellipse, but they are not positive definite. For instance the indefinite isotropic stiffness tensor in~\eqref{eq:isoparamflipped}. This isotropic stiffness tensor defines a irreducible component of the variety that defines stiffness tensors whose slowness surface factor as circle and ellipse. We remove this component through saturation. The corresponding reduced generating set of such stiffness tensors is then given by 9 polynomial equations in $q, Q$ and $c$:
\begin{equation}\label{eq:graphconds}
\begin{aligned}
Q_{12} - c_{13} - c_{23} &= 0, \\
Q_{11} - Q_{22} - c_{11} + c_{22} &= 0, \\
q + Q_{22} - c_{22} - c_{33} &= 0, \\
c_{11}c_{23} - c_{12}c_{13} + c_{12}c_{23} - c_{13}c_{22} &= 0, \\
c_{11}c_{22} - c_{11}c_{33} - c_{12}^{2} - 2c_{12}c_{33}
+ c_{13}^{2} + 2c_{13}c_{23} - c_{22}c_{33} + c_{23}^{2} &= 0, \\
Q_{22}c_{13} + Q_{22}c_{23} - c_{12}c_{23}
- c_{13}c_{33} - c_{22}c_{23} - c_{23}c_{33} &= 0, \\
Q_{22}c_{12} + Q_{22}c_{22} - c_{12}c_{22}
- c_{13}c_{23} - c_{22}^{2} - c_{23}^{2} &= 0, \\
Q_{22}c_{11} - Q_{22}c_{22} - c_{11}c_{22}
- c_{13}^{2} + c_{22}^{2} + c_{23}^{2} &= 0, \\
Q_{22}^{2} - Q_{22}c_{22} - Q_{22}c_{33}
+ c_{22}c_{33} - c_{23}^{2} &= 0.
\end{aligned}
\end{equation}
The corresponding Magma code is in Listing \ref{lst:graph}.

\textbf{Step (3): Find the rank of the Jacobian $Df$.} The conditions~\eqref{eq:graphconds} define our map $f:\admissiblepairsofmetrics \times \analyticallyadmissibletensors \rightarrow \R^N$. The union of the two graphs is $G:=V(f)\subset \admissiblepairsofmetrics \times \analyticallyadmissibletensors$. We now want to determine the rank of the Jacobian $Df$. As we have 9 polynomial equations in 10 variables we expect $Df$ to have rank at most 9. We compute all $7\times 7$ minors on $G$ (see Listing \ref{lst:rank}), as all of these minors vanish on $G$, $Df$ can have rank at most 7.  Additionally, we compute a irredundant set of generators, where the rank of $Df$ drops additionally below 6 (see Listing \ref{lst:rank}). The generators are given by $c_{11}=c_{22}=c_{33}=-c_{12}$ and $c_{13}=c_{23}=0$. Such a stiffness tensor is never positive definite. Thus $\rank(Df)=6$ on $G$ and generically on the whole affine variety (of which $G$ is a subset).

\textbf{Step (4): Smoothness of the graph $G$. } Most points on $G$ are smooth (the singular locus is always a subvariety of positive codimension). By the affine Jacobi criterion \cite[Prop. 10.11]{gathmann2021alggeom} and \cite[Sec. 4.4.8(i)]{gathmann2002alggeomchapter4} a point $p\in G$ is smooth if and only $\rank(Df)=10-\dim(G)$. Therefore, $\dim(G)=4$. We have now found that $G\subset \admissiblepairsofmetrics\times \analyticallyadmissibletensors$ is a smooth $4$-dimensional submanifold.
Its tangent space at any point $p\in G$ is $T_pG=\ker(Df(p))$.

Because $G$ is a smooth manifold, near any point $p$ there is a smooth function $f_p\colon \admissiblepairsofmetrics\times \analyticallyadmissibletensors\to\R^6$ so that near $p$ we have $G=f_p^{-1}(0)$ and $\rank(Df_p)=6$.
(Dropping the number of defining functions to the codimension of a variety is known as writing the variety as a complete intersection, and indeed any variety is a locally complete intersection near a smooth point.)
This function $f_p$ is better suited for use with the implicit function theorem.
It will be important that $\ker(Df_p(p))=\ker(Df(p))=T_pG$ at all points $p\in G$.

\textbf{Step (5): The implicit function theorem.} Let us name the two projections of the product space:
\begin{equation}
\begin{split}
\pi&\colon \admissiblepairsofmetrics\times \analyticallyadmissibletensors\to \admissiblepairsofmetrics
\\
\rho&\colon \admissiblepairsofmetrics\times \analyticallyadmissibletensors\to \analyticallyadmissibletensors
.
\end{split}
\end{equation}
We will show that the restrictions $\pi_G\colon G\to \admissiblepairsofmetrics$ and $\rho_G\colon G\to\Sigma=\rho(G)$ are both diffeomorphisms.
Then by composition the map $\psi_\pm\colon \analyticallyadmissibletensors_\pm\to \admissiblepairsofmetrics_\pm$ (which we already know to be a bijection) is a diffeomorphism as claimed.

Fix any point $p=(g_0,c_0)\in G\subset \admissiblepairsofmetrics\times \analyticallyadmissibletensors$, where $g_0=(Q_0,q_0)$. We will show that our maps $\pi_G$ and $\rho_G$ are diffeomorphisms in a neighborhood of $p$. In the following we decompose $D f$ and $Df_p$ respectively into the partial derivatives corresponding to $\admissiblepairsofmetrics$ and $\analyticallyadmissibletensors$: $Df(p)=[\partial_{\pairsofmetrics}f(p) \, \vert \, \partial_{\stiffnesstensors} f(p)]$. Here $\partial_{\pairsofmetrics} f(p)\in \R^{9\times 4}$ and $\partial_{\stiffnesstensors} f(p)\in \R^{9\times 6}$.

For $\pi_G$ to be a local diffeomorphism, there has to be a smooth map $\Sigma\ni c\mapsto g(c)$ so that $f_p(g(c),c)=0$ for all $c$ near $c_0$. 
Smoothness of the map $\Sigma\ni c\mapsto g(c)$ is now provided by the implicit function theorem if the partial derivative $\partial_{\pairsofmetrics}f_p(p)$ is injective.
As this implies that only the zero vector is in the kernel of $\partial_{\pairsofmetrics}f_p(p)$, this condition can be recast as  $(\R^4\times0)\cap\ker(Df_p(p))=0$, which in turn is equivalent to $(\R^4\times0)\cap\ker(Df(p))=0$ or
\begin{equation}
\label{eq:ift-pi}
\rank(\partial_{\pairsofmetrics}f(p))
=
4
.
\end{equation}
We will return to checking this condition soon.

For $\rho_G$ to be a local diffeomorphism, there has to be a smooth map $g\mapsto c(g)\in \Sigma$ so that $f_p(g,c(g))=0$ for all $g$ near $g_0$.
Now the implicit function theorem requires that the restricted partial derivative $\partial_{\stiffnesstensors}f_p|_{T_{\rho(p)}\Sigma}$ is injective. Before proceeding, we note that $T_{\rho_G(p)}\Sigma=\ker(DP(\rho_G(p)))$ is justified, as the admissible locus is smooth by step (1). Now as $\Sigma$ is the vanishing set of $P=(P_1,P_2)$, this injectivity condition is equivalent to $(0\times\ker(DP(\rho(p))))\cap\ker(Df_p(p))=0$.
Using $\ker(Df_p(p))=\ker(Df(p))$, this is the same as the block matrix
\begin{equation}
B(p)
\coloneqq
\begin{pmatrix}
DP(\rho_G(p))\\
\partial_{\stiffnesstensors} f(p)
\end{pmatrix}
\end{equation}
satisfying
\begin{equation}
\label{eq:ift-rho}
\rank(B(p))
=
6
.
\end{equation}
Note that since there are two polynomial conditions on 6 variables $DP(\rho_G(p))\in \mathbb{R}^{2\times 6}$ and thus $B(p)\in \mathbb{R}^{11\times 6}$.
We will next check when these conditions hold.

\textbf{Step (6): Finding the ranks of $\partial_{\pairsofmetrics}f(p)$ and $B(p)$.}
We now verify the rank conditions~\eqref{eq:ift-pi} and~\eqref{eq:ift-rho} required for the implicit function theorem arguments above, using the Magma computation in Listing~\ref{lst:rankpirho}.

The ideal \texttt{failure1} (resp.\ \texttt{failure2}) is generated, together with $f$, by the $4\times 4$ minors of $\partial_{\pairsofmetrics} f$ (resp.\ the $6\times 6$ minors of $B$). By construction, its vanishing locus is exactly the set of points of $G$ where the rank condition~\eqref{eq:ift-pi} (resp.~\eqref{eq:ift-rho}) \emph{fails}:
\begin{align}
    V(\texttt{failure1}) &= \{p \in G : \rank(\partial_{\pairsofmetrics}f(p)) < 4\},\\
    V(\texttt{failure2}) &= \{p \in G : \rank(B(p)) < 6\}.
\end{align}

The minimal bases of $\langle\texttt{failure1}\rangle \cap G$ and $\langle\texttt{failure2}\rangle\cap G$, displayed first in Listing~\ref{lst:rankpirho}, give explicit equations for the two failure loci. In particular, the basis for $V(\texttt{failure1})$ shows that this locus forces $c_{12} = -c_{22}$ and $c_{11}=c_{22}$; together these give
\begin{equation}
    c_{11}c_{22} - c_{12}^2 = c_{22}^2 - c_{22}^2 = 0,
\end{equation}
so the leading $2\times 2$ principal minor of $c$ vanishes on $V(\texttt{failure1})$, and no point of this locus can correspond to a positive definite stiffness tensor. While this already shows $V(\texttt{failure1})$ is disjoint from $\analyticallyadmissibletensors_\pm$, we give the saturation argument below for a uniform treatment of both rank conditions.

To show that neither failure locus meets the domain of interest, we saturate each ideal by \texttt{detE}$\coloneqq \det(qI-Q)$ and, where necessary, by \texttt{detF}$\coloneqq \det(c)$. Recall that saturating an ideal~$I$ by~$J$ removes from~$V(I)$ any component contained in~$V(J)$; if the resulting ideal is the unit ideal $\langle 1\rangle$, the entire variety~$V(I)$ was contained in~$V(J)$.

Saturating \texttt{failure\_pi} by \texttt{detE} alone already yields
\begin{equation}
    \texttt{I2\_pi} = \langle 1 \rangle,
\end{equation}
so
\begin{equation}
    V(\texttt{failure1}) \subseteq \{\texttt{detE} = 0\}.
\end{equation}

Saturating \texttt{failure} by \texttt{detE} does not yet trivialize the ideal (\texttt{I2} remains positive-dimensional), but saturating further by \texttt{detF} does:
\begin{equation}
    \texttt{I3} = \langle 1 \rangle,
\end{equation}
so
\begin{equation}
    V(\texttt{failure2}) \subseteq \{\texttt{detE} = 0\} \cup \{\texttt{detF} = 0\}.
\end{equation}

Since~$\admissiblepairsofmetrics_\pm$ is defined so that ellipse and circle never touch $\texttt{detE} \neq 0$ throughout, and $\analyticallyadmissibletensors_\pm$ consists of positive definite stiffness tensors, for which $\texttt{detF} = \det(c) > 0$, both failure loci are disjoint from $G_\pm$. We conclude that
\begin{equation}
    \rank(\partial_{\pairsofmetrics}f(p)) = 4
    \quad\text{and}\quad
    \rank(B(p)) = 6
\end{equation}
hold at every point $p \in G_\pm$, as required.
\end{proof}

Having established smoothness allows us to prove Theorem~\ref{thm:how-many-stiffness-tensor-fields}:

\begin{proof}[Proof of Theorem~\ref{thm:how-many-stiffness-tensor-fields}]
    Write $(g,h)\coloneqq\phi^{*}\psi(c)=(\phi^{*}g_p,g_s)$, which by assumption is a
    smooth field of pairs of metrics on $\overline{\Omega}$, with $h=g_s$ at each point a
    positive multiple of the Euclidean metric. By definition,
    $c'\in\psi^{-1}(\phi^{*}\psi(c))$ if and only if $c'$ is a smooth admissible stiffness
    tensor field with $\psi(c'(x))=(g,h)(x)$ for every $x\in\overline{\Omega}$.

    Assume first that $\abs{v}_{\phi^{*}g_p}\geq\abs{v}_{g_s}$ at some
    $x\in\overline{\Omega}$ for some $v\in\R^2\setminus\{0\}$. Then
    $(g,h)(x)\notin\admissiblepairsofmetrics$, so no such $c'$ exists and
    $\#\psi^{-1}(\phi^{*}\psi(c))=0$. This proves~(1), and from now on we assume that
    $(g,h)$ takes its values in $\admissiblepairsofmetrics$.

    By Proposition~\ref{prop:components+psismooth}\,(2) the sets
    $\analyticallyadmissibletensors_\pm$ are the connected components of
    $\analyticallyadmissibletensors$. As $\overline{\Omega}$ is connected and $c'$ is
    continuous, $c'$ takes its values in exactly one of them. If $c'$ takes its values in
    $\analyticallyadmissibletensors_\pm$, then, $\psi_\pm$ being a bijection onto
    $\admissiblepairsofmetrics_\pm$, necessarily $c'=\psi_\pm^{-1}\circ(g,h)$; conversely
    this is a well defined element of the fibre as soon as $(g,h)$ takes its values in
    $\admissiblepairsofmetrics_\pm$, and it is then smooth, being the composition of the
    smooth field $(g,h)$ with the diffeomorphism $\psi_\pm^{-1}$ of
    Proposition~\ref{prop:components+psismooth}\,(3). Hence
    $\#\psi^{-1}(\phi^{*}\psi(c))\leq2$.

    Since $\admissiblepairsofmetrics_+=\admissiblepairsofmetrics$, the field
    $\psi_+^{-1}\circ(g,h)$ is always defined. By Remark~\ref{rem:isotropic-conformal},
    \begin{equation}\label{eq:EminusEminus}
        \admissiblepairsofmetrics\setminus\admissiblepairsofmetrics_-
        =
        \{(g_p,g_s)\in\admissiblepairsofmetrics ;
        g_p \text{ is a multiple of the identity}\},
    \end{equation}
    so $\psi_-^{-1}\circ(g,h)$ is defined precisely when $\phi^{*}g_p$ is nowhere a
    multiple of the identity. Whenever both are defined they are distinct, since at every
    $x\in\overline{\Omega}$ the fibre $\psi^{-1}\bigl((g,h)(x)\bigr)$ has two elements and
    $\det\bigl(\psi_+^{-1}(g,h)(x)\bigr)>\det\bigl(\psi_-^{-1}(g,h)(x)\bigr)$ by
    Lemma~\ref{lem:companion-normal-form}.

    Consequently $\#\psi^{-1}(\phi^{*}\psi(c))=2$ if $\phi^{*}g_p$ is nowhere a multiple of
    the identity, which proves~(3), and $\#\psi^{-1}(\phi^{*}\psi(c))=1$ otherwise, that
    is, if $\phi^{*}g_p$ is a multiple of the identity at some point, which proves~(2). As
    the three cases are mutually exclusive and exhaust all possibilities, the stated
    equivalences follow.
\end{proof}

\section{Characterization of admissible diffeomorphisms}
\label{sec:Diffeos}
 
As per Theorem~\ref{thm:how-many-stiffness-tensor-fields}, the number of stiffness tensor fields corresponding to a pullback of another one depends on whether the tensor is isotropic at any point.
The correspondence between stiffness tensors and pairs of metric tensors is broadly 2-to-1, encoded in the $\pm$ subscripts in the map $\psi_\pm$.
If there is an isotropic point (where the $p$-metric is a multiple of the identity), then we fall out of the range of the $\psi_-$ branch.
But there are also constraints for the $\psi_+$ branch, and we dive into them in this section.
We recall that the $+$ branch contains all isotropic stiffness tensors.
 
We address the question raised after proof of \thmref{thm:travelTimes}: Given a smooth stiffness tensor field
$c\colon\bar\Omega\to\analyticallyadmissibletensors$, with $\psi(c)=(g_p,g_s)$
which diffeomorphisms $\phi \colon \bar \Omega \to \bar \Omega$ satisfy
$\phi^*\psi(c)\in C^\infty(\bar\Omega,\admissiblepairsofmetrics)$ or in other words $\phi^\ast g_p<g_s$ in sense of quadratic forms? We call such maps admissible.
In the main result of this section, Theorem \ref{thm:admdiff:admissibility-preserving-diffeomorphisms}, we establish three equivalent conditions for admissibility. We also show that the set of admissible diffeomorphisms is open in~$C^1$ topology, and prove that this set is not a group.
 
Throughout this section, $\Omega\subset\R^2$ is a connected smooth
bounded domain, and all metrics and diffeomorphisms are smooth on
$\bar\Omega$. We fix $c \in \analyticallyadmissibletensors$ and write
\begin{equation}
  \psi(c)=(g_p,g_s)\in
  C^\infty(\bar\Omega,\admissiblepairsofmetrics),
  \qquad g_s(x)=\sigma(x)\Id.
\end{equation}
Thus $0<g_p(x)<\sigma(x)\Id$, where $\sigma$ is a smooth positive function on $\bar \Omega$. Hence forth also these metrics are fixed, and we identify the metrics with their
matrices in the Euclidean coordinates. We write $\|\cdot\|_{\op}$ for the Euclidean operator norm and $|\cdot|_g$ for a Riemannian norm. If $g$ is Euclidean then we simply denote the respective norm by $|\cdot|$.
 
As in \thmref{thm:how-many-stiffness-tensor-fields}, the pullback of a diffeomorhism $\phi$ acts only on the first metric:
\begin{equation}
  \phi^*(g_p,g_s)=(\phi^*g_p,g_s).
\end{equation}
Since, the matrix of a Riemannian metric is positive definite one can compute their square root matrices. Particularly, the matrix
\begin{equation}
    \label{eq:matrix_B}
  B_\phi(x)=\sigma(x)^{-1/2}g_p(\phi(x))^{1/2}D\phi(x)
\end{equation}
is the total derivative of $\phi$ expressed in the length units of $g_s$
in the domain and $g_p$ in the target space. More precisely, if
$v=w/\sqrt{\sigma(x)}$, then
\begin{equation}
    \label{eq:equivalence_of_norms_via_B}
  \abs{v}_{g_s(x)}=\abs{w},
  \qquad
  \text{and}
  \qquad
  \abs{D\phi(x)v}_{g_p(\phi(x))}=\abs{B_\phi(x)w}.
\end{equation}
Thus $B_\phi(x)$ maps a unit disk of a tangent plane at $x$, expressed in
$g_s$-orthonormal coordinates, to an ellipse in
$g_p$-orthonormal coordinates at $\phi(x)$. For the lengths semiaxis of this ellipse
we write $\alpha_\phi(x)\geq\beta_\phi(x)>0$. Naturally, these are also the singular values of $B_\phi(x)$.
 
We let
\begin{equation}
    \label{eq:char_of_q_phi}
  q_\phi
  :=\max_{x\in \bar \Omega}\alpha_\phi(x)
  =\max_{\substack{x\in\bar\Omega\\v\ne0}}
       \frac{\abs{D\phi(x)v}_{g_p(\phi(x))}}{\abs{v}_{g_s(x)}}.
\end{equation}
Indeed, for a symmetric and positive definite matrix $g$, its positive definite
square root satisfies $\abs{u}_g=\abs{g^{1/2}u}$. At a fixed point
$x\in \bar \Omega$, the substitution $w=\sqrt{\sigma(x)}v$ therefore gives
\begin{equation}
    \label{eq:q_phi_justification}
  \frac{\abs{D\phi(x)v}_{g_p(\phi(x))}}{\abs{v}_{g_s(x)}}
  =\frac{\abs{g_p(\phi(x))^{1/2}D\phi(x)v}}
         {\sqrt{\sigma(x)}\abs{v}}
    =\abs{B_\phi(x)\frac{v}{\abs{v}}}.
\end{equation}
Consequently,
\begin{equation}
    \label{eq:q_phi_justification_2}
  \max_{v\neq 0}
    \frac{\abs{D\phi(x)v}_{g_p(\phi(x))}}{\abs{v}_{g_s(x)}}
  =\|B_\phi(x)\|_{\op}
  =\alpha_\phi(x),
\end{equation}
and taking the maximum over $x\in \bar \Omega$ proves the expression~\eqref{eq:char_of_q_phi}.
 
The main result of this section is as follows.
\begin{theorem}[Characterization of admissible diffeomorphisms]
\label{thm:admdiff:admissibility-preserving-diffeomorphisms}
Let $\Omega\subset \R^2$ be connected, smooth and bounded. Choose
$(g_p,g_s)\in C^\infty(\bar\Omega,\admissiblepairsofmetrics)$ and let
$\phi\colon \bar\Omega\to\bar\Omega$ be a  diffeomorphism.
Then the following conditions are equivalent:
\begin{enumerate}
\item
The map $\phi$ is admissible:
\begin{equation}
  (\phi^*g_p,g_s)
  \in C^\infty(\bar\Omega,\admissiblepairsofmetrics).
\end{equation}
\item $q_\phi<1$.
\item
The map $\phi \colon (\bar \Omega,g_s) \to (\bar\Omega, g_p)$ is $q$-Lipschitz for some $q\in (0,1)$.
\end{enumerate}
\end{theorem}
 
\begin{proof}
Since~(1) means $0<\phi^\ast g_p<g_s$ it follows from equation~\eqref{eq:char_of_q_phi} that conditions~(1) and~(2) are equivalent.
 
Suppose that~(2) holds. Clearly formula~\eqref{eq:char_of_q_phi} yields
\begin{equation}
  \abs{D\phi(x)v}_{g_p(\phi(x))}
  \leq q_\phi \abs{v}_{g_s(x)},
  \qquad \text{for all } x\in \bar \Omega, \: v\in \R^2.
\end{equation}
Hence, integrating along any piecewise smooth curve~$\gamma$ in~$\bar\Omega$ yields
\begin{equation}
  \operatorname{Length}_{g_p}(\phi\circ\gamma)
  \leq q_\phi\operatorname{Length}_{g_s}(\gamma).
\end{equation}
Taking the infimum over curves joining two prescribed points proves~(3), with $q=q_\phi<1$.
 
Suppose that~(3) holds and choose $x \in \Omega$. We first recall that  for any a smooth curve $\gamma$ in $\Omega$ with $\gamma(0)=x$
and $\dot\gamma(0)=v$, we have the following first order expansion for Riemannian distances:
\begin{equation}
  d_g(x,\gamma(t))=t\abs{v}_{g(x)}+o(t)
  \qquad\text{as }t\downarrow0,
  \qquad \text{and $g\in \{g_p,g_s\}$.}
\end{equation}
 
Now fix $v\in\R^2$, and choose a curve
$\gamma$ as above. Since $\phi(x)\in\Omega$ and
$(\phi\circ\gamma)'(0)=D\phi(x)v$, an application of the distance expansion to
$g_s,\gamma$ and to $g_p,\phi\circ\gamma$, with condition~(3) now yields
\begin{equation}
  \abs{D\phi(x)v}_{g_p(\phi(x))}\leq q\abs{v}_{g_s(x)}.
\end{equation}
This inequality extends to boundary points by continuity. Hence
$q_\phi\leq q<1$, follows from~\eqref{eq:char_of_q_phi} proving~(2).
The proof is complete.
\end{proof}
 
We use the following convention
\begin{equation}
  \|f\|_{C^1}
  :=
  \max\left\{
    \sup_{x\in\bar\Omega}\abs{f(x)},\:
    \sup_{x\in\bar\Omega}\|Df(x)\|_{\op}
  \right\},
\end{equation}
for the $C^1$-norm,
and show that the set of admissible diffeomorphisms is open in $C^1$ topology. This follows from then condition (2) of Theorem~\ref{thm:admdiff:admissibility-preserving-diffeomorphisms} as soon as we show that the map $\phi \mapsto q_\phi$ is continuous in $C^1$ norm. This is a content of the following lemma.
 
\begin{lemma}
\label{lem:admdiff:openness}
Let $\Omega\subset\R^2$ be a bounded smooth domain, and $
  (g_p,g_s)\in
  C^\infty(\bar\Omega,\admissiblepairsofmetrics).
$
The map $\phi\mapsto q_\phi$ is continuous in the $C^1$ topology of the space of all diffeomorphisms of $\bar \Omega$. Moreover, the set of all admissible diffeomorphisms of $\bar \Omega$
is open in the $C^1$ topology and contains the identity.
\end{lemma}
 
\begin{proof}
We fix an arbitrary diffeomorphism
$\phi_0$  of $\bar \Omega$.
We use a shorthand notation
$
  S(x)=g_p(x)^{1/2},
$
and define
\begin{equation}
\sigma_{\min}=\min_{x \in \bar \Omega}\sigma(x),
\qquad
  M=\max_{x\in\bar\Omega}\|S(x)\|_{\op},
  \qquad
  N=\max_{x\in\bar\Omega}\|D\phi_0(x)\|_{\op}.
\end{equation}
Compactness implies that $S$ is
uniformly continuous on $\bar \Omega$. In particular, its modulus of continuity
\begin{equation}
  \omega_S(r)
  :=
  \sup_{\substack{x,y\in\bar\Omega\\ \abs{x-y}\leq r}}
       \|S(y)-S(x)\|_{\op}
\end{equation}
satisfies $\omega_S(r)\to0$ as $r\downarrow0$.
 
For another diffeomorphism $\phi$, we write
\begin{equation}
  d_0=\sup_{x\in\bar\Omega}\abs{\phi(x)-\phi_0(x)},
  \qquad
  d_1=\sup_{x\in\bar\Omega}
           \|D\phi(x)-D\phi_0(x)\|_{\op}.
\end{equation}
Hence, the definition of $\omega_S$ gives
\begin{equation}
  \|S(\phi(x))-S(\phi_0(x))\|_{\op}
  \leq\omega_S(d_0).
\end{equation}
 
Then we use the equality
\begin{align}
  B_\phi(x)-B_{\phi_0}(x)
  ={}&
  \sigma(x)^{-1/2}
  \bigl(S(\phi(x))-S(\phi_0(x))\bigr)D\phi_0(x)\\
  &+\sigma(x)^{-1/2}S(\phi(x))
       \bigl(D\phi(x)-D\phi_0(x)\bigr),
\end{align}
with the triangle inequality and submultiplicativity of the operator
norm to find that
\begin{align}
  \|B_\phi(x)-B_{\phi_0}(x)\|_{\op}
  \leq{}&
  \frac{N\omega_S(d_0)+Md_1}{\sqrt{\sigma_{\min}}}.
\end{align}
 
To pass from this estimate to the maxima defining $q_\phi$
and $q_{\phi_0}$, we observe that for every $x \in \bar \Omega$, we have due to~\eqref{eq:char_of_q_phi} and~\eqref{eq:q_phi_justification_2} that
\begin{equation}
  \|B_\phi(x)\|_{\op}
  \leq q_{\phi_0}+ \|B_\phi(x)-B_{\phi_0}(x)\|_{\op}
  \leq
    q_{\phi_0}+\frac{N\omega_S(d_0)+Md_1}{\sqrt{\sigma_{\min}}}.
\end{equation}
Hence, after taking the maximum over $x\in \bar \Omega$ and interchaning the roles of $\phi$ and $\phi_0$ we get
\begin{equation}
  \abs{q_\phi-q_{\phi_0}}
  \leq
  \frac{N\omega_S(d_0)+Md_1}{\sqrt{\sigma_{\min}}}.
\end{equation}
Since $\|\phi-\phi\|_{C^1}=\max\{d_0,d_1\}$ we have proved the
required continuity.
\end{proof}
 
We conclude this section by showing that the set of admissible diffeomorphisms is not a group.
 
\begin{proposition}
\label{prop:admdiff:not-a-subgroup}
For a fixed admissible pair $(g_p,g_s)$, the set of
admissible boundary fixing diffeomorphisms of $\bar \Omega$
is open in the relative $C^1$ topology and contains the identity, but is not a subgroup of all boundary fixing diffeomorphisms of $\bar \Omega$.
\end{proposition}
\begin{proof}
The openness follows from Lemma~\ref{lem:admdiff:openness}
 
Choose an interior point $x_0\in \Omega$ and a function
$\chi\in C_c^\infty(\Omega)$ that equals~$1$ in a neighborhood
of~$x_0$. We define a smooth compactly supported vector field
\begin{equation}
  X(x)=\chi(x)\begin{pmatrix}x_1-x_{0,1}\\0\end{pmatrix}.
\end{equation}
Hence,
\begin{equation}
  X(x_0)=0,
  \qquad
  DX(x_0)=\begin{pmatrix}1&0\\0&0\end{pmatrix}.
\end{equation}
 
Let $\Phi_t(x)$ be the flow of~$X$. Since~$X$ is smooth and compactly supported,~$\Phi$ is smooth in $(t,x)\in \R\times \R^2$ and for each $x \in \R^2$ the function $t\mapsto\Phi_t(x)$ is a global solution to the initial-value problem
\begin{equation}
  \frac{\partial}{\partial t}\Phi_t(x)=X(\Phi_t(x)),
  \qquad \Phi_0(x)=x.
\end{equation}
Furthermore, for all $t\in \R$ $x\mapsto \Phi_t(x)$ is a smooth diffeomorphism of $\R^2$.
 
By uniqueness of initial value problems $\Phi_t$ fixes every point where $X$ vanishes. Particularly, $x_0$ and a neighborhood of
$\partial\Omega$. Hence, $\Phi_t$ is a smooth
diffeomorphism of $\bar\Omega$ that fixes the boundary. On the other hand smooth dependence on $(t,x)$ as well as the  compactness of $\bar\Omega$ imply $\Phi_t\to\operatorname{id}$ in $C^1$ as $t\to0$.
For sufficiently small $t>0$, openness therefore gives the admissibility of
$\Phi_t$.
 
To compute $D\Phi_t(x_0)$ we first apply the chain rule and get for all $x \in \R^2$ that
\begin{equation}
  \frac{\partial}{\partial t}D\Phi_t(x)
  =DX(\Phi_t(x))D\Phi_t(x),
  \qquad D\Phi_0(x)=\Id.
\end{equation}
Since $\Phi_t(x_0)=x_0$, evaluation of this equation at $x=x_0$ yields that
$
D\Phi_t(x_0)
$
solves the constant-coefficient matrix equation
\begin{equation}
  A'(t)=\begin{pmatrix}1&0\\0&0\end{pmatrix}A(t),
  \qquad A(0)=\Id
  \qquad \Rightarrow \qquad
  D\Phi_t(x_0)=\begin{pmatrix}e^t&0\\0&1\end{pmatrix}.
\end{equation}
 
Therefore, for every $m \in \mathbb{Z}$ and $t>0$ we have that $\Phi_t^m=\Phi_{mt}$, and
\begin{equation}
  ((\Phi_t^m)^*g_p)_{x_0}(e_1,e_1)
  =e^{2mt}(g_p)_{x_0}(e_1,e_1).
\end{equation}
This exceeds $(g_s)_{x_0}(e_1,e_1)$ for sufficiently large $m$.
Thus some positive iterate of an admissible map $\Phi_t$ is not admissible,
which rules out the subgroup property.
\end{proof}

\appendix

\section{Brief introduction to algebraic geometry}
\label{app:AG}

As many of the readers may be unfamiliar with algebraic geometry, we give a crash course in this appendix with a specific focus on the tools and ideas we apply.
For a fuller introduction to the field to readers with a background in analysis, we recommend~\cite{invitation-to-AG}.

A (commutative) ring is a set with two compatible operations: addition and multiplication.
A simple example of a ring is a familiar field like $\R$ or $\C$.
The rings we use will mostly be polynomial rings, such as the ring $\C[x,y]$ of all polynomial maps $\C^2\to\C$, which can also be identified as abstract polynomials in the two\footnote{The number of variables can be anything, but we stick with two for concreteness.} variables $x$ and $y$ with coefficients in $\C$.
The set of polynomials becomes a ring under pointwise operations on the polynomial functions.

An ideal $J$ is a subset of a ring $R$ that is closed under addition ($p,q\in J\implies p+q\in J$) and absorbing under multiplication ($p\in J,q\in R\implies pq\in J$).
Any set $S\subset\C^2$ gives rise to an ideal
\begin{equation}
\I(S)
=
\{
f\in\C[x,y]
;
f|_S=0
\}
\end{equation}
of polynomial functions vanishing on $S$.
It will be most natural to apply this map $\I$ to varieties, which here simply means sets defined by any number of polynomial equations.
Corresponding to an ideal $J$ we have the corresponding vanishing set
\begin{equation}
\V(J)
=
\{
(x,y)\in\C^2
;
f(x,y)=0
\text{ for all }
f\in J
\}
,
\end{equation}
which is a variety.

A variety is a geometric concept, while an ideal is an algebraic one.
The map~$\V$ taking~$J$ to~$\V(J)$ is a map from algebra to geometry, and the map~$\I$ taking~$S$ to~$\I(S)$ vice versa.
These maps are almost inverses of each other, namely:
\begin{itemize}
\item
For any set $S\subset\C^2$, the set $\V(\I(S))\subset\C^2$ is closure of~$S$ in the Zariski topology.
(In fact, this closure operator uniquely determines the topology.)
\item
For any ideal $J\subset\C[x,y]$, the ideal~$\I(\V(J))$ is what is called the radical~$\sqrt{J}$ of~$J$.
(The radical is the set of those ring elements that have an integer power in~$J$. These elements form an ideal, and~$J$ is said to be radical when $J=\sqrt{J}$. The radical describes the same vanishing set with the simplest equations possible.)
\end{itemize}
We do get unique correspondence between Zariski-closed sets in~$\C^2$ and radical ideals of the polynomial of $\C[x,y]$.
This correspondence is at the heart of elementary algebraic geometry.

An ideal $J\subset R$ is generated by some elements $q_1,q_2,\dots$ in the sense that~$J$ is the smallest ideal containing all these.
This is equivalent with elements of~$J$ being finite $R$-linear combinations of the form $r_1q_1+r_2q_2+\dots$.
As our rings are Noetherian, all rings are finitely generated.
The minimal number of generators does not in general match the codimension of the variety, so generators of ideals do not behave like a basis of a subspace of a vector space.
We do at times pick a minimal basis for a more minimalistic description or a Gröbner basis to make certain calculations possible.
Passing from a general basis to a Gröbner basis is vaguely analogous to passing from a general basis to an orthonormal one in linear algebra, but with the surprising difference that the basis can grow substantially in this process.

What we described above is an affine variety, as it is a subspace of the affine space~$\C^2$.
It is often convenient to enlarge the affine space~$\C^2$ into the projective space~$\C\ps^2$, defined as the quotient of $\C^3\setminus\{0\}$ by identifying~$a$ and~$b$ when $a=\lambda b$ for $\lambda\in\C$.
Points are written as $(x,y)\in\C^2$ and $[x:y:z]\in\C\ps^2$, so that $(1,2)\neq(3,6)$ but $[1:2:0]=[3:6:0]$.
Polynomials defining a variety in a projective space need to be homogeneous so that the variety is well defined.
The affine space embeds into the projective space naturally as $(x,y)\xmapsto{\iota}[x:y:1]$, so~$\C\ps^2$ is a compactification of~$\C^2$ whose halo is $\{[x:y:0];[x,y]\in\C\ps^1\}\approx\C\ps^1$.
An affine variety has a projective closure, determined by minimally extending it from~$\C^2$ to~$\C\ps^2$ via the use of Gröbner bases.

Many of our varieties live in product spaces.
Suppose~$X$ and~$Y$ are affine or projective spaces and consider the projection $\pi\colon X\times Y\to X$.
If we have a variety $V\subset X\times Y$, we may want to understand the image $\pi(V)\subset X$.
Is it a variety?
If yes, which polynomial conditions describe it?

In general,~$\pi(V)$ is a variety when~$Y$ is a projective space but not always when~$Y$ is an affine space.
That is, projecting out a projective space is clean.
If~$Y$ is affine, points projecting from infinity (the projective closure minus the original affine manifold) belong to the closure (in either topology) of~$\pi(V)$ but not to~$\pi(V)$.
That is why projecting out an affine space comes at the cost of analyzing the ``ghost points projecting from infinity''.
The generators of the so-called elimination ideal~$\I(\pi(V))$ can be computed by elimination theory, and this, too, is powered by Gröbner bases.

A variety may or may not be smooth.
We use the so-called Jacobi criterion to determine smoothness:
if the Jacobi matrix of the polynomial conditions has constant rank, then the variety is smooth and of constant dimension.

Many things in algebra work better over complex than real numbers.
Even though the physical problems at hand are real, they are best treated by relaxing all parameters to be complex and manually restricting the outcome of the algebraic machinery to the real setting.
This, fortunately, is not an issue in the present paper.

\section{Rotation and reflection actions on stiffness tensors}
\label{ap:rotrefl}

In Voigt notation
\begin{equation}
  \rotC{\theta}c
  \;\cong\;
  \begin{bmatrix}
    (\rotC{\theta}c)_{11} & (\rotC{\theta}c)_{12} & (\rotC{\theta}c)_{13}\\
    (\rotC{\theta}c)_{12} & (\rotC{\theta}c)_{22} & (\rotC{\theta}c)_{23}\\
    (\rotC{\theta}c)_{13} & (\rotC{\theta}c)_{23} & (\rotC{\theta}c)_{33}
  \end{bmatrix},
\end{equation}
with
 \begin{align}
 \begin{aligned}\label{eq:rotCexplicit}
  (\rotC{\theta}c)_{11}
  ={}& \cos^4(\theta)c_{11}
     - 4\sin(\theta)\cos^3(\theta)c_{13}
     + 2\sin^2(\theta)\cos^2(\theta)\bigl(c_{12}+2c_{33}\bigr)\\
   &\quad
     - 4\sin^3(\theta)\cos(\theta)c_{23}
     + \sin^4(\theta)c_{22},\\
  (\rotC{\theta}c)_{12}
  ={}& \cos^4(\theta)c_{12}
     + 2\sin(\theta)\cos^3(\theta)\bigl(c_{13}-c_{23}\bigr)\\
     &\quad+ \sin^2(\theta)\cos^2(\theta)\bigl(c_{11}+c_{22}-4c_{33}\bigr)\\
   &\quad
     - 2\sin^3(\theta)\cos(\theta)\bigl(c_{13}-c_{23}\bigr)
     + \sin^4(\theta)c_{12},\\
  (\rotC{\theta}c)_{13}
 ={}&
 4\cos^4(\theta)c_{13}
 +\sin(\theta)\bigl(c_{11}-2c_{12}-4c_{33}\bigr)\cos^3(\theta)
 \\
 &\quad
 +\bigl(4\sin^2(\theta)c_{23}-3c_{13}\bigr)\cos^2(\theta)
 \\
 &\quad
 +\sin(\theta)\bigl(-c_{22}\sin^2(\theta)+c_{12}+2c_{33}\bigr)\cos(\theta)
 \\
 &\quad
 -\sin^2(\theta)c_{23},\\
  (\rotC{\theta}c)_{22}
  ={}& \sin^4(\theta)c_{11}
     + 4\sin^3(\theta)\cos(\theta)c_{13}
     + 2\sin^2(\theta)\cos^2(\theta)\bigl(c_{12}+2c_{33}\bigr)\\
   &\quad
     + 4\sin(\theta)\cos^3(\theta)c_{23}
     + \cos^4(\theta)c_{22},\\
  (\rotC{\theta}c)_{23}
 ={}&
 4\cos^4(\theta)c_{23}
 +2\left(c_{12}-\frac{c_{22}}{2}+2c_{33}\right)
 \sin(\theta)\cos^3(\theta)
 \\
 &\quad
 +\bigl(4\sin^2(\theta)c_{13}-3c_{23}\bigr)\cos^2(\theta)
 \\
 &\quad
 +\sin(\theta)\bigl(c_{11}\sin^2(\theta)-c_{12}-2c_{33}\bigr)
 \cos(\theta)
 \\
 &\quad
 -\sin^2(\theta)c_{13},\\
  (\rotC{\theta}c)_{33}
  ={}& \cos^4(\theta)c_{33}
     + 2\sin(\theta)\cos^3(\theta)\bigl(c_{13}-c_{23}\bigr)\\
     &\quad+ \sin^2(\theta)\cos^2(\theta)\bigl(c_{11}+c_{22}-2c_{12}-2c_{33}\bigr)\\
   &\quad
     - 2\sin^3(\theta)\cos(\theta)\bigl(c_{13}-c_{23}\bigr)
     + \sin^4(\theta)c_{33}.
     \end{aligned}
 \end{align}

 Likewise,
\begin{equation}
  \reflC{\theta}c
  \;\cong\;
  \begin{bmatrix}
    (\reflC{\theta}c)_{11} & (\reflC{\theta}c)_{12} & (\reflC{\theta}c)_{13}\\
    (\reflC{\theta}c)_{12} & (\reflC{\theta}c)_{22} & (\reflC{\theta}c)_{23}\\
    (\reflC{\theta}c)_{13} & (\reflC{\theta}c)_{23} & (\reflC{\theta}c)_{33}
  \end{bmatrix},
\end{equation}
with
 \begin{align}
  (\reflC{\theta}c)_{11}
  ={}& \cos^4(2\theta)c_{11}
     + 4\sin(2\theta)\cos^3(2\theta)c_{13}
     + 2\sin^2(2\theta)\cos^2(2\theta)\bigl(c_{12}+2c_{33}\bigr)\\
   &\quad
     + 4\sin^3(2\theta)\cos(2\theta)c_{23}
     + \sin^4(2\theta)c_{22},\\
  (\reflC{\theta}c)_{12}
  ={}& \cos^4(2\theta)c_{12}
     - 2\sin(2\theta)\cos^3(2\theta)\bigl(c_{13}-c_{23}\bigr)\\
     &\quad+ \sin^2(2\theta)\cos^2(2\theta)\bigl(c_{11}+c_{22}-4c_{33}\bigr)\\
   &\quad
     + 2\sin^3(2\theta)\cos(2\theta)\bigl(c_{13}-c_{23}\bigr)
     + \sin^4(2\theta)c_{12},\\
  (\reflC{\theta}c)_{13}
 ={}&
 -4\cos^4(2\theta)c_{13}
 +\sin(2\theta)\bigl(c_{11}-2c_{12}-4c_{33}\bigr)\cos^3(2\theta)
 \\
 &\quad
 -\bigl(4\sin^2(2\theta)c_{23}-3c_{13}\bigr)\cos^2(2\theta)
 \\
 &\quad
 +\sin(2\theta)\bigl(-c_{22}\sin^2(2\theta)+c_{12}+2c_{33}\bigr)\cos(2\theta)
 \\
 &\quad
 +\sin^2(2\theta)c_{23},\\
  (\reflC{\theta}c)_{22}
  ={}& \sin^4(2\theta)c_{11}
     - 4\sin^3(2\theta)\cos(2\theta)c_{13}
     + 2\sin^2(2\theta)\cos^2(2\theta)\bigl(c_{12}+2c_{33}\bigr)\\
   &\quad
     - 4\sin(2\theta)\cos^3(2\theta)c_{23}
     + \cos^4(2\theta)c_{22},\\
  (\reflC{\theta}c)_{23}
 ={}&
 -4\cos^4(2\theta)c_{23}
 +2\left(c_{12}-\frac{c_{22}}{2}+2c_{33}\right)
 \sin(2\theta)\cos^3(2\theta)
 \\
 &\quad
 -\bigl(4\sin^2(2\theta)c_{13}-3c_{23}\bigr)\cos^2(2\theta)
 \\
 &\quad
 +\sin(2\theta)\bigl(c_{11}\sin^2(2\theta)-c_{12}-2c_{33}\bigr)
 \cos(2\theta)
 \\
 &\quad
 +\sin^2(2\theta)c_{13}\\
   (\reflC{\theta}c)_{33}
  ={}& \cos^4(2\theta)c_{33}
     - 2\sin(2\theta)\cos^3(2\theta)\bigl(c_{13}-c_{23}\bigr)\\
     &\quad+ \sin^2(2\theta)\cos^2(2\theta)\bigl(c_{11}+c_{22}-2c_{12}-2c_{33}\bigr)\\
   &\quad
     + 2\sin^3(2\theta)\cos(2\theta)\bigl(c_{13}-c_{23}\bigr)
     + \sin^4(2\theta)c_{33}.
 \end{align}

\section{Extended proofs}
\label{app:long-proofs}

This appendix is devoted to giving detailed proofs to complement some of the very condensed ones we gave above.
The extended proofs contain much more exposition and intermediate steps.

\begin{proof}[Detailed proof of Proposition \ref{prop:alg-admissible}]
In the following we treat all coefficients and components as complex numbers and study the real implications later.
    A general slowness polynomial is of the form
    \begin{align}\label{eq:slownessgeneral}
    \begin{aligned}
        \spoly{c}(x,\xi)&=(c_{11}c_{33} - c_{13}^2)\xi_1^4+ (2c_{11}c_{23} - 2c_{12}c_{13})\xi_1^3\xi_2\\
        &\hspace{6mm}+ (c_{11}c_{22} - c_{12}^2 - 2c_{12}c_{33} + 2c_{13}c_{23})\xi_1^2\xi_2^2
        - (c_{11} + c_{33})\xi_1^2\\
        &\hspace{6mm}+ (-2c_{12}c_{23} + 2c_{13}c_{22})\xi_1\xi_2^3- (2c_{13} + 2c_{23})\xi_1\xi_2\\
        &\hspace{6mm}+ (c_{22}c_{33} - c_{23}^2)\xi_2^4- (c_{22} + c_{33})\xi_2^2+ 1.
        \end{aligned}
    \end{align}
    For the slowness polynomial $\spoly{c}(x,\xi)$ to vanish on a circle it must be divisible by the polynomial $\xi_1^2+\xi_2^2-r^2$ for some $r\in\C$. As $\spoly{c}$ has degree 4, the other factor must also be quadratic.
    Therefore
    \begin{equation}
        \spoly{c}(x,\xi)=(\xi_1^2+\xi_2^2-r^2)(k_{20}\xi_1^2 + k_{11}\xi_1\xi_2 + k_{02}\xi_2^2 + k_{10}\xi_1 + k_{01}\xi_2+ k_{00}),
    \end{equation}
    with some coefficients $k_{ij}$.
    Expanding $\spoly{c}$ above and matching the general expression for $\spoly{c}$ in~\eqref{eq:slownessgeneral} gives a finite set of polynomial conditions in the variables $k_{ij}$ with $i,j=0,1,2$, the radius $r$ and entries of the stiffness tensor $c$. These polynomial conditions generate an ideal $I$ in the polynomial ring $\C[c,r,k]$.
    This ideal represents all algebraic constraints that the coefficients $k_{ij}$, $r$ and the entries of the stiffness tensor must satisfy for $\spoly{c}$ to match the target polynomial.
    
    To study the solutions of this system (i.e., the possible values of~$c$, $k$ and~$r$), we compute a Gröbner basis for~$I$ with respect to a graded reverse lexicographic monomial order, first in~$k_{ij}$ and~$r$ and then~$c$ (see the corresponding magma code in Listing \ref{lst:magma1}):
    \begin{equation}
\label{eq:original-mixed-conditions}
\begin{split}
    &
    r^2 k_{00} + 1, \\&
    c_{11}c_{13}c_{23}r^2 - c_{11}c_{12}c_{33}r^2 + c_{12}k_{00} + c_{11}c_{12} - c_{13}^2 - c_{13}c_{23} + c_{12}c_{33}, \\&
    c_{13}^2r^2 - c_{11}c_{33}r^2 + k_{00} + c_{11} + c_{33}, \\&
    c_{13}c_{22}r^2 - c_{12}c_{23}r^2 - c_{13} - c_{23}, \\&
    c_{23}^2r^2 - c_{22}c_{33}r^2 + k_{00} + c_{22} + c_{33}, \\&
    k_{00}^2 + c_{22}k_{00} + c_{33}k_{00} - c_{23}^2 + c_{22}c_{33}, \\&
    c_{12}c_{23}k_{00} + c_{22}c_{23}k_{00} - c_{13}c_{23}^2 - c_{23}^3 + c_{12}c_{23}c_{33} + c_{22}c_{23}c_{33}, \\&
    c_{11}k_{00} - c_{22}k_{00} - c_{13}^2 + c_{23}^2 + c_{11}c_{33} - c_{22}c_{33}, \\&
    c_{13}k_{00} + c_{23}k_{00} + c_{13}c_{22} - c_{12}c_{23}, \\&
    k_{20} + c_{13}^2 - c_{11}c_{33}, \\&
    k_{11} - 2c_{13}c_{22} + 2c_{12}c_{23}, \\&
    k_{02} + c_{23}^2 - c_{22}c_{33}, \\&
    k_{10}, \\&
    k_{01}, \\&
    c_{13}^3 + c_{11}c_{13}c_{22} - c_{13}c_{22}^2 - c_{11}c_{12}c_{23} + c_{13}^2c_{23} + c_{12}c_{22}c_{23} - c_{13}c_{23}^2 \\&
    \hspace{6mm}- c_{23}^3 - c_{11}c_{13}c_{33} + c_{13}c_{22}c_{33} - c_{11}c_{23}c_{33} + c_{22}c_{23}c_{33}, \\&
    c_{12}^2 - c_{13}^2 - c_{11}c_{22} - 2c_{13}c_{23} - c_{23}^2 + c_{11}c_{33} + 2c_{12}c_{33} + c_{22}c_{33}, \\&
    c_{12}c_{13} + c_{13}c_{22} - c_{11}c_{23} - c_{12}c_{23}.
\end{split}
\end{equation}
These are conditions on $(c,k,r)$, whereas we would like to have conditions on $c$ only. This can be achieved by elimination theory, eliminating $k$ and $r$.
This corresponds geometrically to projecting the set defined by the vanishing of the polynomials~\eqref{eq:original-mixed-conditions} from the product space of $(c,k,r)$ to the space of $c$ alone.
Given the monomial order used above, the Gr\"obner basis of the elimination ideal consists of the last three generators listed above in~\eqref{eq:original-mixed-conditions}.

The vanishing set of the elimination ideal gives the Zariski closure of the image of the projection.
A minimal basis (but not a Gr\"obner basis) for the elimination ideal is given by the polynomials $P_1$ and $P_2$ as in~\eqref{eq:P1} and~\eqref{eq:P2} which can be seen from the Magma code in Listing \ref{lst:magma2}.
There may be points in the closure that do not lie in the image of the projection. These spurious points can be identified as follows.  
We take the projective closure with respect to the variables $(k,r)$. This corresponds to homogenizing the defining polynomials by introducing a new variable $z \in \mathbb{C}$, so that each polynomial becomes homogeneous in the variables $(k,r,z)$ \cite[Ch.~8, \S4, Prop.~2]{cox2025a}. These homogenized polynomials generate the ideal $I^h$, whose vanishing set defines the projective closure (with respect to the variables $(k,r)$). More precisely, the projective closure of an affine variety $W \subseteq \mathbb{A}^n$ is given by~\cite[Ch.~8, \S4, Def.~6]{cox2025a}
\begin{equation}
\overline{W} = \V\big(\I(W)^h\big) \subseteq \mathbb{P}^n
.
\end{equation}
Moreover, a Gröbner basis of $\I(W)$ with respect to a graded monomial order homogenizes to a Gröbner basis of $\I(W)^h$ \cite[Ch.~8, \S4, Thm.~4]{cox2025a}, which allows effective computation of $\overline{W}$. By \cite[Ch.~8, \S4, Prop.~7(i)]{cox2025a}, one has
\begin{equation}
\overline{W} \cap \{ z \neq 0 \} = W,
\end{equation}
hence the points with \( z = 0 \) are precisely the points of the projective closure not lying in the affine variety, i.e.\ the points at infinity.
The projection to the $c$-coordinates need not be closed, but by the Closure Theorem \cite[Ch.~3, \S2, Thm.~3]{cox2025a}, its Zariski closure is given by an elimination ideal. Thus, the additional (spurious) points arise as images of points at infinity.

The points at infinity are found by letting $z=0$, and this leads to the set of equations
\begin{equation}
\label{eq:homogenized-conditions-at-infty}
\begin{split}
    0 &= c_{11}(c_{13}c_{23} - c_{12}c_{33})r^2, \\
    0 &= (c_{13}^2 - c_{11}c_{33})r^2, \\
    0 &= (c_{13}c_{22} - c_{12}c_{23})r^2, \\
    0 &= (c_{23}^2 - c_{22}c_{33})r^2, \\
    0&=k_{00}=k_{20}=k_{11}=k_{02}=k_{10}=k_{01}, \\
    0 &= c_{13}^3 + c_{11}c_{13}c_{22} - c_{13}c_{22}^2 - c_{11}c_{12}c_{23} + c_{13}^2c_{23} + c_{12}c_{22}c_{23} - c_{13}c_{23}^2 - c_{23}^3 - c_{11}c_{13}c_{33}\\
    &\hspace{6mm}+ c_{13}c_{22}c_{33} - c_{11}c_{23}c_{33} + c_{22}c_{23}c_{33}, \\
    0 &= c_{12}^2 - c_{13}^2 - c_{11}c_{22} - 2c_{13}c_{23} - c_{23}^2 + c_{11}c_{33} + 2c_{12}c_{33} + c_{22}c_{33}, \\
    0 &= c_{12}c_{13} + c_{13}c_{22} - c_{11}c_{23} - c_{12}c_{23}
    .
\end{split}
\end{equation}
With $z=0$ and $k=0$, we must have $r\neq0$ in order to have a valid point in the projective space.
This can only happen when the following conditions hold:
\begin{equation}
\label{eq:spurious-points}
\begin{split}
    0 &= c_{11}(c_{13}c_{23} - c_{12}c_{33}), \\
    0 &= c_{13}^2 - c_{11}c_{33}, \\
    0 &= c_{13}c_{22} - c_{12}c_{23}, \\
    0 &= c_{23}^2 - c_{22}c_{33}
    .
\end{split}
\end{equation}
It is possible that the same stiffness tensor $c$ is in the image of both an affine point and a point at infinity under the projection.
Therefore the set of spurious stiffness tensors is contained in (but not necessarily equal to) the intersection of the variety defined by~\eqref{eq:spurious-points} and that defined by $P_1$ and $P_2$.

The intersection of two vanishing varieties of ideals is the vanishing variety of the radical of the sum of the two ideals.
In our case a minimal basis for this radical is given by (see the corresponding magma code in Listing~\eqref{lst:magma2})
\begin{equation}
\label{eq:spurious2}
\begin{split}
&
    c_{22}c_{33} - c_{23}^2, \\&
    c_{12}c_{23} - c_{13}c_{22}, \\&
    c_{11}c_{33} - c_{13}^2, \\&
    c_{11}c_{23} - c_{12}c_{13}, \\&
    c_{11}c_{22} - c_{12}^2 - 2c_{12}c_{33} + 2c_{13}c_{23}.
\end{split}
\end{equation}
These are conditions for a complex stiffness tensor.
The set of spurious real stiffness tensors is contained in the complex one.
Now that the algebraic geometric arguments are over, we take again the stiffness tensor to be real.

Positive definiteness means that the matrix $c$ of~\eqref{eq:Voigt-matrix} is positive definite, but~\eqref{eq:spurious2} makes many of the $2\times2$ minors vanish.
Therefore none of the spurious stiffness tensors are in positive definite locus.
\end{proof}

\begin{proof}[Detailed proof of Proposition~\ref{prop:two-circles}]
    We use the Magma computation in Listing~\ref{lst:isoslow}. The slowness polynomial of $c$ equals a product of two circle polynomials $(r_1x^2+r_1y^2-1)(r_2x^2+r_2y^2-1)$ for some $r_1,r_2$ if and only if all coefficients of their difference vanish as polynomials in $x,y$. Setting these coefficients to zero defines the ideal $I$ in the variables $c_{11},c_{12},c_{13},c_{22},c_{23},c_{33},r_1,r_2$. Eliminating $r_1$ and $r_2$ from $I$ via \texttt{EliminationIdeal} yields necessary and sufficient conditions on $c$ alone; by elimination theory, any $c$ satisfying these conditions lifts to a solution with some $r_1,r_2$. The resulting conditions are
\begin{equation}
    c_{13}+c_{23}=0,\quad c_{11}-c_{22}=0,\quad c_{12}c_{23}+c_{22}c_{23}=0,\quad c_{12}^2-c_{22}^2+2c_{12}c_{33}+2c_{22}c_{33}=0.
\end{equation}
The last two equations factor as
\begin{equation}
    c_{23}(c_{12}+c_{22})=0 \quad\text{and}\quad (c_{12}+c_{22})(c_{12}-c_{22}+2c_{33})=0,
\end{equation}
so the variety splits into two irreducible components according to whether $c_{12}+c_{22}=0$ or not, confirmed by \texttt{IrreducibleComponents}:
\textit{Component~1} ($c_{12}+c_{22}=0$): together with $c_{11}=c_{22}$ and $c_{13}=-c_{23}$, with $c_{11}, c_{13}$ and $c_{33}$ free, this gives exactly the tensors of the form~\eqref{eq:isoslow2}.
\textit{Component~2} ($c_{12}+c_{22}\neq 0$): the third equation then forces $c_{23}=0$, hence $c_{13}=0$ from the first equation, and the fourth gives $c_{12}=c_{22}-2c_{33}$, together with $c_{11}=c_{22}$. This gives exactly the tensors of the form~\eqref{eq:isoslow1}.
Since these two components are the complete irreducible decomposition of the elimination ideal, the slowness surface of $c$ is the union of two concentric circles if and only if $c$ is either of the form~\eqref{eq:isoslow1} or~\eqref{eq:isoslow2}.

Using that $c_{22}=c_{11}$ and $c_{23}=-c_{13}$, the Gröbner basis in all variables contains the following relations for the radii $r_1$ and $r_2$ of the two circles in the slowness surface:
\begin{equation}\label{eq:radii}
    r_1 + r_2 = c_{11}+c_{33}, \qquad r_2^2 - (c_{11}+c_{33})r_2 + (c_{11}c_{33}-c_{13}^2) = 0.
\end{equation}
Hence $r_2$ is the root of the second order polynomial above and from Vieta's formula that connects the polynomial coefficients to the sums and products of its roots it follows that $r_1$ is also a root of the same second order polynomial. Additionally, it follows that the product of $r_1$ and $r_2$ satisfies $r_1 r_2 = c_{11}c_{33}-c_{13}^2$. 

By this radius characterization, the slowness surface of a tensor depends on $c$ only through the quantities $c_{11}+c_{33}$ and $c_{11}c_{33}-c_{13}^2$, which determine the radii $r_1,r_2$ via~\eqref{eq:radii}. Hence a tensor of the form~\eqref{eq:isoslow1} and one of the form~\eqref{eq:isoslow2} give rise to the same slowness surface if and only if their values of $c_{11}+c_{33}$ and $c_{11}c_{33}-c_{13}^2$ agree. In particular, this condition is generally \emph{not} equivalent to $c_{11}=c_{33}$ and $c_{13}=0$: for instance, rotating the tensor~\eqref{eq:circlesemi} by an angle $\theta\in [0,2\pi)$ produces a one-parameter family of tensors of the form~\eqref{eq:isoslow2} (with $c_{11},c_{33},c_{13}$ varying nontrivially in $\theta$) that all share the same slowness surface as~\eqref{eq:circlesemi} itself, since $c_{11}+c_{33}$ and $c_{11}c_{33}-c_{13}^2$ are invariant under rotation.
\end{proof}

\begin{proof}[Detailed proof of Lemma \ref{lem:rotation-orbit}]
Assign to a stiffness tensor of the form~\eqref{eq:isoslow1} or~\eqref{eq:isoslow2} the
symmetric \emph{block}
\begin{equation}\label{eq:block}
    B \coloneqq \begin{pmatrix} c_{11} & c_{13}\\ c_{13} & c_{33}\end{pmatrix},
\end{equation}
formed from its $(1,1)$, $(1,3)$ and $(3,3)$ Voigt entries, the off-diagonal entry being the
$(1,3)$-entry $+c_{13}$ and not the $(2,3)$-entry $-c_{13}$. Both~$c$ and~$\tilde c$ have
block $B_0=\operatorname{diag}(c_{11},c_{33})$; write~$B'$ for the block of~$c'$.

By Proposition~\ref{prop:two-circles} two such tensors share a slowness surface if and only
if their blocks have equal trace and determinant, that is, if and only if
\begin{equation}
    \tr(B_0)=c_{11}+c_{33}=c_{11}'+c_{33}'=\tr(B')
\end{equation}
and
\begin{equation}
    \det(B_0)=c_{11}c_{33}=c_{11}'c_{33}'-(c_{13}')^2=\det(B').
\end{equation}
Since~$c$ and~$\tilde c$ have the same block, the first assertion follows.

We now consider $\rotC{\theta}\overline{c}$ for a general stiffness tensor $\overline{c}$ of
the form~\eqref{eq:isoslow2} with parameters $\overline{c}_{11}$, $\overline{c}_{13}$ and
$\overline{c}_{33}$. Specialising equation~\eqref{eq:rotCexplicit} in
Appendix~\ref{ap:rotrefl} to these parameters gives
\begin{align}
    \rotC{\theta}\overline{c} &=
    \begin{bmatrix}
        (\rotC{\theta}\overline{c})_{11} & -(\rotC{\theta}\overline{c})_{11} & (\rotC{\theta}\overline{c})_{13}\\
        -(\rotC{\theta}\overline{c})_{11}& (\rotC{\theta}\overline{c})_{11} & -(\rotC{\theta}\overline{c})_{13}\\
        (\rotC{\theta}\overline{c})_{13} & -(\rotC{\theta}\overline{c})_{13} & (\rotC{\theta}\overline{c})_{33}
    \end{bmatrix},\\
\intertext{with}
    (\rotC{\theta}\overline{c})_{11} &=
        \tfrac{1}{2}(\overline{c}_{11}+\overline{c}_{33})
        +\tfrac{1}{2}(\overline{c}_{11}-\overline{c}_{33})\cos(4\theta)
        -\overline{c}_{13}\sin(4\theta),\\
    (\rotC{\theta}\overline{c})_{33} &=
        \tfrac{1}{2}(\overline{c}_{11}+\overline{c}_{33})
        -\tfrac{1}{2}(\overline{c}_{11}-\overline{c}_{33})\cos(4\theta)
        +\overline{c}_{13}\sin(4\theta),\\
    (\rotC{\theta}\overline{c})_{13} &=
        \tfrac{1}{2}(\overline{c}_{11}-\overline{c}_{33})\sin(4\theta)
        +\overline{c}_{13}\cos(4\theta).
\end{align}
In particular $\rotC{\theta}\overline{c}$ is again of the form~\eqref{eq:isoslow2}, and
writing $\overline{B}$ for the block of $\overline{c}$ and $\overline{B}^{\theta}$ for the
block of $\rotC{\theta}\overline{c}$, the three identities above say precisely that
\begin{equation}\label{eq:Bconj2}
    \overline{B}^{\theta} = \rotA{2\theta}\,\overline{B}\,\rotA{2\theta}^{T};
\end{equation}
a rotation of the tensor by $\theta$ induces a rotation of its block by $2\theta$.

Hence, if $c'=\rotC{\theta}\tilde c$, then $B'=\rotA{2\theta}B_0\rotA{2\theta}^{T}$ by~\eqref{eq:Bconj2}, and since $\tr(QB_0Q^{T})=\tr(B_0)$ and
$\det(QB_0Q^{T})=\det(Q)^{2}\det(B_0)=\det(B_0)$ for every $Q\in O(2)$, it follows from
Proposition~\ref{prop:two-circles} that $c'$ and $c$ share the same slowness surface.

If we conversely assume that $c'$ is of the form~\eqref{eq:isoslow2} and shares a slowness
surface with $c$, then Proposition~\ref{prop:two-circles} gives
\begin{equation}
    \tr(B_0)=\tr(B') \quad \text{and} \quad \det(B_0)=\det(B').
\end{equation}
This implies that $B_0$ and $B'$ have the same characteristic polynomial
\begin{equation}
    \lambda^2-\tr(B_0)\lambda + \det(B_0),
\end{equation}
and thus the same eigenvalues $\lambda_1=c_{11}$ and $\lambda_2=c_{33}$. Since both matrices
are real symmetric, the spectral theorem yields some $Q\in O(2)$ with
\begin{equation}\label{eq:Bprime}
    B'=Q B_0 Q^T,
\end{equation}
after reordering the eigenvalues by a permutation matrix in $O(2)$ if necessary. Such a $Q$
is either a rotation, with $\det(Q)=1$, or a reflection, with $\det(Q)=-1$. In the latter
case write $Q=\reflA{\psi}=\rotA{2\psi}D$ with $D=\operatorname{diag}(1,-1)$; as $B_0$ is
diagonal, $DB_0D^{T}=B_0$ and therefore $QB_0Q^{T}=\rotA{2\psi}B_0\rotA{2\psi}^{T}$. Every
$O(2)$-conjugate of $B_0$ is thus already a rotation conjugate, so
$B'=\rotA{\phi}B_0\rotA{\phi}^{T}$ for some $\phi$. With $\theta\coloneqq\phi/2$, \eqref{eq:Bconj2} shows that $\rotC{\theta}\tilde c$ is a tensor of the form
\eqref{eq:isoslow2} with block $B'$. Such a tensor is determined by its block, so
$c'=\rotC{\theta}\tilde c$.
\end{proof}

\section{Magma code to support proofs in Section \ref{sec:AlgAdStiffness}}

\begin{lstlisting}[language=Magma, caption={Magma code for the generators for the Gröbner basis}, label=lst:magma1]
// stiffness tensor c
// coefficients k
// radius r

Q<c11,c12,c13,c22,c23,c33,r,k20,k11,k02,k10,k01,k00> :=
    PolynomialRing(Rationals(), [1,1,1,1,1,1,100,100,100,100,100,100,100]);

R<x,y> := PolynomialRing(Q, 2);

slownesspoly := (c11*c33 - c13^2)*x^4 +
                (2*c11*c23 - 2*c12*c13)*x^3*y +
                (c11*c22 - c12^2 - 2*c12*c33 + 2*c13*c23)*x^2*y^2 +
                (-c11 - c33)*x^2 +
                (-2*c12*c23 + 2*c13*c22)*x*y^3 +
                (-2*c13 - 2*c23)*x*y +
                (c22*c33 - c23^2)*y^4 +
                (-c22 - c33)*y^2 + 1;

circlepoly := x^2 + y^2 - r^2;
otherpoly := k20*x^2 + k11*x*y + k02*y^2 + k10*x + k01*y + k00;

coeffs := Coefficients(slownesspoly - circlepoly*otherpoly);
I := ideal<Q | coeffs>;

print "Groebner basis in all variables:";
GroebnerBasis(I);

print "Conditions on the stiffness tensor only:";
MinimalBasis(EliminationIdeal(I, {c11,c12,c13,c22,c23,c33}));

// BASIS OUTPUT:
[
    r^2*k00 + 1,
    c11*c13*c23*r^2 - c11*c12*c33*r^2 + c12*k00 + c11*c12 - c13^2 - c13*c23 +
        c12*c33,
    c13^2*r^2 - c11*c33*r^2 + k00 + c11 + c33,
    c13*c22*r^2 - c12*c23*r^2 - c13 - c23,
    c23^2*r^2 - c22*c33*r^2 + k00 + c22 + c33,
    k00^2 + c22*k00 + c33*k00 - c23^2 + c22*c33,
    c12*c23*k00 + c22*c23*k00 - c13*c23^2 - c23^3 + c12*c23*c33 + c22*c23*c33,
    c11*k00 - c22*k00 - c13^2 + c23^2 + c11*c33 - c22*c33,
    c13*k00 + c23*k00 + c13*c22 - c12*c23,
    k20 + c13^2 - c11*c33,
    k11 - 2*c13*c22 + 2*c12*c23,
    k02 + c23^2 - c22*c33,
    k10,
    k01,
    c13^3 + c11*c13*c22 - c13*c22^2 - c11*c12*c23 + c13^2*c23 + c12*c22*c23 -
        c13*c23^2 - c23^3 - c11*c13*c33 + c13*c22*c33 - c11*c23*c33 +
        c22*c23*c33,
    c12^2 - c13^2 - c11*c22 - 2*c13*c23 - c23^2 + c11*c33 + 2*c12*c33 + c22*c33,
    c12*c13 + c13*c22 - c11*c23 - c12*c23
]
\end{lstlisting}

The vanishing set:
\begin{lstlisting}[language=Magma, caption={Magma code to obtain conditions on the stiffness tensor c.}, label=lst:magma2]
        Q<c11,c12,c13,c22,c23,c33> := PolynomialRing(Rationals(),6);

// conditions for spurious points
generatorsSpurious :=
[
    c11*(c13*c23 - c12*c33),
    c13^2 - c11*c33,
    c13*c22 - c12*c23,
    c23^2 - c22*c33
];

// generators of elimination ideal
generatorsElimination :=
[
    c13^3 + c11*c13*c22 - c13*c22^2 - c11*c12*c23 + c13^2*c23 + c12*c22*c23 -
        c13*c23^2 - c23^3 - c11*c13*c33 + c13*c22*c33 - c11*c23*c33 +
        c22*c23*c33,
    c12^2 - c13^2 - c11*c22 - 2*c13*c23 - c23^2 + c11*c33 + 2*c12*c33 + c22*c33,
    c12*c13 + c13*c22 - c11*c23 - c12*c23
];

// determinant of c as a matrix (not used)
det := c11*c22*c33-c12^2*c33-c11*c23^2+2*c12*c13*c23-c13^2*c22;

// ideal corresponding to spurious points
Is := Radical(ideal<Q | generatorsSpurious>);
bIs := MinimalBasis(Is);

// ideal corresponding to elimination variety
Ie := Radical(ideal<Q | generatorsElimination>);
bIe := MinimalBasis(Ie);

// intersection of varieties is given by sum of ideals:
Ii := Radical(Is + Ie);
bIi := MinimalBasis(Ii);

bIs;
bIe;
bIi;

/// OUTPUT:

[
    c22*c33 - c23^2,
    c12*c23 - c13*c22,
    c11*c33 - c13^2,
    c11*c23 - c12*c13
]
[
    -c11*c23 + c12*c13 - c12*c23 + c13*c22,
    -c11*c22 + c11*c33 + c12^2 + 2*c12*c33 - c13^2 - 2*c13*c23 + c22*c33 - c23^2
]
[
    c22*c33 - c23^2,
    c12*c23 - c13*c22,
    c11*c33 - c13^2,
    c11*c23 - c12*c13,
    c11*c22 - c12^2 - 2*c12*c33 + 2*c13*c23
]
\end{lstlisting}

\begin{lstlisting}[language=Magma, caption={Magma code to characterize all stiffness tensors whose slowness surface is the union of two circles in Proposition \ref{prop:two-circles}.}, label=lst:isoslow]
Q<c11,c12,c13,c22,c23,c33,r1,r2> :=
    PolynomialRing(Rationals(), [1,1,1,1,1,1,100,100]);

R<x,y> := PolynomialRing(Q, 2);

slownesspoly := (c11*c33 - c13^2)*x^4 +
                (2*c11*c23 - 2*c12*c13)*x^3*y +
                (c11*c22 - c12^2 - 2*c12*c33 + 2*c13*c23)*x^2*y^2 +
                (-c11 - c33)*x^2 +
                (-2*c12*c23 + 2*c13*c22)*x*y^3 +
                (-2*c13 - 2*c23)*x*y +
                (c22*c33 - c23^2)*y^4 +
                (-c22 - c33)*y^2 + 1;

circlepoly1 := r1*x^2 + r1*y^2 - 1;
circlepoly2 := r2*x^2 + r2*y^2 - 1;

coeffs := Coefficients(slownesspoly - circlepoly1*circlepoly2);
I := ideal<Q | coeffs>;

print "Groebner basis in all variables:";
GroebnerBasis(I);

print "Conditions on the stiffness tensor only:";
mb := MinimalBasis(EliminationIdeal(I, {c11,c12,c13,c22,c23,c33}));
mb;

P := Proj(Q);
X := Scheme(P, mb);
IrreducibleComponents(X);

/// OUTPUT:
Groebner basis in all variables:
[
    r2^2 - c22*r2 - c33*r2 - c23^2 + c22*c33,
    r1 + r2 - c22 - c33,
    c12^2 - c22^2 + 2*c12*c33 + 2*c22*c33,
    c12*c23 + c22*c23,
    c11 - c22,
    c13 + c23
]
Conditions on the stiffness tensor only:
[
    c13 + c23,
    c11 - c22,
    c12*c23 + c22*c23,
    c12^2 - c22^2 + 2*c12*c33 + 2*c22*c33
]
[
    Scheme over Rational Field defined by
    c11 - c22,
    c12 + c22,
    c13 + c23,
    Scheme over Rational Field defined by
    c11 - c22,
    c12 - c22 + 2*c33,
    c13,
    c23
]
\end{lstlisting}

The singular locus and checking whether isotropic stiffness tensors are contained in the vanishing set of $P_1$ and $P_2$ defined in~\eqref{eq:P1} and~\eqref{eq:P2}:
\begin{lstlisting}[language=Magma, caption={Magma code to obtain the singular locus of the vanishing set of $P_1$ and $P_2$ and checking whether isotropic stiffness tensors are included in the vanishing set.}, label=lst:magma3]

Q<c11,c12,c13,c22,c23,c33> := PolynomialRing(Rationals(),6);

P := [
c12*c13 + c13*c22 - c11*c23 - c12*c23,
c12^2 - c13^2 - c11*c22 - 2*c13*c23 - c23^2 + c11*c33 + 2*c12*c33 + c22*c33
];

//Singular points are where rank of DP drops, so they are given by:
print("Singular points of the variety P1=P2=0 are exactly where these vanish:");
MinimalBasis(Radical(ideal<Q|Minors(JacobianMatrix(P),2) cat P>));

isotropic := [c13,c23,c22-c11,2*c33-c11+c12];

print("Isotropic stiffness tensors are contained in the variety P1=P2=0:");
ideal<Q|P> subset ideal<Q|isotropic>;

/// OUTPUT:

Singular points of the variety P1=P2=0 are exactly where these vanish:
[
    c13 + c23,
    c12 + c22,
    c11 - c22
]
Isotropic stiffness tensors are contained in the variety P1=P2=0:
true
\end{lstlisting}

\section{Magma code to support the proof of Lemma \ref{lem:smoothMap}}
\begin{lstlisting}[language=Magma, caption={Magma code to finding the graph condition}, label=lst:graph]
Q<s,p11,p12,p22,c11,c12,c13,c22,c23,c33> := PolynomialRing(Rationals(),10);
P<x,y> := PolynomialRing(Q,2);

slownesspoly := (c11*c33 - c13^2)*x^4 + (2*c11*c23 - 2*c12*c13)*x^3*y + (c11*c22 - c12^2 - 2*c12*c33 + 2*c13*c23)*x^2*y^2 + (-c11 - c33)*x^2 + (-2*c12*c23 + 2*c13*c22)*x*y^3 + (-2*c13 - 2*c23)*x*y + (c22*c33 - c23^2)*y^4 + (-c22 - c33)*y^2 + 1;

circlepoly := s*x^2 + s*y^2 - 1;
otherpoly := p11*x^2 + 2*p12*x*y + p22*y^2 - 1;

conditions := Coefficients(slownesspoly-circlepoly*otherpoly);

FlippedIsotropic := [
    s - c22,
    p11 - c33,
    p12,
    p22 - c33,
    c11 - c22,
    c12 - c22 + 2*c33,
    c13,
    c23
];

MinimalBasis(Saturation(ideal<Q|conditions>,ideal<Q|FlippedIsotropic>));

/// OUTPUT:
[
    p12 - c13 - c23,
    p11 - p22 - c11 + c22,
    s + p22 - c22 - c33,
    c11*c23 - c12*c13 + c12*c23 - c13*c22,
    c11*c22 - c11*c33 - c12^2 - 2*c12*c33 + c13^2 + 2*c13*c23 - c22*c33 + c23^2,
    p22*c13 + p22*c23 - c12*c23 - c13*c33 - c22*c23 - c23*c33,
    p22*c12 + p22*c22 - c12*c22 - c13*c23 - c22^2 - c23^2,
    p22*c11 - p22*c22 - c11*c22 - c13^2 + c22^2 + c23^2,
    p22^2 - p22*c22 - p22*c33 + c22*c33 - c23^2
]
\end{lstlisting}

\begin{lstlisting}[language=Magma, caption={Magma code to find the rank of $Df$}, label=lst:rank]
    Q<s,p11,p12,p22,c11,c12,c13,c22,c23,c33> := PolynomialRing(Rationals(),10);

f := [
    p12 - c13 - c23,
    p11 - p22 - c11 + c22,
    s + p22 - c22 - c33,
    c11*c23 - c12*c13 + c12*c23 - c13*c22,
    c11*c22 - c11*c33 - c12^2 - 2*c12*c33 + c13^2 + 2*c13*c23 - c22*c33 + c23^2,
    p22*c13 + p22*c23 - c12*c23 - c13*c33 - c22*c23 - c23*c33,
    p22*c12 + p22*c22 - c12*c22 - c13*c23 - c22^2 - c23^2,
    p22*c11 - p22*c22 - c11*c22 - c13^2 + c22^2 + c23^2,
    p22^2 - p22*c22 - p22*c33 + c22*c33 - c23^2
];

print "f vanishes implies that rank(Df)<7: ", ideal<Q|Minors(JacobianMatrix(f),7)> subset ideal<Q|f>;
print "Conditions for rank(Df)<6 on V(f): ", MinimalBasis(Radical(ideal<Q|Minors(JacobianMatrix(f),6) cat f>));

/// OUTPUT:
f vanishes implies that rank(Df)<7:  true
Conditions for rank(Df)<6 on V(f):  [
    c23,
    c22 - c33,
    c13,
    c12 + c33,
    c11 - c33,
    p22 - c33,
    p12,
    p11 - c33,
    s - c33
]
\end{lstlisting}

\begin{lstlisting}[language=Magma, caption={Magma code to find the ranks of $\partial_\pairsofmetrics f$ and $B(p)$ to confirm the statements in~\eqref{eq:ift-pi} and~\eqref{eq:ift-rho}}, label=lst:rankpirho]
Q<s,p11,p12,p22,c11,c12,c13,c22,c23,c33> := PolynomialRing(Rationals(),10);

f := [
    p12 - c13 - c23,
    p11 - p22 - c11 + c22,
    s + p22 - c22 - c33,
    c11*c23 - c12*c13 + c12*c23 - c13*c22,
    c11*c22 - c11*c33 - c12^2 - 2*c12*c33 + c13^2 + 2*c13*c23 - c22*c33 + c23^2,
    p22*c13 + p22*c23 - c12*c23 - c13*c33 - c22*c23 - c23*c33,
    p22*c12 + p22*c22 - c12*c22 - c13*c23 - c22^2 - c23^2,
    p22*c11 - p22*c22 - c11*c22 - c13^2 + c22^2 + c23^2,
    p22^2 - p22*c22 - p22*c33 + c22*c33 - c23^2
];

P := [
    -c11*c23 + c12*c13 - c12*c23 + c13*c22,
    -c11*c22 + c11*c33 + c12^2 + 2*c12*c33 - c13^2 - 2*c13*c23 + c22*c33 - c23^2
];

c := SymmetricMatrix(Q,[c11,c12,c22,c13,c23,c33]);

detE := (s-p11)*(s-p22)-p12^2; //det(sI-p)
detF := Determinant(c); //det(c)

Df := JacobianMatrix(f);
DP := ColumnSubmatrix(JacobianMatrix(P),5,6);

DEf := ColumnSubmatrix(Df,1,4);
DFf := ColumnSubmatrix(Df,5,6);

B := VerticalJoin(DP,DFf);

failure1 := Minors(DEf,4);
failure2 := Minors(B,6);

MinimalBasis(Radical(ideal<Q|failure1 cat f>));
MinimalBasis(Radical(ideal<Q|failure2 cat f>));

//saturation study:
failure_pi := MinimalBasis(Radical(ideal<Q | Minors(DEf,4) cat f>));

I1_pi := ideal<Q|failure_pi>;
I2_pi := Saturation(I1_pi, ideal<Q|detE>);
I1_pi; I2_pi;

failure := MinimalBasis(Radical(ideal<Q|Minors(B,6) cat f>));

I1 := ideal<Q|failure>;
I2 := Saturation(I1,ideal<Q|detE>);
I3 := Saturation(I2,ideal<Q|detF>);

I1;I2;I3;

/// OUTPUT:
[
    c13 + c23,
    c12 + c22,
    c11 - c22,
    p22 - 1/2*c22 - 1/2*c33,
    p12,
    p11 - 1/2*c22 - 1/2*c33,
    s - 1/2*c22 - 1/2*c33,
    c22^2 - 2*c22*c33 + 4*c23^2 + c33^2
]
[
    p12 - c13 - c23,
    p11 - p22 - c11 + c22,
    s + p22 - c22 - c33,
    c12^2 + c12*c22 + c12*c33 - 2*c13*c23 + c22*c33 - 2*c23^2,
    c11*c23 + c13*c22 - c13*c33 - c23*c33,
    c11*c23 - 1/2*c12*c13 + 1/2*c12*c23 - 1/2*c13*c33 - 1/2*c23*c33,
    c11*c22 - c11*c33 + c12*c22 - c12*c33 + c13^2 - c23^2,
    c11*c12 + 2*c11*c22 - c11*c33 + c12*c22 - 2*c12*c33 - c22*c33,
    p22*c13 + p22*c23 - c12*c23 - c13*c33 - c22*c23 - c23*c33,
    p22*c12 + p22*c22 - c12*c22 - c13*c23 - c22^2 - c23^2,
    p22*c11 - p22*c22 - c11*c33 + c12*c22 - c12*c33 + c22^2,
    p22^2 - p22*c22 - p22*c33 + c22*c33 - c23^2
]
Ideal of Polynomial ring of rank 10 over Rational Field
Order: Lexicographical
Variables: s, p11, p12, p22, c11, c12, c13, c22, c23, c33
Homogeneous, Dimension >0
Basis:
[
    c13 + c23,
    c12 + c22,
    c11 - c22,
    p22 - 1/2*c22 - 1/2*c33,
    p12,
    p11 - 1/2*c22 - 1/2*c33,
    s - 1/2*c22 - 1/2*c33,
    c22^2 - 2*c22*c33 + 4*c23^2 + c33^2
]
Ideal of Polynomial ring of rank 10 over Rational Field
Order: Lexicographical
Variables: s, p11, p12, p22, c11, c12, c13, c22, c23, c33
Homogeneous
Basis:
[
    1
]
Ideal of Polynomial ring of rank 10 over Rational Field
Order: Lexicographical
Variables: s, p11, p12, p22, c11, c12, c13, c22, c23, c33
Homogeneous, Dimension >0
Basis:
[
    p12 - c13 - c23,
    p11 - p22 - c11 + c22,
    s + p22 - c22 - c33,
    c12^2 + c12*c22 + c12*c33 - 2*c13*c23 + c22*c33 - 2*c23^2,
    c11*c23 + c13*c22 - c13*c33 - c23*c33,
    c11*c23 - 1/2*c12*c13 + 1/2*c12*c23 - 1/2*c13*c33 - 1/2*c23*c33,
    c11*c22 - c11*c33 + c12*c22 - c12*c33 + c13^2 - c23^2,
    c11*c12 + 2*c11*c22 - c11*c33 + c12*c22 - 2*c12*c33 - c22*c33,
    p22*c13 + p22*c23 - c12*c23 - c13*c33 - c22*c23 - c23*c33,
    p22*c12 + p22*c22 - c12*c22 - c13*c23 - c22^2 - c23^2,
    p22*c11 - p22*c22 - c11*c33 + c12*c22 - c12*c33 + c22^2,
    p22^2 - p22*c22 - p22*c33 + c22*c33 - c23^2
]
Ideal of Polynomial ring of rank 10 over Rational Field
Order: Lexicographical
Variables: s, p11, p12, p22, c11, c12, c13, c22, c23, c33
Homogeneous, Dimension >0
Basis:
[
    p22^2 - p22*c22 - p22*c33 + c22*c33 - c23^2,
    s + p22 - c22 - c33,
    p11 - p22,
    p12,
    c11 - c22,
    c12 + c22,
    c13 + c23
]
Ideal of Polynomial ring of rank 10 over Rational Field
Order: Lexicographical
Variables: s, p11, p12, p22, c11, c12, c13, c22, c23, c33
Homogeneous
Basis:
[
    1
]
\end{lstlisting}

\bibliographystyle{abbrv}
\bibliography{biblio}

@PREAMBLE{
"\def\cprime{$'$} "
}

@InCollection{uhlmann2004cauchy,
  Title                    = {The Cauchy data and the scattering relation},
  Author                   = {Uhlmann, Gunther},
  Booktitle                = {Geometric methods in inverse problems and PDE control},
  Publisher                = {Springer},
  Year                     = {2004},
  Pages                    = {263--287}
}

@misc{gathmann2021alggeom,
  author       = {Andreas Gathmann},
  title        = {Algebraic Geometry: Lecture Notes},
  year         = {2021},
  howpublished = {\url{https://agag-gathmann.math.rptu.de/class/alggeom-2021/alggeom-2021.pdf}},
  note         = {accessed 2026-06-07}
}

@misc{gathmann2002alggeomchapter4,
  author       = {Andreas Gathmann},
  title        = {Algebraic Geometry: Chapter 4},
  year         = {2002},
  howpublished = {\url{https://agag-gathmann.math.rptu.de/class/alggeom-2002/alggeom-2002-c4.pdf}},
  note         = {accessed 2026-06-07}
}

@book{cox2025a,
  author    = {David A. Cox and John Little and Donal O'Shea},
  title     = {Ideals, Varieties, and Algorithms: An Introduction to Computational Algebraic Geometry and Commutative Algebra},
  series    = {Undergraduate Texts in Mathematics},
  edition   = {5},
  publisher = {Springer},
  address   = {Cham},
  year      = {2025},
  doi       = {10.1007/978-3-031-91841-4},
  isbn      = {978-3-031-91840-7}
}

@article{belishev1987approach,
  Title                    = {An approach to multidimensional inverse problems for the wave equation},
  Author                   = {Belishev, Michael I},
  journal                = {Doklady Akademii Nauk SSSR},
  Year                     = {1987},
  Number                   = {3},
  Pages                    = {524--527},
  Volume                   = {297}
}

@Article{belishev1992reconstruction,
  Title                    = {To the reconstruction of a {R}iemannian manifold via its spectral data ({B}c--{M}ethod)},
  Author                   = {Belishev, Michael I and Kurylev, Yaroslav V},
  Journal                  = {Communications in partial differential equations},
  Year                     = {1992},
  Number                   = {5-6},
  Pages                    = {767--804},
  Volume                   = {17},
  Publisher                = {Taylor \& Francis}
}

@article{MazzucatoRachele2008,
  author  = {Mazzucato, Anna L. and Rachele, Lizabeth V.},
  title   = {On Transversely Isotropic Elastic Media with Ellipsoidal Slowness Surfaces},
  journal = {Mathematics and Mechanics of Solids},
  volume  = {13},
  number  = {7},
  pages   = {611--638},
  year    = {2008},
  doi     = {10.1177/1081286507078498}
}

@book{Gr-Sj,
  title={Microlocal analysis for differential operators: an introduction},
  author={Grigis, Alain and Sj{\"o}strand, Johannes},
  year={1994},
  publisher={Cambridge University Press},
  address={Cambridge},
  isbn={0-521-44975-8}
}

@article{MazzucatoRachele2007,
  author  = {Mazzucato, Anna L. and Rachele, Lizabeth V.},
  title   = {On Uniqueness in the Inverse Problem for Transversely Isotropic Elastic Media with a Disjoint Wave Mode},
  journal = {Wave Motion},
  volume  = {44},
  number  = {7--8},
  pages   = {605--625},
  year    = {2007},
  doi     = {10.1016/j.wavemoti.2007.03.004}
}

@misc{de2023reconstruction,
  title={Reconstruction of generic anisotropic stiffness tensors from partial data around one polarization},
  author={de Hoop, Maarten V and Ilmavirta, Joonas and Lassas, Matti and V{\'a}rilly-Alvarado, Anthony},
  journal={arXiv:2307.03312},
  year={2023}
}

@article{mazzucato2006partial,
  title={Partial uniqueness and obstruction to uniqueness in inverse problems for anisotropic elastic media},
  author={Mazzucato, Anna L and Rachele, Lizabeth V},
  journal={Journal of Elasticity},
  volume={83},
  number={3},
  pages={205--245},
  year={2006},
  publisher={Springer}
}

@article{ChadwickNorris1990,
  author  = {Chadwick, P. and Norris, A. N.},
  title   = {Conditions under Which the Slowness Surface of an Anisotropic Elastic Material Is the Union of Aligned Ellipsoids},
  journal = {The Quarterly Journal of Mechanics and Applied Mathematics},
  volume  = {43},
  number  = {4},
  pages   = {589--603},
  year    = {1990},
  doi     = {10.1093/qjmam/43.4.589}
}

@article{ilmavirta2026rigidity,
  title={Rigidity of homogeneous {L}am{é} systems},
  author={Ilmavirta, Joonas and Saksala, Teemu and Yan, Lili},
  journal={Journal of Inverse and Ill-Posed Problems, in press, arXiv:2602.08860},
  year={2026}
}

@article{zhai2025determination,
  title={Determination of the density in the linear elastic wave equation},
  author={Zhai, Jian},
  journal={Journal of Differential Equations},
  volume={459},
  pages={114064},
  year={2026},
  publisher={Elsevier}
}

@article{stefanov2017elastic,
  title={Local recovery of the compressional and shear speeds from the hyperbolic {DN} map},
  author={Stefanov, Plamen and Uhlmann, Gunther and Vasy, Andras},
  journal={Inverse Problems},
  volume={34},
  number={1},
  pages={014003},
  year={2017},
  publisher={IOP Publishing}
}

@Article{pestov2005two,
  Title                    = {Two dimensional compact simple {R}iemannian manifolds are boundary distance rigid},
  Author                   = {Pestov, Leonid and Uhlmann, Gunther},
  Journal                  = {Annals of {M}athematics},
  Year                     = {2005},
  Pages                    = {1093--1110},

  Publisher                = {JSTOR}
}

@Article{croke1991rigidity,
  Title                    = {Rigidity and the distance between boundary points},
  Author                   = {Croke, Christopher B },
  Journal                  = {Journal of Differential Geometry},
  Year                     = {1991},
  Number                   = {2},
  Pages                    = {445--464},
  Volume                   = {33},
  Publisher                = {Lehigh University}
}

@article{burago2010boundary,
  title={Boundary rigidity and filling volume minimality of metrics close to a flat one},
  author={Burago, Dmitri and Ivanov, Sergei},
  journal={Annals of mathematics},
  pages={1183--1211},
  year={2010},
  publisher={JSTOR}
}

@Article{michel1981rigidite,
  Title                    = {Sur la rigidit{\'e} impos{\'e}e par la longueur des g{\'e}od{\'e}siques},
  Author                   = {Michel, Ren{\'e}},
  Journal                  = {Inventiones mathematicae},
  Year                     = {1981},
  Number                   = {1},
  Pages                    = {71--83},
  Volume                   = {65},
  Publisher                = {Springer}
}

@article{de2019inverse,
  title={Determination of a compact {F}insler manifold from its boundary distance map and an inverse problem in elasticity},
  author={de Hoop, Maarten V and Ilmavirta, Joonas and Lassas, Matti and Saksala, Teemu},
  journal={Communications in Analysis and Geometry},
  year={2023},
   volume={31},
  number={7},
  pages={1693--1747}
}

@article{SU3,
author="Stefanov, Plamen and Uhlmann, Gunther",
title="Boundary ridigity and stability for generic simple metrics",
journal="Journal of the {A}merican {M}athematical {S}ociety",
volume=18,
year=2005,
pages="975-1003"
}

@book{Paternain2023,
  title = {Geometric Inverse Problems: With Emphasis on Two Dimensions},
  ISBN = {9781316510872},
  url = {http://dx.doi.org/10.1017/9781009039901},
  DOI = {10.1017/9781009039901},
  publisher = {Cambridge University Press},
  author = {Paternain,  Gabriel P. and Salo,  Mikko and Uhlmann,  Gunther},
  year = {2023},
  month = jan 
}

@misc{IKK25,
  title={Horizontal and Vertical Regularity of Elastic Wave Geometry},
  author={Ilmavirta, Joonas and Kirkkopelto, Pieti and Kykk{\"a}nen, Antti},
  note={arXiv:2511.16466},
  year={2025}
}

@article{bhuiyan2016anisotropic,
  title={Anisotropic parameters of dry and saturated sand under stress},
  author={Bhuiyan, Mohammad H and Holt, Rune M},
  journal={Geophysics},
  volume={81},
  number={5},
  pages={C229--C241},
  year={2016},
  publisher={Society of Exploration Geophysicists and American Association of Petroleum~…}
}

@article{dencker1982propagation,
  title={On the propagation of polarization sets for systems of real principal type},
  author={Dencker, Nils},
  journal={Journal of Functional Analysis},
  volume={46},
  number={3},
  pages={351--372},
  year={1982},
  publisher={Elsevier}
}

@article{de2019unique,
  title={Unique recovery of piecewise analytic density and stiffness tensor from the elastic-wave {D}irichlet-to-{N}eumann map},
  author={de Hoop, Maarten V and Nakamura, Gen and Zhai, Jian},
  journal={SIAM Journal on Applied Mathematics},
  volume={79},
  number={6},
  pages={2359--2384},
  year={2019},
  publisher={SIAM}
}

@article{rachele2000boundary,
  title={Boundary determination for an inverse problem in elastodynamics},
  author={Rachele, Lizabeth},
  journal={Communications in Partial Differential Equations},
  year={2000},
  volume    = {25},
  number    = {11-12},
  pages     = {1951--1996},
  doi       = {10.1080/03605300008821575},
  url       = {https://doi.org/10.1080/03605300008821575}
}

@article{rachele2000inverse,
  title={An inverse problem in elastodynamics: uniqueness of the wave speeds in the interior},
  author={Rachele, Lizabeth},
  journal={Journal of Differential Equations},
  volume={162},
  number={2},
  pages={300--325},
  year={2000},
  publisher={Elsevier}
}

@article{rachele2003uniqueness,
  title={Uniqueness of the density in an inverse problem for isotropic elastodynamics},
  author={Rachele, Lizabeth},
  journal={Transactions of the American Mathematical Society},
  volume={355},
  number={12},
  pages={4781--4806},
  year={2003}
}

@article{stefanov2017local,
  title={Local and global boundary rigidity and the geodesic {X}-ray transform in the normal gauge},
  author={Stefanov, Plamen and Uhlmann, Gunther and Vasy, Andr{\'a}s},
  journal={Annals of Mathematics},
  volume={194},
  number={1},
  pages={1--95},
  year={2021},
  publisher={JSTOR}
}

@article{bhattacharyya2018local,
  title={Local uniqueness of the density from partial boundary data for isotropic elastodynamics},
  author={Bhattacharyya, Sombuddha},
  journal={Inverse Problems},
  volume={34},
  number={12},
  pages={125001},
  year={2018},
  publisher={IOP Publishing}
}

@article{survey,
title = "Travel Time Tomography",
author = "Plamen Stefanov and Gunther Uhlmann and Andras Vasy and Hanming Zhou",
year = "2019",
month = jun,
day = "1",
doi = "10.1007/s10114-019-8338-0",
language = "English",
volume = "35",
pages = "1085--1114",
journal = "Acta Mathematica Sinica, English Series",
issn = "1439-8516",
publisher = "Springer Verlag",
number = "6",
}

@book{paternain2023book,
  title={Geometric inverse problems},
  author={Paternain, Gabriel P and Salo, Mikko and Uhlmann, Gunther},
  volume={204},
  year={2023},
  publisher={Cambridge University Press}
}

@article{belishev2002dynamical,
  title={The dynamical {L}am\'e system: regularity of solutions, boundary controllability and boundary data continuation},
  author={Belishev, MI and Lasiecka, I},
  journal={ESAIM: Control, Optimisation and Calculus of Variations},
  volume={8},
  pages={143--167},
  year={2002},
  publisher={EDP Sciences}
}

@article{ilmavirta2023gauge,
  title = {{Gauge freedoms in the anisotropic elastic Dirichlet-to-Neumann map}},
  volume = {19},
  ISSN = {1930-8345},
  url = {http://dx.doi.org/10.3934/ipi.2024055},
  DOI = {10.3934/ipi.2024055},
  number = {4},
  journal = {Inverse Problems and Imaging},
  publisher = {American Institute of Mathematical Sciences (AIMS)},
  author = {Ilmavirta,  Joonas and Schl\"{u}ter,  Hjørdis},
  year = {2025},
  pages = {816–828}
}

@misc{oksanen2024rigidity,
  title={Rigidity in the {L}orentzian {C}alder{\'o}n problem with formally determined data},
  author={Lauri Oksanen and {Rakesh} and Mikko Salo},
  note={arXiv:2409.18604},
  year={2024}
}

@article{ivanov2013local,
  title={Local monotonicity of {R}iemannian and {F}insler volume with respect to boundary distances},
  author={Ivanov, Sergei},
  journal={Geometriae Dedicata},
  volume={164},
  number={1},
  pages={83--96},
  year={2013},
  publisher={Springer}
}

@article{eller2002uniqueness,
  title={Uniqueness and stability in the {C}auchy problem for {M}axwell and elasticity systems},
  author={Eller, Matthias and Isakov, Victor and Nakamura, Gen and Tataru, Daniel},
  journal={Nonlinear partial differential equations and their applications},
  volume={31},
  pages={329--349},
  year={2002},
  publisher={Studies in Mathematics and Applications}
}

@article{uhlmann2016inverse,
  title={The inverse problem for the local geodesic ray transform},
  author={Uhlmann, Gunther and Vasy, Andr{\'a}s},
  journal={Inventiones mathematicae},
  volume={205},
  number={1},
  pages={83--120},
  year={2016},
  publisher={Springer}
}

@article{Alexakis_Feiz_Oksanen_23,
  title={Lorentzian {C}alder{\'o}n problem near the {M}inkowski geometry},
  author={Alexakis, Spyros and Feizmohammadi, Ali and Oksanen, Lauri},
  journal={Journal of the European Mathematical Society},
  year={2024}
}

@article{Alexakis_Feiz_Oksanen_22,
	author = {Alexakis, Spyros and Feizmohammadi, Ali and Oksanen, Lauri},
	journal = {Inventiones {M}athematicae},
	pages = {87-138},
	title = {Lorentzian {C}alder{\'o}n problem under curvature bounds},
	volume = {229},
	year = {2022}}

@Book{Katchalov2001,
  Title                    = {Inverse boundary spectral problems},
  Author                   = {Katchalov, Alexander and Kurylev, Yaroslav and Lassas, Matti},
  Publisher                = {Chapman \& Hall/CRC, Boca Raton, FL},
  Year                     = {2001},
  Series                   = {Monographs and Surveys in Pure and Applied Mathematics},
  Volume                   = {123},

  Doi                      = {10.1201/9781420036220},
  ISBN                     = {1-58488-005-8},
  Mrclass                  = {58J50 (35-02 35R30 58J60 93B30)},
  Mrnumber                 = {1889089 (2003e:58045)},
  Mrreviewer               = {Paul E. Sacks},
  Pages                    = {xx+290},
  Url                      = {http://dx.doi.org/10.1201/9781420036220}
}

@article{belishev2006dynamical,
  title={Dynamical inverse problem for a {L}am{\'e} type system.},
  author={Belishev, MI2270698},
  journal={Journal of Inverse \& Ill-Posed Problems},
  volume={14},
  number={8},
  year={2006}
}

@article{eskin2007inverse,
  title={Inverse hyperbolic problems with time-dependent coefficients},
  author={Eskin, Gregory},
  journal={Communications in Partial Differential Equations},
  volume={32},
  number={11},
  pages={1737--1758},
  year={2007},
  publisher={Taylor \& Francis}
}

@Article{Tataru,
  Title                    = {Unique continuation for solutions to {PDE}'s; between {H}\"ormander's theorem and {H}olmgren's theorem},
  Author                   = {Tataru, Daniel},
  Journal                  = {Communications in partial differential equations},
  Year                     = {1995},
  Number                   = {5-6},
  Pages                    = {855--884},
  Volume                   = {20},

  Publisher                = {Dekker}
}

@article{de2020recovery,
  title={Recovery of material parameters in transversely isotropic media},
  author={de Hoop, Maarten V and Uhlmann, Gunther and Vasy, Andr{\'a}s},
  journal={Archive for Rational Mechanics and Analysis},
  volume={235},
  number={1},
  pages={141--165},
  year={2020},
  publisher={Springer}
}

@article{zou2024partial,
  title={Partial Global Recovery in the Elastic Travel Time Tomography Problem for Transversely Isotropic Media},
  author={Zou, Yuzhou},
  journal={Annales de l'Institut Fourier},
  volume={74},
  number={5},
  pages={2077--2139},
  year={2024}
}

@misc{oksanen2026counterexamples,
  title={{Counterexamples to the Lorentzian Calder\'on problem}},
  author={Oksanen, Lauri and Sarkkinen, Miika},
  note={arXiv:2604.20320},
  year={2026}
}

@misc{lassas2026determination,
  title={Determination of an anisotropic perturbation in elastic inverse scattering},
  author={Lassas, Matti and Ma, Shiqi and Oksanen, Lauri and Salo, Mikko and Zhai, Jian},
  note={arXiv:2604.27390},
  year={2026}
}

@misc{2d-elastic-characterization,
  title={Characterization of uniqueness, stability, and reducibility for stiffness tensors in {2D}},
  author={Cocan, Antonio and de Hoop, Maarten V and Ilmavirta, Joonas and Lassas, Matti and V{\'a}rilly-Alvarado, Anthony},
  note={In preparation},
  year={2026}
}

@misc{ashwin-elasticity,
  title={Determination of Wave Speed from interior sources},
  author={Joonas Ilmavirta and Teemu Saksala and Ashwin Tarikere and Gunther Uhlmann},
  note={arXiv:2609.06225},
  year={2026}
}

@book{invitation-to-AG,
  author    = {Karen E. Smith and Lauri Kahanp{\"a}{\"a} and
               Pekka Kek{\"a}l{\"a}inen and William Traves},
  title     = {An Invitation to Algebraic Geometry},
  series    = {Universitext},
  publisher = {Springer},
  address   = {New York},
  year      = {2000},
  doi       = {10.1007/978-1-4757-4497-2},
  isbn      = {978-0-387-98980-8}
}

@article{magma,
    AUTHOR = {Bosma, Wieb and Cannon, John and Playoust, Catherine},
     TITLE = {The {M}agma algebra system. {I}. {T}he user language},
      NOTE = {Computational algebra and number theory (London, 1993)},
   JOURNAL = {J. Symbolic Comput.},
  FJOURNAL = {Journal of Symbolic Computation},
    VOLUME = {24},
      YEAR = {1997},
    NUMBER = {3-4},
     PAGES = {235--265},
      ISSN = {0747-7171,1095-855X},
   MRCLASS = {68Q40},
  MRNUMBER = {1484478},
       DOI = {10.1006/jsco.1996.0125},
       URL = {https://doi-org.ezproxy.rice.edu/10.1006/jsco.1996.0125},
}

\end{document}